\documentclass[11pt,a4paper]{amsart}
\usepackage{amssymb,amsthm,enumerate,amsmath,mathrsfs,comment}
\usepackage[noadjust]{cite}
\usepackage{hyperref}
\usepackage[capitalise]{cleveref}

\usepackage{tikz}
\usetikzlibrary{calc,patterns,through}

\theoremstyle{plain}
\newtheorem{thm}{Theorem}[section]
\newtheorem{lemma}[thm]{Lemma}
\newtheorem{cor}[thm]{Corollary}
\newtheorem{prop}[thm]{Proposition}

\newtheorem{claim}{}[thm] %Claim
\newenvironment{subproof}{\begin{proof}[Subproof.]}{\end{proof}}

\theoremstyle{definition}
\newtheorem*{definition}{Definition}

\Crefname{thm}{Theorem}{Theorems}
\crefmultiformat{claim}{#2#1#3}{ and~#2#1#3}{, #2#1#3}{ and~#2#1#3}

\DeclareMathOperator{\cl}{cl}
\DeclareMathOperator{\fcl}{fcl}
\DeclareMathOperator{\si}{si}
\DeclareMathOperator{\co}{co}
\newcommand{\del}{\backslash}
\newcommand{\ba}{\backslash}
\newcommand{\cocl}{\cl^*}

\newcommand{\dY}{$\Delta$\nobreakdash-$Y$}
\newcommand{\Yd}{$Y$\nobreakdash-$\Delta$}

\begin{document}
\title{Excluded minors are almost fragile}

\author[N.\ Brettell]{Nick Brettell}
\address{Department of Mathematics and Computer Science, Eindhoven University of Technology,
%P.O. Box 513, 5600 MB Eindhoven,
The Netherlands}
\email{nbrettell@gmail.com}
\author[B.\ Clark]{Ben Clark}
\address{Department of Mathematics, Louisiana State University, Baton Rouge, Louisiana, USA}
\email{bclark@lsu.edu}
\author[J.\ Oxley]{James Oxley}
\address{Department of Mathematics, Louisiana State University, Baton Rouge, Louisiana, USA}
\email{oxley@math.lsu.edu}
\author[C.\ Semple]{Charles Semple}
\address{Department of Mathematics and Statistics, University of Canterbury, New Zealand}
\email{charles.semple@canterbury.ac.nz}
\author[G.\ Whittle]{Geoff Whittle}
\address{School of Mathematics and Statistics, Victoria University of Wellington, New Zealand}
\email{geoff.whittle@vuw.ac.nz}

\thanks{The first, fourth, and fifth authors %Charles Semple and Geoff Whittle
were supported by the New Zealand Marsden Fund.}

\date{\today}

\begin{abstract}
  Let $M$ be an excluded minor for the class of $\mathbb{P}$-representable matroids for some partial field $\mathbb P$, and let $N$ be a $3$-connected strong $\mathbb{P}$-stabilizer that is non-binary.
  We prove that either $M$ is bounded relative to $N$, or, up to replacing $M$ by a \dY-equivalent excluded minor, we can choose a pair of elements $\{a,b\}$ such that either $M\del \{a,b\}$ is $N$-fragile, or $M^* \del \{a,b\}$ is $N^*$-fragile.
\end{abstract}

\maketitle

\section{Introduction}

One of the longstanding goals of matroid theory is to find excluded-minor characterisations
of classes of representable matroids.
%Indeed, to some extent, progress in matroid theory can be measured by success in problems of this type.
Results to date include Tutte's excluded-minor characterisation of binary and regular matroids \cite{tutte1958homotopy}; Bixby's and, independently, Seymour's excluded-minor characterisation of ternary matroids \cite{bixby1979reid,seymour1979matroid}; Geelen, Gerards and Kapoor's excluded-minor characterisation of GF$(4)$-representable matroids \cite{Geelen2000excluded}; and Hall, Mayhew and van Zwam's excluded-minor characterisation of the near-regular matroids, that is, the matroids representable over all fields with at least three elements \cite{hall2011excluded}. 
%Recently Geelen, Gerards and Whittle announced a proof
%of Rota's Conjecture \cite{geelen2014solving}.  However, their techniques 
%are extremal and give no insight into how one might find the exact list of
%excluded minors for such classes. Extending the range of known exact excluded-minor theorems
%for basic classes of matroids remains a problem of genuine interest and, indeed, a significant
%challenge that tests the state of the art of techniques in matroid theory.

%At this stage we need to note that regular matroids and many other naturally arising
%classes of representable matroids such as near-regular, 
%dyadic and $\sqrt[6]{1}$-matroids \cite{whittle1997matroids} 
%can be described as classes of matroids representable
%over an algebraic structure called a {\em partial field}. Of course, a field is an example
%of a partial field, and classes of 
%matroids representable over partial fields enjoy many of the properties that hold for matroids
%representable over fields. 

The immediate problem that looms large is that of finding the excluded minors for
the class of GF$(5)$-representable matroids. While this problem is beyond the range of
current techniques, a road map for an attack is outlined in \cite{pendavingh2010confinement}. In essence,
this road map reduces the problem to a finite sequence of problems of a type that we now describe.
First note that regular matroids and many other naturally arising classes of representable matroids such as near-regular, dyadic and $\sqrt[6]{1}$-matroids \cite{whittle1997matroids} can be described as classes of matroids representable over an algebraic structure called a {\em partial field}. 
We wish to find the excluded-minor characterisation for the class of $\mathbb P$-representable matroids for some
fixed partial field $\mathbb P$. We have a 3-connected matroid $N$ with the property that
every $\mathbb P$-representation of $N$ extends {\em uniquely} to a $\mathbb P$-representation
of any  3-connected $\mathbb P$-representable matroid having $N$ as a minor. Such a matroid 
$N$ is
called a {\em strong stabilizer} for the class of $\mathbb P$-representable matroids. 
With these ingredients,
the goal is to bound the size of an excluded minor for the class of $\mathbb P$-representable
matroids having 
the strong stabilizer $N$ as a minor.
This situation is a more general version of the one that arises in the proof of the excluded-minor characterisation of GF$(4)$-representable matroids~\cite{Geelen2000excluded}. There, the partial field is GF$(4)$
and the strong stabilizer is $U_{2,4}$.
(See also the introduction to \cite{bww1} for more detail on %the motivation behind
this strategy.)

Ideally, we would develop techniques that would reduce problems of the above type to routine
computation. But an annoying barrier arises. Let $N$ be a matroid. A matroid $M$ is
$N$-{\em fragile} if, for all elements $e$ of $M$, at most one of $M\backslash e$ or $M/e$
has an $N$-minor. It seems that, for a strong stabilizer $N$ for a partial field $\mathbb P$,
to bound the size of an excluded minor for $\mathbb P$-representable matroids
that contains $N$ as a minor,
we need to have some insight into the structure of $\mathbb P$-representable
$N$-fragile  matroids. The goal of this paper is to demonstrate
that this is,
in essence, the fundamental problem. We prove %, under a particular assumption,
that if $M$ is an excluded minor for the class of $\mathbb P$-representable matroids
having an $N$-minor, then
either the size or rank of $M$ is bounded relative to $N$, or, up to replacing $M$ by a \dY-equivalent excluded minor, $M$ (or its dual) has a pair of elements $\{a,b\}$ such that $M \ba a,b$ is an $N$-fragile (or $N^*$-fragile) matroid. More specifically, we prove the following:

\begin{thm}
\label{one}
Let $\mathbb{P}$ be a partial field, let $M$ be an excluded minor for the class of $\mathbb{P}$-representable matroids, and let $N$ be a non-binary strong stabilizer for the class of $\mathbb{P}$-representable matroids, where $M$ has an $N$-minor.
For some matroid $M_1$ that is \dY-equivalent to $M$, and some $(M',N')$ in $\{(M_1,N),(M_1^{*}, N^{*})\}$, the matroid $M'$ is an excluded minor having an $N'$-minor such that at least one of the following holds:
\begin{itemize}
  \item[(i)] $|E(M')|\leq |E(N')|+9$;
  \item[(ii)] $r(M')\leq r(N')+7$; or 
  \item[(iii)] there is a pair $\{a,b\} \subseteq E(M)$ such that $M'\del a,b$ is a $3$-connected $N'$-fragile matroid with an $N'$-minor.
\end{itemize}
\end{thm}

\noindent We defer the definition of the \dY-equivalence to the next section.

\Cref{one} tells us that an
%All going well, this will mean that our
excluded minor for $\mathbb P$-representable matroids 
will either have bounded size or will be very close to an $N$-fragile matroid.
Current techniques for bounding the size of an excluded minor in the latter case rely on 
obtaining explicit information about the structure of $N$-fragile matroids and this
needs to be done on a case-by-case basis. Even for quite simple matroids this can be
a difficult problem. Here is an example. Recall the non-Fano matroid $F_7^-$. 
The barrier to finding the excluded minors for the
class of dyadic matroids is that we do not understand the structure of 
dyadic $F_7^-$-fragile matroids and such an understanding seems some way off.

On the other hand, $U_{2,5}$ and $U_{3,5}$ are strong stabilizers for representability over two interesting partial fields and we do know the structure of $U_{2,5}$- and $U_{3,5}$-fragile matroids within these classes \cite{clark2015structure}.
The first is the partial field~$\mathbb H_5$ which was introduced by Pendavingh and van Zwam \cite{pendavingh2010confinement}. The class of matroids 
representable over this field is the class obtained by taking the 3-connected matroids that
have exactly six inequivalent representations over GF$(5)$ and closing the class
under minors. This class forms the bottom layer of Pendavingh and van Zwam's hierarchy of 
GF$(5)$-representable
matroids. Finding excluded minors for this class would be a key first step towards finding
the excluded minors for matroids representable over  GF$(5)$. 

The other partial field is the \textit{$2$-regular} or \textit{$2$-uniform} partial field, denoted $\mathbb U_2$. This is a member of a family of partial
fields. The matroids representable over $\mathbb U_0$ and $\mathbb U_1$ are the
regular and near-regular matroids respectively. Regular matroids are the matroids representable
over all fields, and near-regular matroids are the matroids representable over all fields
with at least three elements. Let $\mathcal M_4$ denote the matroids representable over
all fields of size at least four. It would certainly be interesting to have a characterisation
of the class $\mathcal M_4$.
The class of $\mathbb U_2$-representable matroids is contained in $\mathcal M_4$,
and it is known \cite{Semple} that this class is a 
proper subclass of $\mathcal M_4$.
Nonetheless, knowing the excluded minors for $\mathbb U_2$ would be a key step towards
characterising the class $\mathcal M_4$.
The interesting matroids to uncover are the excluded
minors for $\mathbb U_2$ that belong to $\mathcal M_4$. Attention could then
be focussed on members of $\mathcal M_4$ having these matroids as minors. It is possible that
%these will form thin, highly structured classes.
these will form highly structured classes of bounded branch width.

With the results of this paper, %further work on detachable pairs,
and the characterisation of the $\mathbb U_2$- and 
$\mathbb H_5$-representable $U_{2,5}$- and $U_{3,5}$-fragile matroids, 
there is real hope that obtaining the full list of excluded minors for 
these classes is an achievable goal. Beyond these classes all bets are off. Experience with
graph minors tells us that we must expect to hit a wall quite soon --- consider, for example,
the excluded minors for the class of toroidal graphs or the class of
\dY-reducible graphs \cite{yu2006more}. We know from \cite{mayhew2008matroids} that there are at least
564 excluded minors for GF$(5)$-representable matroids. It is possible that obtaining the
full list will be forever beyond our reach. But the quest is surely a worthy one.
  
\section{Preliminaries and the main theorems}

In this section we gather preliminaries %on matroid connectivity and representation theory
that are used throughout the paper. We will then be able to state the main results: \cref{mainthm1,mainthm2}.
In particular, \cref{mainthm2} implies \cref{one}.
Most of the relevant results and terminology on matroid connectivity can either be found in Oxley \cite{oxley2011matroid} or in the recent literature on removing elements relative to a fixed basis \cite{oxley2008maintaining,whittle2013fixed,brettell2014splitter}. The results and terminology on matroid representation theory can be found in \cite{pendavingh2010lifts,pendavingh2010confinement,mayhew2010stability}. %Any undefined terminology or notation used in this paper will follow these sources.

We write ``by orthogonality'' to refer to the property that a circuit and a cocircuit cannot meet in one element.
In the context of partitions of the form $(X,\{e\},Y)$,  we will also write ``by orthogonality'' to refer to an application of the next lemma.

\begin{lemma}
\label{orthogpartition}
Let $e$ be an element of a matroid $M$, and let $(X,\{e\},Y)$ be a partition of $E(M)$. Then $e\in \cl(X)$ if and only if $e\notin \cl^{*}(Y)$. 
\end{lemma}

\subsection*{Connectivity}

The following results are well known. 

\begin{lemma}
\label{longline3conn}
Let $M$ be a $3$-connected matroid. If $X$ is a rank-$2$ subset of $E(M)$ and $|X|\geq 4$, then $M\del x$ is $3$-connected for all $x\in X$. 
\end{lemma}

\begin{lemma}[Bixby's Lemma~\cite{bixby1982simple}]
\label{bixby}
Let $M$ be a $3$-connected matroid, and let $e\in E(M)$. Then $\si(M/e)$ or $\co(M\del e)$ is $3$-connected.
\end{lemma}

The next three results state elementary properties of $3$-separations that we shall use frequently. We use the notation $e\in \cl^{(*)}(X)$ to mean $e\in \cl(X)$ or $e\in \cl^{*}(X)$. 

\begin{lemma}
\label{calc1}
 Let $X$ be an exactly $3$-separating set in a $3$-connected matroid, and suppose that $e\in E(M)-X$. Then $X\cup e$ is $3$-separating if and only if $e\in \cl^{(*)}(X)$.
\end{lemma}

\begin{lemma}
\label{calc2}
Let $(X,Y)$ be an exactly $3$-separating partition of a $3$-connected matroid $M$. Suppose $|X|\geq 3$ and $x\in X$. Then
\begin{itemize}
 \item[(i)] $x\in \cl^{(*)}(X-x)$; and
 \item[(ii)] $(X-x,Y\cup x)$ is exactly $3$-separating if and only if $x$ is in exactly one of $\cl(X-x)\cap \cl(Y)$ and $\cl^{*}(X-x)\cap \cl^{*}(Y)$.
\end{itemize}
\end{lemma}

\begin{lemma}[{\cite[Lemma 2.11]{brettell2014splitter}}]
\label{gutspluscoguts1}
Let $(X,Y)$ be a $3$-separation of a $3$-connected matroid $M$. If $X\cap \cl(Y)\neq \emptyset$ and $X\cap \cl^{*}(Y)\neq \emptyset$, then $|X\cap \cl(Y)|=1$ and $|X\cap \cl^{*}(Y)|=1$.
\end{lemma}

Let $M$ be a matroid.
A $3$-separation $(X,Y)$ of $M$ is a \textit{vertical $3$-separation} if $\min\{r(X),r(Y)\}\geq 3$. We say that a partition $(X,\{z\},Y)$ is a \textit{vertical $3$-separation} of $M$ when both $(X\cup \{z\},Y)$ and $(X,Y\cup \{z\})$ are vertical $3$-separations and $z\in \cl(X)\cap \cl(Y)$. We will write $(X,z,Y)$ for $(X,\{z\},Y)$.
If $(X,z,Y)$ is a vertical $3$-separation of $M$, then we say that $(X,z,Y)$ is a \emph{cyclic $3$-separation} of $M^*$.

%A vertical $3$-separation is a particular example of what we call a ``path of $3$-separations.
A \textit{path of $3$-separations} of $M$ is a partition $(P_1,\ldots, P_n)$ of $E(M)$ such that $(P_1\cup \cdots \cup P_i, P_{i+1}\cup \cdots \cup P_n)$ is a $3$-separation of $M$ for each $i\in \{1,\ldots, n-1\}$.
In particular, a vertical $3$-separation $(X,z,Y)$ is a path of $3$-separations.

%We observe the following connection between elements that are $(N,B)$-robust but not $(N,B)$-strong and vertical $3$-separations.

\begin{lemma}[{\cite[Lemma 3.5]{stabilizers}}]
\label{existsv3sep}
Let $M$ be a $3$-connected matroid, and $e\in E(M)$.
%If $\si(M/e)$ is not $3$-connected, then $M$ has a vertical $3$-separation $(X,e,Y)$. 
The matroid $M$ has a vertical $3$-separation $(X,e,Y)$ if and only if $\si(M/e)$ is not $3$-connected.
\end{lemma}

Let $k$ be a positive integer, and let $(P,Q)$ be a $k$-separation. We call the set $\cl(P)\cap \cl(Q)$ the \textit{guts} of $(P,Q)$, and $\cl^{*}(P)\cap \cl^{*}(Q)$ the \textit{coguts} of $(P,Q)$.
We also say that an element $z\in \cl(P) \cap \cl(Q)$ is a \textit{guts element}, and $z\in \cl^*(P) \cap \cl^*(Q)$ is a \textit{coguts element}.

We write ``by uncrossing'' to refer to an application of the next result.%The following Lemma is a consequence of the submodularity of the matroid connectivity function.  

\begin{lemma}
Let $M$ be a $3$-connected matroid, and let $X$ and $Y$ be $3$-separating subsets of $E(M)$. Then the following hold.
\begin{itemize}
 \item[(i)] If $|X\cap Y|\geq 2$, then $X\cup Y$ is $3$-separating.
 \item[(ii)] If $|E(M)-(X\cup Y)|\geq 2$, then $X\cap Y$ is $3$-separating.
\end{itemize}
\end{lemma}

\subsection*{Series classes}
We will use the following two results on series classes. We omit the easy proof of the first lemma.

\begin{lemma}
\label{seriesindependent}
Let $M$ be a matroid such that $\co(M)$ is $3$-connected. If $S$ and $S'$ are distinct series classes of $M$, then either $S\cup S'$ is independent, or $\co(M)\cong U_{1,3}$. %is a rank-$3$ matroid whose elements lie on one of three lines through $u$.
\end{lemma}

%\begin{proof}
% Suppose $S\cup S'$ is dependent in $M$. Then there is a circuit $C$ of $M$ contained in $S\cup S'$ that meets both $S$ and $S'$. But then we can contract all but one element of each series class so $\co(M)$ has a parallel pair; a contradiction.  
%\end{proof}

When $S$ is a series class of size two, we say $S$ is a \emph{series pair}.

\begin{lemma}
 \label{seriesnotbasisstrong}
Let $M$ be a $3$-connected matroid, and let $u\in E(M)$ be an element such that $\co(M\del u)$ is $3$-connected and $\co(M\del u)\ncong U_{1,3}$. Let $S$ be a non-trivial series class of $M\del u$. If there is some element $s\in S$ such that $\si(M/s)$ is not $3$-connected, then 
\begin{itemize}
 \item[(i)] $|S|=2$;%$S$ is a series pair of $M\del u$;
 \item[(ii)] $M\del u$ has exactly two distinct non-trivial series classes; and
 \item[(iii)] $\si(M/s')$ is $3$-connected, where $S=\{s,s'\}$.
\end{itemize}
\end{lemma}

\begin{proof}
Suppose that there is some element $s\in S$ such that $\si(M/s)$ is not $3$-connected. 
Observe that $r_{M^*}(S \cup u) = 2$.  It follows that if $|S| \ge 3$, then $M^*\ba s$, and hence $M/s$, is $3$-connected for all $s \in S$ by \cref{longline3conn}; %.  Thus $M/s$ is $3$-connected for all $s \in S$;
a contradiction.
So $|S|=2$, and (i) holds.
Henceforth, we let $S=\{s,s'\}$.

%Since $M/s$ is not $3$-connected, it follows from the dual of Lemma \ref{longline3conn} that $|S|=2$, so (i) holds. 

We now consider (ii) and (iii).
By \cref{existsv3sep}, $M$ has a vertical $3$-separation $(A,s,B)$.
Without loss of generality, we may assume that $A$ is coclosed, and that $u \in A$.
%We may assume without loss of generality that $u\in A$.
Then $(A-u,B)$ is a $2$-separation of $M/s\del u$, and the matroid $M/s\del u$ is $3$-connected up to series classes because $\co(M\del u)$ is $3$-connected. Hence $A-u$ or $B$ is contained in a series class of $M\del u$.
%We claim that $S$ and $A-u$ are disjoint.
Since $S\cup u$ is a triad containing $s$, and $(A,s,B)$ is a vertical $3$-separation,
it follows from orthogonality (see \cref{orthogpartition}) that $S\cup u$ is not contained in $A\cup s$ or $B \cup s$.
So $s' \in B$.
%It now follows that
If $B$ is contained in a series class of $M \ba u$, then $s$ is also in this series class; a contradiction.  So
$A-u$ is contained in a series class, distinct from $S$.
As $A$ is coclosed in $M$, we have that $A-u$ is a series class in $M \ba u$.
Since $s\in \cl(A)$, there is a circuit $C$ of $M$ such that $s\in C\subseteq A\cup s$.
Moreover, $u\in C$ by orthogonality with $C$ and the triad $S\cup u$.

Next we claim that $A-u$ and $S$ are the only series classes of $M\del u$. Suppose there is some series pair $S'$ of $M\del u$ disjoint from $S\cup (A-u)$. Then $S'\cup u$ is a triad of $M$ that meets the circuit $C$ in the single element $u$; a contradiction to orthogonality. Thus $S$ and $A-u$ are the only series classes of $M\del u$, so (ii) holds. 

Finally, %let $s'\in S-s$.
suppose that $\si(M/s')$ is not $3$-connected. Then, by \cref{existsv3sep}, $M$ has a vertical $3$-separation $(A',s',B')$.
We may assume, without loss of generality, that $A'$ is coclosed and that $u\in A'$.
By the same argument as used earlier, $A'-u$ is a series class of $M\del u$.
By (ii), $A'-u=A-u$, and thus $(A',s',B')=(A,s',(B-s')\cup s)$. Then $s'\in \cl(A)$, so there is some circuit $C'$ of $M$ such that $s'\in C'\subseteq A\cup s'$, and $u\in C'$ by orthogonality. But then we have distinct circuits $C\subseteq A\cup s$ and $C'\subseteq A\cup s'$ such that $u\in C\cap C'$. By circuit elimination, there is a circuit $C''$ of $M$ such that $C''\subseteq (A-u) \cup S$. Thus, by \cref{seriesindependent}, $\co(M\del u)\cong U_{1,3}$, a contradiction.
Therefore (iii) holds.   
\end{proof}

\subsection*{Fans}
A subset $F$ of the ground set of a matroid, with $|F| \ge 3$, is a \emph{fan} if there is an ordering $(f_1, f_2, \dotsc, f_k)$ of the elements of $F$ such that
\begin{itemize}
  \item[(a)] $\{f_1,f_2,f_3\}$ is either a triangle or a triad, and\label{fani}
  \item[(b)] for all $i \in \{1,2,\dotsc,k-3\}$, if $\{f_i, f_{i+1}, f_{i+2}\}$ is a triangle, then $\{f_{i+1}, f_{i+2}, f_{i+3}\}$ is a triad, while if $\{f_i, f_{i+1}, f_{i+2}\}$ is a triad, then $\{f_{i+1}, f_{i+2}, f_{i+3}\}$ is a triangle.\label{fanii}
\end{itemize}
When there is no ambiguity, we also say that the ordering $(f_1,f_2,\dotsc,f_k)$ is a fan.
If $F$ has a fan ordering $(f_1, f_2, \dotsc, f_k)$ where $k \geq 4$, then $f_1$ and $f_k$ are the \emph{ends} of $F$, and $f_2, f_3, \dotsc, f_{k-1}$ are the \emph{internal elements} of $F$.
%A fan ordering is unique, up to reversal, when $k \ge 5$. %cite Oxley and Wu?

Let $F$ be a fan with ordering $(f_1, f_2, \dotsc, f_k)$ where $k \geq 4$,
and let $i \in \{1,2,\dotsc,k\}$ if $k \geq 5$, or $i \in \{1,4\}$ if $k=4$.
An element $f_i$ is a \emph{spoke element} of $F$ if $\{f_1, f_2, f_3\}$ is a triangle and $i$ is odd, or if $\{f_1, f_2, f_3\}$ is a triad and $i$ is even; otherwise $f_i$ is a \emph{rim element}.

We say that a fan $F$ is \emph{maximal} if there is no fan that properly contains $F$.

We employ the following results when we encounter fans.

\begin{lemma}[{\cite[Lemma 2.12]{brettell2014splitter}}]
\label{fanends}
Let $M$ be a $3$-connected matroid with $|E(M)| \ge 7$.  %or $M$ is not a wheel or a whirl of rank at most three.
Suppose that $M$ has a fan $F$ of at least $4$ elements, and let $f$ be an end of $F$.
\begin{itemize}
 \item[(i)] If $f$ is a spoke element, then $\co(M\del f)$ is $3$-connected and $\si(M/f)$ is not $3$-connected.
 \item[(ii)] If $f$ is a rim element, then $\si(M/f)$ is $3$-connected and $\co(M\del f)$ is not $3$-connected.
\end{itemize}
\end{lemma}

\begin{lemma}[{\cite[Lemma 3.3]{brettell2014splitter}}]
\label{f2f3}
Let $M$ be a matroid with distinct elements $f_1,f_2,f_3,f_4$. If the only triangle containing $f_3$ is  $\{f_1,f_2,f_3\}$ and the only triad containing $f_2$ is $\{f_2,f_3,f_4\}$, then $\si(M/f_3)\cong \co(M\del f_2)$.
\end{lemma}

%Let $M$ be a matroid, and fix a basis $B$ of $M$.
%Later, we will look to remove elements from $M$ in a particular way relative to $B$.
%In this setting, %the labelling of the fan, relative to $B$, is important; towards this end,
%we will require the following definitions.
%Let $F$ be a $4$-element fan of $M$. % with ordering $(f_1,f_2,f_3,f_4)$ where $\{f_1,f_2,f_3\}$ is a triangle. We say that $F$ is a \textit{type-I fan relative to $B$} or a \textit{type-II fan relative to $B$} if $B\cap F$ is $\{f_1,f_3\}$ or $\{f_1,f_3,f_4\}$, respectively.
%We say that $F$ is a \textit{type-I fan relative to $B$} if $F$ has an ordering $(f_1,f_2,f_3,f_4)$ where $\{f_1,f_2,f_3\}$ is a triangle and $F \cap B = \{f_1,f_3\}$;
%and $F$ is a \textit{type-II fan relative to $B$} if $F$ has an ordering $(f_1,f_2,f_3,f_4)$ where $\{f_1,f_2,f_3\}$ is a triangle and $B\cap F = \{f_1,f_3,f_4\}$.

\subsection*{Retaining an $N$-minor}

% A \textit{type-I fan relative to $B$} in $M$ is a $4$-element fan  with ordering $(f_1,f_2,f_3,f_4)$ where $\{f_1,f_2,f_3\}$ is a triangle and $F\cap B=\{f_1,f_3\}$. A \textit{type-II fan relative to $B$} in $M$ is a $4$-element fan $F$ with ordering $(f_1,f_2,f_3,f_4)$ where $\{f_1,f_2,f_3\}$ is a triangle and $F\cap B=\{f_1,f_3,f_4\}$.

Let $M$ and $N$ be matroids, and let $x$ be an element of $M$.
If $M\del x$ has an $N$-minor, then $x$ is $N$-\textit{deletable}.
If $M/x$ has an $N$-minor, then $x$ is $N$-\textit{contractible}.
If neither $M\del x$ nor $M/x$ has an $N$-minor, then $x$ is $N$-\textit{essential}.
If $x$ is both $N$-deletable and $N$-contractible, then we say that $x$ is \textit{$N$-flexible}.
%A matroid $M$ is \textit{$N$-fragile} if no element of $M$ is $N$-flexible.
%If $M$ is $N$-fragile and has an $N$-minor, then $M$ is \textit{strictly $N$-fragile}.
%In this paper, $N$-fragile will always mean strictly $N$-fragile.
A matroid $M$ is \textit{$N$-fragile} if $M$ has an $N$-minor, and no element of $M$ is $N$-flexible
(note that some authors refer to this as ``strictly $N$-fragile'').

For $X \subseteq E(M)$, we will also say that $X$ is \emph{$N$-deletable} (or \emph{$N$-contractible}) when $M \del X$ (or $M / X$, respectively) has an $N$-minor.

The next two %three
results give some %useful
conditions for when we can keep an $N$-minor when dealing with $2$-separations.

\begin{lemma}[{\cite[Lemma 4.3]{brettell2014splitter}}]
\label{sicominor}
Let $N$ be a $3$-connected matroid such that $|E(N)|\geq 4$. If $M$ has an $N$-minor, then $\si(M)$ has an $N$-minor. 
\end{lemma}

%\begin{lemma}[{\cite[Lemma 2.6]{oxley2008maintaining}}]
%\label{keepingN}
%Let $e$ and $f$ be distinct elements of a $3$-connected matroid $M$, and suppose that $\si(M/e)$ is $3$-connected. Then either 
%\begin{itemize}
 %\item[(i)] $M/e\del f$ is connected, or
 %\item[(ii)] $\si(M/e)\cong U_{2,3}$ and $M$ has no triangle containing $\{e,f\}$.
%\end{itemize}
%Moreover, if no non-trivial parallel class of $M/e$ contains $f$, then $M/e/f$ is connected.  
%\end{lemma}

\begin{lemma}[{\cite[Lemma 2.7]{oxley2008maintaining}}]
\label{cplminorlemma}
Let $(X,Y)$ be a $2$-separation of a connected matroid $M$ and let $N$ be a $3$-connected minor of $M$. Then $\{X,Y\}$ has a member $S$ such that $|S \cap E(N)|\leq 1$. Moreover, when $s\in S$,
\begin{itemize}
 \item[(i)] if $M/s$ is connected, then $M/s$ has an $N$-minor; and
 \item[(ii)] if $M\del s$ is connected, then $M\del s$ has an $N$-minor.
\end{itemize}
\end{lemma}

Let $(X,z,Y)$ be a vertical $3$-separation of a matroid $M$.
Then $(X,Y)$ is a $2$-separation of $M/z$ such that $|X|\geq 3$ and $|Y|\geq 3$.
Let $N$ be a $3$-connected matroid with $|E(N)| \ge 3$.
If $M/z$ has an $N$-minor, then it follows from \cref{cplminorlemma} %a well-known result (see \cite[Proposition~8.3.5]{oxley2011matroid})
that either $|X \cap E(N)|\leq 1$ or $|Y \cap E(N)|\leq 1$.
When $|X \cap E(N)|\leq 1$, we refer to $X$ as the \textit{non-$N$-side} and $Y$ as the \textit{$N$-side} of $(X,z,Y)$.

%\begin{lemma}
% \label{CPL1}\cite[Lemma 4.5]{brettell2014splitter}
%Let $N$ be a $3$-connected minor of a $3$-connected matroid $M$. Let $(X,\{z\},Y)$ be a vertical $3$-separation of $M$ such that $M/z$ has an $N$-minor, where $|X\cap E(N)|\leq 1$. If $Y\cup \{z\}$ is closed, then there is at most one element of $X$ that is not $N$-flexible. Moreover, if such an element $x$ exists, then $x\in \cl^{*}(Y)$ and $z\in \cl(X-\{x\})$  
%\end{lemma}

The following result is %an essential tool for dealing with elements that are $(N,B)$-robust but not $(N,B)$-strong. It is
a routine upgrade of \cite[Lemma~4.5]{brettell2014splitter} that also covers the case when the $N$-side of the vertical $3$-separation is not closed. 

%\begin{lemma}
%\label{CPL2}
%Let $N$ be a $3$-connected minor of a $3$-connected matroid $M$. Let $(X,z,Y)$ be a vertical $3$-separation of $M$ such that $M/z$ has an $N$-minor, where $|X\cap E(N)|\leq 1$. 
%\begin{itemize}
 %\item[(i)] If $Y\cup z$ is closed, then every element of $X$ is $N$-contractible and there is at most one element $x$ of $X$ that is not $N$-deletable. Moreover, if such an element $x$ exists, then $x\in \cl^{*}(Y)$ and $z\in \cl(X-x)$. 
 %\item[(ii)] If $Y\cup z$ is not closed, then every element of $X-\cl_M(Y)$ is $N$-contractible, and at most one element of $X$ is not $N$-deletable. Moreover, if such an element $x$ exists, then $x\in \cl^{*}_M(\cl_M(Y))-\cl_M(Y)$ and $z\in \cl_M(X-(\cl_M(Y)\cup x))$.
%\end{itemize}
%\end{lemma}
\begin{lemma}
\label{CPL2}
Let $N$ be a $3$-connected minor of a $3$-connected matroid $M$. Let $(X,z,Y)$ be a vertical $3$-separation of $M$ such that $M/z$ has an $N$-minor, where $|X\cap E(N)|\leq 1$. Then, every element of $X$ is either $N$-deletable or $N$-contractible in $M/z$.  In particular, letting $Y' = \cl_M(Y)-z$,
\begin{itemize}
 \item[(i)] every element of $X-Y'$ is $N$-contractible in $M/z$, and
 \item[(ii)]at most one element of $X$ is not $N$-deletable; moreover, if such an element $x$ exists, then $x\in \cl^{*}_M(Y')-Y'$ and $z\in \cl_M(X-(Y'\cup x))$.
\end{itemize}
\end{lemma}

\begin{proof}
  It is immediate from the proof of \cite[Lemma 4.5]{brettell2014splitter} that the lemma holds when $Y \cup z$ is closed; in particular, (i) holds.
  We may therefore assume that $Y\cup z$ is not closed. Let $s\in X\cap \cl_M(Y)$. We first show that $s$ is $N$-deletable. Since $(X,Y)$ is a $2$-separation of the connected matroid $M/z$ and $s\in \cl_{M/z}(X)\cap \cl_{M/z}(Y)$, it follows that $M/z\del s$ is connected. Then, by \cref{cplminorlemma}, $M/z\del s$ has an $N$-minor, so $s$ is $N$-deletable. %in $M$.
  Thus any element of $X$ that is not $N$-deletable belongs to $X-Y'$. %The remaining properties will follow from the next claim.  
%This proves (i).
%We now consider (ii). % when $Y\cup z$ is not closed.

\begin{claim}
\label{bixbysetup}
The partition $(X-s, z, Y\cup s)$ is a vertical $3$-separation of $M$.
%$\si(M/s)$ is not $3$-connected. 
\end{claim}

\begin{subproof}
%We must show that $((X-s)\cup z, Y\cup s)$ and $(X-s, Y\cup \{s,z\})$ are vertical $3$-separations of $M$.
By \cref{calc2}(i), $s\in \cl_M^{(*)}(X-s)$.
Since $s\in \cl_M(Y)$ it follows from orthogonality that $s\notin \cl_M^{*}(X-s)$.
Therefore $s\in \cl_M(X-s)$ and, by \cref{calc2}(ii), $(X-s, Y\cup \{s,z\})$ is exactly $3$-separating.
By a similar argument, $((X-s)\cup z, Y\cup s)$ is exactly $3$-separating.
As $s\in \cl_M(X-s)$, we see that $r(X-s)=r(X)\geq 3$.
%Therefore $(X-s, Y\cup \{s,z\})$ and $((X-s)\cup z, Y\cup s)$ are vertical $3$-separations of $M$.
Hence the partition $(X-s,z, Y\cup s)$ is a vertical $3$-separation of $M$.
\end{subproof}

By repeatedly applying \ref{bixbysetup}, we see that $(X-Y', z, Y')$ is a vertical $3$-separation of $M$ with $|(X-Y') \cap E(N)| \le 1$. %satisfying the hypotheses of (i), so the remaining properties of (ii) follow by applying (i) to this partition. 
As $Y' \cup z$ is closed, (ii) holds.
The fact that each element of $X$ is either $N$-deletable or $N$-contractible now follows from \cref{gutspluscoguts1}.
\end{proof}

%It remains to show that $r(X-\{s\})\geq 3$. Seeking a contradiction, suppose that $r(X-\{s\})=2$. Then $2=r(X-\{s\})=r(X)$ since $s\in \cl_M(X-\{s\})$, which contradicts the fact that $r(X)\geq 3$. 

%Let $N$ be a minor of a matroid $M$, and
%suppose $C$ and $D$ are disjoint subsets of $E(M)$ such that $M/C \ba D \cong N$.
%We call the ordered pair $(C,D)$ an \emph{$N$-labelling of $M$}. 

The next result is a consequence of \cref{CPL2} and Bixby's Lemma.

\begin{lemma}
\label{ZspanningB}
Let $N$ be a $3$-connected minor of a $3$-connected matroid~$M$. 
Let $(X,z,Y)$ be a vertical $3$-separation of $M$ such that $M/z$ has an $N$-minor, $|X\cap E(N)|\leq 1$, 
%Suppose $M/z$ has an $N$-minor and $\si(M/z)$ is not $3$-connected.
%Let $(X,z,Y)$ be a vertical $3$-separation such that $|X \cap E(N)| \le 1$
and $Y \cup z$ is closed.
Then there is at most one element of $X$ that is not $N$-flexible. % in $M$.
Moreover, if  $s\in X$ is not $N$-flexible, % in $M$,
then $s$ is $N$-contractible %in $M$
and $\si(M/s)$ is $3$-connected.
\end{lemma}

\subsection*{Representation Theory}
\label{stablesetup}

A \textit{partial field} is a pair $(R, G)$, where $R$ is a commutative ring with unity, and $G$ is a subgroup of the group of units of $R$ such that $-1 \in G$. If $\mathbb{P}=(R,G)$ is a partial field, then we write $p\in \mathbb{P}$ whenever $p\in G\cup \{0\}$.

Let $\mathbb{P}$ be a partial field, and let $A$ be an $X\times Y$ matrix with entries from $\mathbb{P}$. Then $A$ is a $\mathbb{P}$-\textit{matrix} if every subdeterminant of $A$ is contained in $\mathbb{P}$. If $X'\subseteq X$ and $Y'\subseteq Y$, then we write $A[X',Y']$ to denote the submatrix of $A$ induced by $X'$ and $Y'$.
When $X$ and $Y$ are disjoint, if $Z\subseteq X\cup Y$, then we denote by $A[Z]$ the submatrix induced by $X\cap Z$ and $Y\cap Z$, and we denote by $A-Z$ the submatrix induced by $X-Z$ and $Y-Z$.
 
\begin{thm}[{\cite[Theorem 2.8]{pendavingh2010confinement}}]
\label{pmatroid}
Let $\mathbb{P}$ be a partial field, and let $A$ be an $X\times Y$ $\mathbb{P}$-matrix, where $X$ and $Y$ are disjoint. Let
\begin{equation*}
\mathcal{B}=\{X\}\cup \{X\triangle Z : |X\cap Z|=|Y\cap Z|, \det(A[Z])\neq 0\}. 
\end{equation*}
 Then $\mathcal{B}$ is the set of bases of a matroid on $X\cup Y$.
\end{thm}

We say that the matroid in Theorem \ref{pmatroid} is $\mathbb{P}$-\textit{representable}, and that $A$ is a $\mathbb{P}$-\textit{representation} of $M$. We write $M=M[I|A]$ if $A$ is a $\mathbb{P}$-matrix, and $M$ is the matroid whose bases are described in Theorem \ref{pmatroid}. 

Let $A$ be an $X\times Y$ $\mathbb{P}$-matrix, with $X \cap Y = \emptyset$, and let $x\in X$ and $y\in Y$ such that $A_{xy}\neq 0$. Then we define $A^{xy}$ to be the $(X\triangle \{x,y\})\times (Y\triangle \{x,y\})$ $\mathbb{P}$-matrix given by %the operation of pivoting
\begin{displaymath}
  (A^{xy})_{uv} =
\begin{cases}
    A_{xy}^{-1} \quad & \textrm{if } uv = yx\\
    A_{xy}^{-1} A_{xv} & \textrm{if } u = y, v\neq x\\
    -A_{xy}^{-1} A_{uy} & \textrm{if } v = x, u \neq y\\
    A_{uv} - A_{xy}^{-1} A_{uy} A_{xv} & \textrm{otherwise.}
\end{cases}
\end{displaymath}
We say that $A^{xy}$ is obtained from $A$ by \textit{pivoting} on $xy$.
%After pivoting, $x$ labels a column, and $y$ labels a row. 
Note that $A^{xy}$ is a $\mathbb{P}$-matrix, by \cite[Proposition~3.3]{semple1996partial}.

Two $\mathbb{P}$-matrices are \textit{scaling equivalent} if one can be obtained from the other by repeatedly scaling rows and columns by non-zero elements of $\mathbb{P}$. Two $\mathbb{P}$-matrices are \textit{geometrically equivalent} if one can be obtained from the other by a sequence of the following operations: scaling rows and columns by non-zero entries of $\mathbb{P}$, permuting rows, permuting columns, and pivoting.
 
%We assume throughout the paper that an excluded minor $M$ has a coindependent pair of elements $a,b$ such that $M\del a,b$ is $3$-connected with an $N$-minor, where $N$ is a $3$-connected strong $\mathbb{P}$-stabilizer. %Thus $M\del a$, $M\del b$, and $M\del a,b$ are all $N$-stable.

Let $\mathbb{P}$ be a partial field, and let $M$ and $N$ be %$\mathbb{P}$-representable
matroids such that $N$ is a %$3$-connected
minor of $M$. Suppose that the ground set of $N$ is $X'\cup Y'$, where $X'$ is a basis of $N$. We say that $M$ is $\mathbb{P}$-\textit{stabilized by $N$} if, whenever $A_1$ and $A_2$ are $X\times Y$ $\mathbb{P}$-matrices, with $X'\subseteq X$ and $Y'\subseteq Y$, such that
\begin{enumerate}
 \item[(i)] $M=M[I|A_1]=M[I|A_2]$,
 \item[(ii)] $A_1[X',Y']$ is scaling equivalent to $A_2[X',Y']$, and
 \item[(iii)] $N=M[I|A_1[X',Y']]=M[I|A_2[X',Y']],$ 
\end{enumerate}
then $A_1$ is scaling equivalent to $A_2$.
If $M$ is $\mathbb{P}$-stabilized by $N$, and every $\mathbb{P}$-representation of $N$ extends to a $\mathbb{P}$-representation of $M$, then we say $M$ is \emph{strongly $\mathbb{P}$-stabilized} by $N$.
%When $N$ is a (strong) $\mathbb{P}$-\textit{stabilizer for $M$}, we also say that $M$ is (strongly) \emph{$\mathbb{P}$-stabilized} by $N$.

Let $\mathcal{M}$ be a class of matroids. We say that $N$ is a \textit{$\mathbb{P}$-stabilizer for $\mathcal{M}$} if, for every $3$-connected $\mathbb{P}$-representable matroid $M\in \mathcal{M}$ with an $N$-minor, $M$ is $\mathbb{P}$-stabilized by $N$.
%We say that $N$ is a \textit{strong $\mathbb{P}$-stabilizer for $\mathcal{M}$} if $N$ is a $\mathbb{P}$-stabilizer for $\mathcal{M}$ and, for every $3$-connected $\mathbb{P}$-representable matroid $M\in \mathcal{M}$ with an $N$-minor, every $\mathbb{P}$-representation of $N$ extends to a $\mathbb{P}$-representation of $M$.
We say that $N$ is a \textit{strong $\mathbb{P}$-stabilizer for $\mathcal{M}$} if, for every $3$-connected $\mathbb{P}$-representable matroid $M\in \mathcal{M}$ with an $N$-minor, $M$ is strongly $\mathbb{P}$-stabilized by $N$.
Usually, we will be interested in the class of $\mathbb{P}$-representable matroids for some partial field~$\mathbb{P}$.
In this case, when it is clear from context, we will simply say ``$N$ is a strong $\mathbb{P}$-stabilizer''. %omitting reference to the class of $\mathbb{P}$-representable matroids.
%When we say that $N$ is a ``strong $\mathbb{P}$-stabilizer'' omitting reference to a class of matroids, the omitted class is the class of $\mathbb{P}$-representable matroids.

\subsection*{Certifying non-representability}
Let $M$ be a matroid with a minor $N$.
%Recall that, for $X \subseteq E(M)$, we say that $X$ is $N$-deletable when $M \del X$ has an $N$-minor.
If $M$ has a pair of elements $\{a,b\}$ such that $M\ba a,b$ is $3$-connected and has an $N$-minor, then we say $\{a,b\}$ is a \emph{deletion pair with respect to $N$}.
If $M$ has a pair of elements $\{a,b\}$ that are $N$-deletable, $M \del a,b$ is connected, and $\co(M \del a)$, $\co(M \del b)$, and $\co(M \del a,b)$ are $3$-connected, then we say $\{a,b\}$ is a \emph{weak deletion pair with respect to $N$}.

When $B$ is a basis for a matroid $M$, we write $B^*$ to denote $E(M)-B$.
%Let $B$ be a basis for a matroid $M$, let $B^*=E(M)-B$, and let $Z$ be a subset of $E(M)$.
When $Z \subseteq E(M)$, we write $M_B[Z]$ to denote the minor $M/(B-Z)\del (B^{*}-Z)$. 

Throughout the rest of this section, we assume that $\mathbb{P}$ is a partial field. %, $N$ is a $3$-connected $\mathbb{P}$-representable matroid that is a strong $\mathbb{P}$-stabilizer, $M$ is a $3$-connected matroid with an $N$-minor, and $\{a,b\} \subseteq E(M)$ is a good $N$-deletable pair.

\begin{thm}[{\cite[Theorem 5.5]{mayhew2010stability}}]
\label{companion}
Let $M$ and $N$ be $3$-connected matroids. %, and let $\mathbb{P}$ be a partial field.
Suppose $M$ has an $N$-minor,
$N$ is a $\mathbb{P}$-representable matroid that is a strong $\mathbb{P}$-stabilizer,
and $\{a,b\} \subseteq E(M)$ is a weak deletion pair with respect to $N$ such that
%and there are distinct elements $a,b \in E(M)$ such that
$M\del a$ and $M\del b$ are $\mathbb{P}$-representable. %, and
%$M \del \{a,b\}$ is $3$-connected and has an $N$-minor.
%
Let $D$ be an $X_N\times Y_N$ $\mathbb{P}$-matrix such that $N=M[I | D]$. Choose $B, E_N\subseteq E(M)-\{a,b\}$ such that $B$ is a basis of $M\del \{a,b\}$, $X_N\subseteq B$, and $M_B[E_N]=N$.
Then there exists a $B\times B^*$ matrix~$A$ with entries in $\mathbb{P}$ such that
\begin{itemize}
\item[(i)] $A-a$ and $A-b$ are $\mathbb{P}$-matrices,
\item[(ii)] $M[I|A-a]=M\del a$ and $M[I|A-b]=M\del b$, and
\item[(iii)] $A[E_N]$ is scaling equivalent to $D$.
\end{itemize}
Moreover, the matrix~$A$ is unique up to row and column scaling.
\end{thm}

Usually, we will apply \cref{companion} to a matroid $M$ that is not $\mathbb{P}$-representable.
We call the matrix~$A$ a ``companion matrix'' for $M$.
%
%We call the matrix $A$ of \cref{companion} a \textit{companion matrix} for $M$. 

%%%
\begin{definition}
Let $M$ be a matroid % with a basis $B$. %, and let $B^*=E(M)-B$.
and let $E(M)=X \cup Y$ where $X$ and $Y$ are disjoint.
%Let $A$ be a $B \times B^*$
Let $A$ be an $X \times Y$
matrix with entries in $\mathbb{P}$ such that, for some distinct $a, b \in Y$, both $A-a$ and $A-b$ are $\mathbb{P}$-matrices, $M \del a=M[I|A-a]$, and $M \del b=M[I|A-b]$. %, and $M \ba a$ and $M \ba b$ are $\mathbb{P}$-representable.
Then $A$ is an $X \times Y$ \emph{companion $\mathbb{P}$-matrix} for $M$.
\end{definition}

Let $M$ be an excluded minor for the class of $\mathbb{P}$-representable matroids. %, for some partial field $\mathbb{P}$.
Then it is easily seen that $M$ is $3$-connected.
By \cref{companion}, given a $3$-connected matroid $N$ that is a minor of $M$ and a strong $\mathbb{P}$-stabilizer, a $\mathbb{P}$-representation for $N$, and a weak deletion pair, there is a $B \times B^*$ companion $\mathbb{P}$-matrix~$A$ for $M$, where $B$ is an appropriately chosen basis of $M$. % and $B^* = E(M)-B$. %, and $A$ is unique up to row and column scaling.
%%%
%Assume that an excluded minor $M$ for the class of $\mathbb{P}$-representable matroids has a pair of elements $a,b$ such that $M\del a,b$ is $3$-connected with an $N$-minor. For such an excluded minor $M$, we can construct a ``companion matrix'' for $M$. That is, a matrix $A$ with entries in $\mathbb{P}$ that is unique up to scaling such that $M\del a=M[I|(A-a)]$ and $M\del b=M[I|(A-b)]$. Moreover, we can choose a basis $B$ for $M$ such that the $B\times B^{*}$ companion matrix $A$ displays a $2\times 2$ ``bad submatrix''. That is, for some $x,y\in B$ the set $\{a,b,x,y\}$ incriminates $(M,A)$. Thus when we say that $B$ is a basis of $M$ in this paper, we mean a basis such that $B$ gives a $B\times B^{*}$ companion matrix $A$ and for some elements $x,y\in B$ the set $\{a,b,x,y\}$ incriminates $(M,A)$. 

A companion matrix for an excluded minor contains a certificate of non-representability over $\mathbb{P}$.

\begin{definition}
Let $B$ be a basis of $M$, and let $A$ be a $B\times B^*$ matrix with entries in $\mathbb{P}$. A subset $Z$ of $E(M)$ \textit{incriminates} the pair $(M, A)$ if $A[Z]$ is square and one of the following holds: 
\begin{enumerate}
 \item[(i)] $\det(A[Z])\notin \mathbb{P}$,
 \item[(ii)] $\det(A[Z])=0$ but $B\triangle Z$ is a basis of $M$, or
 \item[(iii)] $\det(A[Z])\neq 0$ but $B\triangle Z$ is dependent in $M$.
\end{enumerate}
\end{definition}

The next result follows immediately from the definition.

\begin{lemma}
  Let $M$ be a matroid, let $A$ be an $X\times Y$ matrix with entries in $\mathbb{P}$, where $X$ and $Y$ are disjoint, and $X\cup Y=E(M)$. Exactly one of the following statements is true:
\begin{itemize}
 \item[(i)] $A$ is a $\mathbb{P}$-matrix and $M=M[I | A]$, or
 \item[(ii)] there is some $Z\subseteq X\cup Y$ that incriminates $(M, A)$.
\end{itemize}
\end{lemma}

%For the remainder of this section we will assume that $M$ is a %$3$-connected
%matroid, $A$ is an $X \times Y$ matrix with entries in $\mathbb{P}$ such that $X$ and $Y$ are disjoint and $X \cup Y = E(M)$, and $a,b \in Y$.

The next theorem shows that there is some companion matrix $A'$ for $M$ that has a $4$-element incriminating set. 

\begin{thm}[{\cite[Theorem 5.8]{mayhew2010stability}}]
  \label{incrim4}
  Let $M$ be a matroid, let $A$ be an $X \times Y$ companion $\mathbb{P}$-matrix for $M$, let $a,b \in Y$, and suppose that
  %Let $M$ be a matroid, let $A$ be an $X \times Y$ matrix with entries in $\mathbb{P}$ such that $X$ and $Y$ are disjoint and $X \cup Y = E(M)$, and let $a,b \in Y$.
  %Suppose $A-a$ and $A-b$ are $\mathbb{P}$-matrices, $M\del a=M[I | (A-a)]$, $M\del b=M[I | (A-b)]$, and
  $Z\subseteq X\cup Y$ incriminates $(M, A)$. Then there is some $X'\times Y'$ matrix $A'$, and $x,y\in X'$, such that
  \begin{itemize}
    \item[(i)] $a,b\in Y'$,
    \item[(ii)]  $A-a$ is geometrically equivalent to $A'-a$,
    \item[(iii)] $A-b$ is geometrically equivalent to $A'-b$, and
    \item[(iv)] $\{x,y,a,b\}$ incriminates $(M,A')$.
  \end{itemize}
\end{thm}

Let $N$ be a $3$-connected non-binary matroid. % with $|E(N)| \ge 3$.
A matroid~$M$ with an $N$-minor is \textit{$N$-stable} if, whenever $(X,Y)$ is a $2$-separation of $M$ %with $|X\cap E(N)|\leq 1$,
where $X$ is the non-$N$-side,
then the matroid $M_X$, corresponding to $X$ in the $1$- or $2$-sum decomposition of $M$ induced by $(X,Y)$, is binary. 

The following result is proved by Hall, Mayhew, and van Zwam \cite[Propositions~3.1 and 3.2]{hall2011excluded}.

\begin{lemma}
\label{stableisstable}
Let $N$ be a $3$-connected strong $\mathbb{P}$-stabilizer that is non-binary, and let $M$ be a $\mathbb{P}$-representable matroid that has an $N$-minor.
If $M$ is $N$-stable, then $M$ is strongly $\mathbb{P}$-stabilized by $N$.  
\end{lemma}

We next consider how a matroid can lose the property of being $N$-stable after a single-element extension.
We say that a matroid $M$ is \textit{$3$-connected up to series pairs} if $\co(M)$ is $3$-connected and every non-trivial series class of $M$ is a series pair.

\begin{lemma}
\label{plusu24}
  Let $N$ be a $3$-connected non-binary matroid.
Let $M$ be a matroid with an element~$e$ such that $M \del e$ has an $N$-minor, where $e$ is not a coloop.
  Suppose that $M\del e$ is $3$-connected up to series pairs, and that $M$ is not $N$-stable.
  Then $M$ has a $2$-separation $(S \cup e, Q)$ where $S \cup e$ is a triangle and a triad, for some series pair~$S$ of $M \del e$.
\end{lemma}
\begin{proof}
  Suppose $M \del e$ is $3$-connected up to series pairs and $M$ is not $N$-stable.
  If $M$ is $3$-connected up to series pairs, then it is trivially $N$-stable; a contradiction.
  So there is some $2$-separation $(A,B)$ of $M$ with $e \in A$, say.
  Since $(A-e,B)$ is $2$-separating in $M \ba e$, and $M \ba e$ is $3$-connected up to series pairs, we deduce that $|A| \le 3$.
  Let $M = Q_A \oplus_2 Q_B$ be the $2$-sum decomposition corresponding to the $2$-separation $(A,B)$ of $M$.
  Since $M \ba e$ has an $N$-minor and $|E(N)| \ge 4$, we have $|A \cap E(N)| \le 1$.
  As $M$ is not $N$-stable, $Q_A$ has a $U_{2,4}$-minor.
  But $|A| \le 3$, so $Q_A \cong U_{2,4}$. 
  The result follows.
\end{proof}

\begin{definition}
  Let $M$ be a matroid with a $2$-separation $(P,Q)$ where $P$ is a triangle and a triad.
  Then $P$ is an \emph{unstable triple} of $M$.
\end{definition}

  Let $N$ be a $3$-connected non-binary matroid, and let $M$ be a connected matroid with $e \in E(M)$ such that $M \del e$ has an $N$-minor.
  By \cref{plusu24}, if $M\del e$ is $3$-connected up to series pairs, but $M$ is not $N$-stable, then $M$ has an unstable triple, which contains $e$.

  Observe also that if $P$ is an unstable triple, then $P-p$ is a series pair in $M \ba p$, for each $p \in P$.
%$e$ is in the span of a series pair $S$ of $R\del e$ in such a way that the component in the $2$-sum decomposition of $R$ containing $S\cup e$ has a $U_{2,4}$-minor.

  \smallskip

The next lemma gives sufficient conditions for showing a certain minor of $M$ is not $\mathbb{P}$-representable.
It can be proved by a straightforward modification of a result of Mayhew, Whittle, and van Zwam \cite[Theorem 5.12]{mayhew2010stability}. The conditions (iv) and (v) %of \cite[Theorem 5.12]{mayhew2010stability}
are changed from ``$M_B[Z_1]$ and $M_B[Z_2]$ are $3$-connected up to series-parallel classes'' to ``$M_B[Z_1]$ and $M_B[Z_2]$ are $N$-stable'', using \cref{stableisstable}.

\begin{lemma}
\label{notrepcert} 
Let $M$ be a matroid, let $A$ be a $B \times B^*$ matrix with entries in $\mathbb{P}$, where $\{x,y,a,b\}$ incriminates $(M,A)$ for $x,y \in B$ and $a,b \in B^*$.
Let $N$ be a non-binary strong stabilizer for the class of $\mathbb{P}$-representable matroids.
Suppose that $C\subseteq E(M)$ is such that $M_B[C]$ is $N$-fragile. If there exist subsets $Z,Z_1,Z_2 \subseteq E(M)$ such that
\begin{itemize}
 \item[(a)] $a\in Z_1-Z_2$ and $b\in Z_2-Z_1$, 
 \item[(b)] $C\cup \{x,y\}\subseteq Z\subseteq Z_1\cap Z_2$,
 \item[(c)] $M_B[Z]$ is connected,
 \item[(d)] $M_B[Z_1]$ is $N$-stable,
 \item[(e)] $M_B[Z_2]$ is $N$-stable, and
 \item[(f)] $\{x,y,a,b\}$ incriminates $(M_B[Z_1\cup Z_2], A[Z_1\cup Z_2])$,
\end{itemize}
then $M_B[Z_1\cup Z_2]$ is not strongly $\mathbb{P}$-stabilized by $N$.
\end{lemma}

The following special case of \cref{notrepcert} is sufficient for our needs. % in this paper.
\begin{lemma}
\label{notrepcert2} 
Let $M$ be a matroid, let $E = E(M)$, let $A$ be a $B \times B^*$ matrix with entries in $\mathbb{P}$, where $\{x,y,a,b\}$ incriminates $(M,A)$ for $x,y \in B$ and $a,b \in B^*$.
Let $N$ be a non-binary strong stabilizer for the class of $\mathbb{P}$-representable matroids.
If there exists $u \in E-\{x,y,a,b\}$ such that
\begin{itemize}
  \item[(a)] $M_B[E-\{a,b,u\}]$ is connected and has an $N$-minor, and
  \item[(b)] $M_B[E-\{b,u\}]$ and $M_B[E-\{a,u\}]$ are $N$-stable, %and
  %\item[(c)] $\{x,y,a,b\}$ incriminates $(M_B[E-u], A[E-u])$,
\end{itemize}
then $M_B[E-u]$ is not strongly $\mathbb{P}$-stabilized by $N$.
\end{lemma}

We write ``by an allowable pivot'' to refer to an application of either of the next two results.

\begin{lemma}[{\cite[Lemma 5.10]{mayhew2010stability}}]
  Let $A$ be a $B\times B^{*}$ companion $\mathbb{P}$-matrix for $M$. Suppose that $\{x,y,a,b\}$ incriminates $(M,A)$, where $\{x,y\}\subseteq B$ and $\{a,b\}\subseteq B^{*}$. If $p\in \{x,y\}$, $q\in B^{*}-\{a,b\}$, and $A_{pq}\neq 0$, then $\{x,y,a,b\}\triangle \{p,q\}$ incriminates $(M,A^{pq})$.
\end{lemma}

\begin{lemma}[{\cite[Lemma 5.11]{mayhew2010stability}}]
  Let $A$ be a $B\times B^{*}$ companion $\mathbb{P}$-matrix for $M$. Suppose that $\{x,y,a,b\}$ incriminates $(M,A)$, where $\{x,y\}\subseteq B$ and $\{a,b\}\subseteq B^{*}$. If $p\in B-\{x,y\}$, $q\in B^{*}-\{a,b\}$, $A_{pq}\neq 0$, and either $A_{pa}=A_{pb}=0$ or $A_{xq}=A_{yq}=0$, then $\{x,y,a,b\}$ incriminates $(M,A^{pq})$. 
\end{lemma}

The elements of a set $\{x,y,a,b\}$ that incriminates $(M,A)$ label a $2\times 2$ submatrix $A[\{x,y,a,b\}]$ of $A$. We will refer to the next result by saying ``the bad submatrix has no zero entries.''

\begin{lemma}
 \label{nobadsubdetzeros}
 Let $A$ be a $B\times B^{*}$ companion $\mathbb{P}$-matrix for $M$.
Suppose that $\{x,y,a,b\}$ incriminates $(M,A)$, where $\{x,y\}\subseteq B$ and $\{a,b\}\subseteq B^{*}$.
Then $A_{ij}\not=0$ for $i\in \{x,y\}$ and $j\in\{a,b\}$. 
\end{lemma}

\begin{proof}
  Towards a contradiction, suppose that $A_{ij}=0$ for some $i\in \{x,y\}$ and $j\in\{a,b\}$.
  We may assume without loss of generality that $A_{xb}=0$. Then $\det(A[\{x,y,a,b\}])\in \mathbb{P}$. Since $\{x,y,a,b\}$ incriminates the pair $(M,A)$, it follows that either
\begin{enumerate}
 \item[(i)] $\det(A[\{x,y,a,b\}])=0$ but $B\triangle \{x,y,a,b\}$ is a basis of $M$, or %dependent in $M[I|A]$; or 
 \item[(ii)] $\det(A[\{x,y,a,b\}])\neq 0$ but $B\triangle \{x,y,a,b\}$ is dependent in $M$.%a basis of $M[I|A]$.
\end{enumerate}

Assume that (i) holds.
As %$B\triangle \{x,y,a,b\}$ is a basis of $M$, but
$\det(A[\{x,y,a,b\}])=A_{xa}\cdot A_{yb}=0$ and non-zero elements of $\mathbb{P}$ are units, it follows that $A_{xa}=0$ or $A_{yb}=0$.
Suppose that $A_{xa}=0$.
Let $B'=B\triangle \{x,y,a,b\}$.
Now $B$ and $B'$ are bases of $M$ and $x\in B-B'$, so, by basis exchange, there is some $z\in B'-B=\{a,b\}$ such that $(B-x)\cup z$ is a basis of $M$.
This is a contradiction because $M\del b=M[I|A-b]$, $M\del a=M[I|A-a]$ and $A_{xa}=A_{xb}=0$, so both $(B-x)\cup a$ and $(B-x)\cup b$ are dependent in $M$.
Thus $A_{xa}\neq 0$.
Similarly, since $a\in B'-B$, it follows that $(B'-a)\cup x$ or $(B'-a)\cup y$ is a basis of $M\del a=M[I|A-a]$.
Thus $A_{yb}\neq 0$. %; a contradiction.
We deduce that (i) does not hold.

Therefore (ii) holds. %, so $B'=B\triangle \{x,y,a,b\}$ is dependent in $M$, but $B'$ is a basis of $M[I|A]$.
Since $\det(A[\{x,y,a,b\}])=A_{xa}\cdot A_{yb}\neq 0$, it follows that $A_{xa}\neq 0$ and $A_{yb}\neq 0$. Now $M\del b=M[I|A-b]$ and $A_{xa}\neq 0$, so $(B-x)\cup a$ is a basis of $M$. Similarly, $M\del a=M[I|A-a]$ and $A_{yb}\neq 0$, so $(B-y)\cup b$ is also a basis of $M$. Let $B_1= (B-x)\cup a$ and $B_2=(B-y)\cup b$. Then $x\in B_2-B_1$, so, by basis exchange, there is some $z\in B_1-B_2$ such that $(B_2-x)\cup z$ is a basis of $M$.
But $B_1-B_2=\{a,y\}$, so either $(B-x)\cup b$ or $B'$ is a basis.
In the former case, since $A_{xb}=0$, it follows that $(B-x)\cup b$ is dependent in $M\del a=M[I|A-a]$ and hence in $M$.
Since $B'$ is dependent by assumption, we obtain a contradiction, thus completing the proof.
%This is a contradiction because $B'$ is dependent by assumption while, since $A_{xb}=0$, it follows that $(B-x)\cup b$ is dependent in $M\del a=M[I|A-a]$ and hence in $M$.
\end{proof}

\subsection*{Robust and strong elements}

%We need just a few more definitions before we can state the main theorems.
%In this paper, we will always be trying to keep some minor when removing elements.
%We will frequently fix some basis, and remove elements relative to this basis as we now describe.
When working with a matroid $M$ and a $\mathbb{P}$-representation $A$ of $M$, there is a natural basis $B$ of $M$ that labels the rows of $A$.
We will frequently look to remove elements ``relative to $B$''; that is, in such a way that we obtain a $\mathbb{P}$-representation of the minor of $M$ by removing rows and columns of $A$, without pivoting.  This leads to the following definitions.

Let $M$ be a $3$-connected matroid, let $B$ be a basis of $M$, and let $N$ be a $3$-connected minor of $M$.
Recall that we write $B^*$ to denote $E(M)-B$.
An element $e \in E(M)$ is \textit{$(N,B)$-robust} if either
\begin{enumerate}
 \item[(i)] $e\in B$ and $M/e$ has an $N$-minor, or
 \item[(ii)] $e\in B^*$ and $M\del e$ has an $N$-minor.
\end{enumerate}

Note that an $N$-flexible element of $M$ is clearly $(N,B)$-robust for any basis $B$ of $M$. 

An element $e \in E(M)$ is \textit{$(N,B)$-strong} if either
\begin{enumerate}
 \item[(i)] $e\in B$, and $\si(M/e)$ is $3$-connected and has an $N$-minor; or
 \item[(ii)] $e\in B^*$, and $\co(M\del e)$ is $3$-connected and has an $N$-minor.
\end{enumerate}

\subsection*{Delta-wye exchange}

Let $M$ be a matroid with a triangle $T=\{a,b,c\}$.
Consider a copy of $M(K_4)$ having $T$ as a triangle with $\{a',b',c'\}$ as the complementary triad labelled such that $\{a,b',c'\}$, $\{a',b,c'\}$ and $\{a',b',c\}$ are triangles.
Let $P_{T}(M,M(K_4))$ denote the generalised parallel connection of $M$ with this copy of $M(K_4)$ along the triangle $T$.
Let $M'$ be the matroid $P_{T}(M,M(K_4))\backslash T$ where the elements $a'$, $b'$ and $c'$ are relabelled as $a$, $b$ and $c$ respectively.
The matroid $M'$ is said to be obtained from $M$ by a \emph{\dY\ exchange} on the triangle~$T$, and is denoted $\Delta_T(M)$.
Dually, $M''$ is obtained from $M$ by a \emph{\Yd\ exchange} on the triad $T^*=\{a,b,c\}$ if $(M'')^*$ is obtained from $M^*$ by a \dY\ exchange on $T^*$.  The matroid $M''$ is denoted $\nabla_{T^*}(M)$.
%Dually, letting $M$ be a matroid with a triad $T^*$, we say $\nabla_T(M) = \Delta_{T^*}(M^*)$ is obtained from $M$ by a \emph{\Yd\ exchange} on the triad $T^*=\{a,b,c\}$.

We say that a matroid $M_1$ is \emph{\dY-equivalent} to a matroid $M_0$ if $M_1$ can be obtained from $M_0$ by a sequence of \dY\ and \Yd\ exchanges.

Oxley, Semple, and Vertigan proved that excluded minors for the class of $\mathbb{P}$-representable matroids are closed under \dY\ exchange.
\begin{prop}[{\cite[Theorem~1.1]{oxley2000generalized}}]
  \label{osvdelta}
  Let $\mathbb{P}$ be a partial field, and let $M$ be an excluded minor for the class of $\mathbb{P}$-representable matroids.
  If $M'$ is \dY-equivalent to $M$, then $M'$ is an excluded minor for the class of $\mathbb{P}$-representable matroids.
\end{prop}

\subsection*{Detachable pairs}

Let $M$ be a $3$-connected matroid, and let $N$ be a $3$-connected minor of $M$.
A pair $\{a,b\} \subseteq E(M)$ is {\em $N$-detachable} if either $M\del a,b$ or $M/a,b$ is $3$-connected and has an $N$-minor. 
A $4$-element subset of $E(M)$ is a \emph{quad} if it is a circuit and a cocircuit of $M$.
When $P \subseteq E(M)$ is an exactly $3$-separating set of $M$ with
a partition $\{L_1,\dotsc,L_t\}$ for $t\geq 3$ such that
\begin{itemize}
    \item[(a)] $|L_i|=2$ for each $i\in\{1,\dotsc,t\}$, 
    \item[(b)] $L_i\cup L_j$ is a quad for all distinct $i,j\in\{1,\dotsc,t\}$, and
    \item[(c)] $L_i$ is not contained in a triangle or a triad, for each $i\in\{1,\dotsc,t\}$,
\end{itemize}
then $P$ is a \emph{spike-like $3$-separator} of $M$.

Brettell, Whittle, and Williams \cite{bww1,bww2,bww3} proved that either $M$ has a spike-like $3$-separator, or, after performing at most one \dY\ or \Yd\ exchange on $M$, we obtain a matroid with a detachable pair.
More specifically: %, the following theorem is proved:

\begin{thm}[{\cite[Theorem~1.1]{bww1}}]
  \label{detachthm}
  Let $M$ be a $3$-connected matroid, % with $|E(M)| \ge 14$,
  and let $N$ be a $3$-connected minor of $M$ such that $|E(N)| \ge 4$ and $|E(M)|-|E(N)| \ge 10$.
  Then either
  \begin{itemize}
    \item[(i)] $M$ has an $N$-detachable pair,
    \item[(ii)] there is a matroid $M'$ obtained by performing a single \dY\ or \Yd\ exchange on $M$ such that $M'$ has %an $N$-minor and
      an $N$-detachable pair, or
    \item[(iii)] there is a spike-like $3$-separator~$P$ of $M$ such that at most one element of $E(M)-E(N)$ is not in $P$.
  \end{itemize}
\end{thm}

We note that our definition of a spike-like $3$-separator is more restrictive than that which appears in \cite{bww1}, where condition (c) did not appear.  However, if $M$ has a spike-like $3$-separator for which (c) does not hold, then either (i) or (ii) of \cref{detachthm} holds by \cite[Theorem~3.2]{bww1}.

Now let $\mathbb{P}$ be a partial field, let $N$ be a $3$-connected strong $\mathbb{P}$-stabilizer for the class of $\mathbb{P}$-representable matroids, and let $M$ be an excluded minor for the class of $\mathbb{P}$-representable matroids.
Then $M$ is $3$-connected.
The results in this paper rely on the existence of a pair of elements $\{a,b\}$ such that $M\del a,b$ is $3$-connected with an $N$-minor.
By \cref{detachthm}, we can guarantee such a pair exists, up to dualising and performing at most one \dY\ or \Yd\ exchange, unless $M$ has a spike-like $3$-separator.
We address the possibility of $M$ having a spike-like $3$-separator in \cref{secdetachable}.

\subsection*{The main theorems}

Let $M$ be an excluded minor for the class of $\mathbb{P}$-representable matroids for some partial field $\mathbb P$, and let $N$ be a $3$-connected strong $\mathbb{P}$-stabilizer.
Let $\{a,b\}$ be a pair of elements of $M$ such that $M\del a,b$ is $3$-connected with an $N$-minor.
%deletion pair with respect to $N$.
Our first theorem describes,
in the case that $M \ba a,b$ is not $N$-fragile and $|E(M)| > |E(N)| + 9$,
the local structure of $M\del a,b$ for any such deletion pair $\{a,b\}$.

\begin{thm}
  \label{mainthm1}
  Let $M$ be an excluded minor for the class of $\mathbb{P}$-representable matroids, and let $N$ be a non-binary $3$-connected strong $\mathbb{P}$-stabilizer for the class of $\mathbb{P}$-representable matroids. 
  Suppose $M$ has a pair of elements $\{a,b\}$ such that $M\ba a,b$ is $3$-connected with an $N$-minor.
  Then either
  \begin{itemize}
    \item[(i)] $|E(M)|\leq |E(N)|+9$, or 
    \item[(ii)] $M$ has a $B\times B^{*}$ companion $\mathbb{P}$-matrix $A$ for which $\{x,y,a,b\}$ incriminates $(M,A)$, where $\{x,y\}\subseteq B$ and $\{a,b\}\subseteq B^{*}$, and either
      \begin{itemize} 
        \item[(a)] $M\del a,b$ is $N$-fragile, and $M\del a,b$ has at most one $(N,B)$-robust element $u$ outside of $\{x,y\}$, where if such an element $u$ exists, then $u\in B^{*}-\{a,b\}$ is an $(N,B)$-strong element of $M\del a,b$, and $\{u,x,y\}$ is a coclosed triad of $M\del a,b$, or
        \item[(b)] $M\del a,b$ is not $N$-fragile, but there is an element $u \in B^*-\{a,b\}$ that is $(N,B)$-strong in $M \ba a,b$; either
          \begin{itemize}
            \item[(I)] the $N$-flexible, and $(N,B)$-robust, elements of $M\del a,b$ are contained in $\{u,x,y\}$, or
            \item[(II)] the $N$-flexible, and $(N,B)$-robust, elements of $M\del a,b$ are contained in $\{u,x,y,z\}$, where $z \in B$, and $(z,u,x,y)$ is a maximal fan of $M \del a,b$, or
            \item[(III)]the $N$-flexible, and $(N,B)$-robust, elements of $M\del a,b$ are contained in $\{u,x,y,z,w\}$, where $z \in B$, $w \in B^*$, and $(w,z,x,u,y)$ is a maximal fan of $M \del a,b$;
          \end{itemize}
          the unique triad in $M \ba a,b$ containing $u$ is $\{u,x,y\}$; and $M$ has a cocircuit $\{x,y,u,a,b\}$ and a triangle $\{d,x,y\}$ for some $d\in \{a,b\}$.
      \end{itemize}
      %Moreover, if $M \ba a,b$ has no $(N,B)$-robust elements outside of $\{x,y\}$, then, for any $B_1\times B_1^{*}$ companion $\mathbb{P}$-matrix $A_1$ where $\{x_1,y_1,a,b\}$ incriminates $(M,A_1)$, with $\{x_1,y_1\}\subseteq B_1$ and $\{a,b\}\subseteq B_1^{*}$, the matroid $M \ba a,b$ has no $(N,B_1)$-robust elements outside of $\{x_1,y_1\}$.
  \end{itemize}
\end{thm}

If $M$ is sufficiently larger than $N$,
%If the excluded minor $M$ is sufficiently larger than the stabilizer $N$,
then up to performing at most one \dY\ exchange, we can eliminate case (ii)(b) of \cref{mainthm1} by choosing a different %pair of elements $\{a,b\}$ for which $M\ba a,b$ is $3$-connected with an $N$-minor.
deletion pair.
(Recall that excluded minors for the class of $\mathbb{P}$-representable matroids are closed under \dY\ exchange by \cref{osvdelta}.)
This is the second main theorem of this paper, \cref{mainthm2}.
This theorem implies \cref{one}, but \cref{mainthm2} provides additional information on the existence of $(N_0,B)$-robust elements in $M_0\ba a,b$, and the local structure of $M_0 \ba a,b$ when an $(N_0,B)$-robust element exists.

\begin{thm}
\label{mainthm2}
Let $M$ be an excluded minor for the class of $\mathbb{P}$-representable matroids, and let $N$ be a non-binary $3$-connected strong $\mathbb{P}$-stabilizer, where $M$ has an $N$-minor.
For some $M_1$ that is \dY-equivalent to $M$, and some $(M_0,N_0)$ in $\{(M_1,N),(M_1^{*}, N^{*})\}$, the matroid $M_0$ is an excluded minor with an $N_0$-minor, and at least one of the following holds:
\begin{itemize}
  \item[(i)] $|E(M_0)|\leq |E(N_0)|+9$;
  \item[(ii)] $r(M_0)\leq r(N_0)+7$; or 
  \item[(iii)] there is a pair $\{a,b\} \subseteq E(M)$ such that $M_0\del a,b$ is $3$-connected with an $N_0$-minor, and $M_0\del a,b$ is $N_0$-fragile.  Moreover,
    there is some basis $B$ for $M_0$ and a $B\times B^{*}$ companion $\mathbb{P}$-matrix $A$ for which $\{x,y,a,b\}$ incriminates $(M,A)$, where $\{x,y\}\subseteq B$, $\{a,b\}\subseteq B^{*}$, and %each of the following holds:
    both of the following hold:
    \begin{itemize}
      %\item[(a)] $M_0\del a,b$ is $N_0$-fragile, 
      \item[(a)] $M_0\del a,b$ has at most one $(N_0,B)$-robust element outside of $\{x,y\}$, and
      \item[(b)] if $u$ is an $(N_0,B)$-robust element of $M_0\ba a,b$, then $u\in B^{*}-\{a,b\}$, the element $u$ is $(N_0,B)$-strong in $M_0\del a,b$, and $\{u,x,y\}$ is a triad of $M_0\del a,b$. %;
        %otherwise, when $M_0 \ba a,b$ has no $(N,B)$-robust elements outside of $\{x,y\}$, then, for any $B_1\times B_1^{*}$ companion $\mathbb{P}$-matrix $A_1$ where $\{x_1,y_1,a,b\}$ incriminates $(M_0,A_1)$, with $\{x_1,y_1\}\subseteq B_1$ and $\{a,b\}\subseteq B_1^{*}$, the matroid $M_0 \ba a,b$ has no $(N,B_1)$-robust elements outside of $\{x_1,y_1\}$.
    \end{itemize}
\end{itemize}
\end{thm}

The remainder of the paper is structured as follows.
In \cref{strongelements}, we bound the number of $(N,B)$-strong elements in an excluded minor $M$ with a $3$-connected strong stabilizer $N$ and a basis $B$.
In \cref{gadget}, we bound $|E(M)|$ relative to $|E(N)|$ in the case where %$M\del a,b$ has %a confining set.
the $(N,B)$-strong elements are contained in
a $4$- or $5$-element set with particular properties, which we call a ``confining set''.
In \cref{robustelements} we show that elements %of $M\del a,b$
that are $(N,B)$-robust but not $(N,B)$-strong give rise to a structured collection of $3$-separations, called a ``path of $3$-separations''. %, of $M\del a,b$.
In \cref{robustpath}, we use the structure given by the path of $3$-separations to bound the number of $(N,B)$-robust elements and prove \cref{mainthm1}.
In \cref{secdetachable}, we %analyse the case where the existence of an $N$-detachable pair cannot be guaranteed, showing that $|E(M)|$ is bounded relative to $|E(N)|$ in this case.
show that $|E(M)|$ is bounded relative to $|E(N)|$ in the case where the existence of an $N$-detachable pair cannot be guaranteed.
Finally, in \cref{mainthmsec}, we prove \cref{mainthm2}.

%The next section contains a number of connectivity preliminaries. Section \ref{stablesetup} contains the more preliminaries needed for our setup.

\section{Strong elements}
\label{strongelements}

Let $\mathbb{P}$ be a partial field, and
let $N$ be a $3$-connected strong $\mathbb{P}$-stabilizer for the class of $\mathbb{P}$-representable matroids such that $N$ is non-binary; so, in particular, $|E(N)| \ge 4$.
Suppose $M$ is an excluded minor for the class of $\mathbb{P}$-representable matroids, and $M$ has a pair of elements $\{a,b\}$ such that $M\del a,b$ is $3$-connected with an $N$-minor.
%deletion pair $\{a,b\}$ with respect to $N$.
Let $A$ be a $B\times B^{*}$ companion $\mathbb{P}$-matrix of $M$
%By \cref{companion,incrim4}, $M$ has a $B\times B^{*}$ companion $\mathbb{P}$-matrix~$A$
such that $\{x,y,a,b\}$ incriminates $(M,A)$, where $\{x,y\}\subseteq B$ and $\{a,b\}\subseteq B^{*}$.
Let $M' = M \del a,b$.
We work under these assumptions for the entirety of the section.

%Let $B$ be a basis of $M'$.
Recall that an element $e \in E(M')$ is \textit{$(N,B)$-strong} if either
\begin{enumerate}
 \item[(i)] $e\in B$, and $\si(M'/e)$ is $3$-connected and has an $N$-minor; or
 \item[(ii)] $e\in B^*$, and $\co(M'\del e)$ is $3$-connected and has an $N$-minor.
\end{enumerate}

In this section, we bound the number of $(N,B)$-strong elements of $M'$. % outside of the elements $x$ and $y$. % of the set $\{x,y,a,b\}$ that incriminates $(M,A)$.
The main result is that $M'$ has at most two $(N,B)$-strong elements outside of $\{x,y\}$, and any such elements are in $B^{*}$. %-\{a,b\}$.

%\ldots throughout this section, $N$ is a $3$-connected non-binary matroid, so $|E(N)| \ge 4$.

%\ldots other assumptions? 
   %note that in the cosegment case, must have two elements in $B^{*}-\{a,b\}$ to get the $\{x,y\}$ cospanning property. 

 %This assumption has no effect on the main theorems, since if $M$ has at most 8 elements the the first conclusion clearly holds.

\begin{lemma}
\label{nostrongbasis}
%Let $A$ be a $B\times B^{*}$ companion $\mathbb{P}$-matrix for $M$ such that $\{x,y,a,b\}$ incriminates $(M,A)$, where $\{x,y\}\subseteq B$.
If $u$ is an $(N,B)$-strong element of $M'$ such that $u\notin \{x,y\}$, then $u\notin B$. 
\end{lemma}

\begin{proof}
Suppose that $u$ is an $(N,B)$-strong element of $M'$ such that $u \in B-\{x,y\}$.
%Let $Z=E(M)-\{a,b,u\}$, $Z_1=E(M)-\{b,u\}$, $Z_2=E(M)-\{a,u\}$, and $C=E(N)$.
%Then, by Lemma \ref{notrepcert}, the matroid $M/u$ is not strongly $\mathbb{P}$-stabilized by $N$.
%But $u$ is $(N,B)$-strong, so $M\del a,b/u$, and hence $M/u$, is $3$-connected up to parallel classes.
%Then it follows from Lemma \ref{stableisstable} that $M/u$ is strongly $\mathbb{P}$-stabilized by $N$; a contradiction.   
Since $u$ is $(N,B)$-strong, $M'/u$ is $3$-connected up to parallel classes.
Moreover, as $M \del a$, $M \del b$ and $M$ are $3$-connected, it follows that $M\del a / u$, $M \del b/u$, and $M/u$ are $3$-connected up to parallel classes, and hence are $N$-stable.
As $M\del a / u$ and $M \del b/u$ are $N$-stable, and $M'/u$ is connected, \cref{notrepcert2} implies that $M/u$ is not strongly $\mathbb{P}$-stabilized by $N$.
But, as $M/u$ is $N$-stable, this contradicts \cref{stableisstable}.
\end{proof}

A subset $G$ of $E(M)$ is a \textit{segment} if every $3$-element subset of $G$ is a triangle. A \textit{cosegment} is a segment of $M^{*}$. 

%The next result shows that if $M'$ has an $(N,B)$-strong element outside of $\{x,y\}$, then either $M$ has a confining set or $M'\del u$ is $3$-connected up to series pairs.

\begin{lemma}
 \label{nostronglongline}
 %Let $A$ be a $B\times B^{*}$ companion $\mathbb{P}$-matrix for $M$ such that $\{x,y,a,b\}$ incriminates $(M,A)$, where $\{x,y\}\subseteq B$.
 Suppose $u$ is an $(N,B)$-strong element of $M'$ such that $u\notin \{x,y\}$.
 If $u$ is in a cosegment~$G$ of $M'$ such that $|G|\geq 4$, then $|G|=4$ and $G\cap B=\{x,y\}$.   
\end{lemma}

\begin{proof}
  Let $G$ be a cosegment of $M'$ with $|G| \ge 4$.
  Since $G$ is a corank-$2$ set, $|G\cap B^{*}|\leq 2$.
  Hence $|G\cap B|\geq |G|-2$.
  Since $u$ is $(N,B)$-strong, $u\in B^{*}$ by \cref{nostrongbasis}.
  So $M'\del u$ has an $N$-minor, and hence the elements of the series class $G-u$ of $M'\del u$ are $N$-contractible.
  Suppose that there is some $c\in G$ that is in $B-\{x,y\}$.
  Then %$c$ is $N$-contractible, and
  $M'/c$ is $3$-connected by the dual of \cref{longline3conn}, so $c$ is an $(N,B)$-strong element, contradicting \cref{nostrongbasis}.
We deduce that $|G|=4$ and that $G\cap B=\{x,y\}$.
\end{proof}

%We say that a matroid $Q$ is \textit{$3$-connected up to series classes} if $\co(Q)$ is $3$-connected, and $Q$ is 
%%We say that a matroid $Q$ is
%\textit{$3$-connected up to series pairs} if $\co(Q)$ is $3$-connected and non-trivial series classes of $Q$ have exactly $2$ elements. 
%
The following lemma applies to an $(N,B)$-strong element~$u$ for which $M' \ba u$ is not only $3$-connected up to series classes, but also
$3$-connected up to series pairs.

\begin{lemma}
\label{unstablepair}
Suppose $u\in B^{*}-\{a,b\}$ is an $(N,B)$-strong element of $M'$ such that $M'\del u$ is $3$-connected up to series pairs.
Then at least one of $M\del a,u$ or $M\del b,u$ is not $N$-stable.
\end{lemma}

\begin{proof}
  Towards a contradiction, suppose that both $M\del a,u$ and $M\del b,u$ are $N$-stable.
  Then, as $M \del a,b,u$ is connected,
  \cref{notrepcert2} implies that $M\del u$ is not strongly $\mathbb{P}$-stabilized by $N$.

  We claim that $M \del u$ is $N$-stable.
  Suppose that $M\del a,u$ is not $3$-connected up to series pairs.
  Then, as $M \del a,b,u$ is $3$-connected up to series pairs, and $M \del a$ is $3$-connected, $b$ is in a parallel pair of $\co(M \del a,u)$, which does not exist in $M\del a$.
  Hence, there is a triangle $S \cup b$ of $M$, where $S$ is a series pair of $M \del a,u$.
  Now $S \cup b$ is $2$-separating in $M \del a,u$.
  Since $M \del a,u$ is $N$-stable, the $S \cup b$ component in the $2$-sum decomposition of $M \del a,u$ does not have a $U_{2,4}$-minor.
  It follows that $b$ is in the guts of a $2$-separation $(S,T)$ where $S$ is a series pair of $M \del a,u$.
  We deduce that either $M\del a,u$ is $3$-connected up to series pairs, or $b$ is in the guts of some $2$-separation $(S,T)$ of $M\del a,u$ where $S$ is a series pair of $M\del a,u$.
  By symmetry, either $M\del b,u$ is $3$-connected up to series pairs, or $a$ is in the guts of some $2$-separation $(S',T')$ of $M\del b,u$ where $S'$, say, is a series pair of $M\del b,u$.
  It now follows, by \cref{plusu24}, that $M\del u$ is $N$-stable.

  By \cref{stableisstable}, $M\del u$ is strongly $\mathbb{P}$-stabilized by $N$; a contradiction.
\end{proof}

%FIXME
%Also, restructure.  3.3 is the main result but comes after the necessary definitions ..

%is the point of 3.2 to allow us to focus on series *pairs* only?  doesn't seem to, we could still have a series class of size 3.

Let $M_1$ be a minor of $M$ where, for some $e \in E(M_1)$, the matroid $M_1 \del e$ has an $N$-minor and is $3$-connected up to series classes, but $M_1$ is not $N$-stable.
Recall that, by \cref{plusu24},
%the matroid $Q \del e$ has a series pair $S$, which we call an $N$-unstable series pair, with the property that $Q=R\oplus_2 U_{2,4}$ for some matroid $R$, where $S \cup e \subseteq E(U_{2,4})$. %;
the matroid $M_1$ has an unstable triple $S \cup e$, where $S$ is a series pair of $M_1 \del e$.
%we call $S$
%%We call a non-trivial series class that is contained in a $U_{2,4}$-minor
%an \textit{$N$-unstable series pair} of $Q \del e$.
%an \textit{$N$-unstable series class} of $Q \del e$, and if this series class has size two, we call it an \textit{$N$-unstable series pair} of $Q \del e$.
%We also say that \ldots \emph{$N$-unstable cosegment} of $Q$.

If $M'$ has an $(N,B)$-strong element $u\in B^*-\{a,b\}$ where $M'\del u$ is $3$-connected up to series pairs, then it follows from \cref{unstablepair} that, %$M\del u,a$ or $M\del u,b$ has an $N$-unstable series pair. %Note that we sometimes abuse terminology and say that the $(N,B)$-strong element $u$ has an unstable series class.
up to swapping $a$ and $b$, the matroid $M \del a,u$ has an unstable triple containing $b$.

We now show that %an $N$-unstable series pair must meet $\{x,y\}$. %Moreover, it follows that up to allowable pivots, we can make an unstable series class contain $\{x,y\}$.
  the intersection of an unstable triple with $B$ is a non-empty subset of $\{x,y\}$.
%an unstable triple meets $\{x,y\}$. %is this quite right?

\begin{lemma}
\label{unstablemeetsxy}
%Let $A$ be a $B\times B^{*}$ companion $\mathbb{P}$-matrix for $M$ such that $\{x,y,a,b\}$ incriminates $(M,A)$, where $\{x,y\}\subseteq B$.
Suppose $u \in B^*-\{a,b\}$ is an $(N,B)$-strong element of $M'$ such that $M' \del u$ is $3$-connected up to series pairs.
%Then, there exists a basis~$B_1$ for $M$, and a $B_1\times B_1^{*}$ companion $\mathbb{P}$-matrix $A_1$ for $M$ such that, for some $\{x_1,y_1\}\subseteq B_1$,
%\begin{itemize}
  %\item[(i)] $M'\ba u$ has a series pair~$S$ such that $\emptyset \subsetneqq S\cap B_1\subseteq \{x_1,y_1\}$, and $S \cup u$ is an unstable triple of either $M \del u,a$ or $M \del u,b$, and
  %\item[(ii)] $\{x_1,y_1,a,b\}$ incriminates $(M,A_1)$, where $\{a,b\} \subseteq B_1^*$.
%\end{itemize}
Then $M'\ba u$ has a series pair~$S$ such that $\emptyset \subsetneqq S\cap B\subseteq \{x,y\}$.  Moreover,
%either $S \cup b$ is an unstable triple of $M \del u,a$, or $S \cup a$ is an unstable triple of $M \del u,b$.
$S \cup b$ is an unstable triple of $M \del a,u$, up to swapping $a$ and $b$.
\end{lemma}

\begin{proof}
  By \cref{unstablepair}, either $M\ba a,u$ or $M \ba b,u$ is not $N$-stable.
  Without loss of generality, we may assume that $M\ba b,u$ is not $N$-stable.
  Then, by \cref{plusu24}, there is a pair~$S$ such that $S \cup a$ is an unstable triple in $M\del b,u$.
  Let $S=\{s_1,s_2\}$.
  Note that, since $S$ is a series pair of $M'\ba u$, both $s_1$ and $s_2$ are $N$-contractible in $M'$.
  We also note that $S\cap B$ is non-empty because, in $M' \ba u$, the pair~$S$ is codependent and $B^*-\{a,b,u\}$ is a cobasis.

  Towards a contradiction,
  suppose that $s_1\in B-\{x,y\}$. Then $s_1$ is not $(N,B)$-strong by \cref{nostrongbasis}, so $\si(M'/s_1)$ is not $3$-connected.
  Hence $\si(M'/s_2)$ is $3$-connected by \cref{seriesnotbasisstrong}, so it follows from \cref{nostrongbasis} that either $s_2\in \{x,y\}$ or $s_2\in B^{*}-\{a,b\}$. 

\begin{claim}
\label{unstablemeetsxyclaim}
Up to an allowable pivot, we can assume that $s_2\in \{x,y\}$. 
\end{claim}

\begin{subproof}
  Observe that since $S$ is a series pair of $M' \ba u$ but $M'$ is $3$-connected, $S \cup u$ is a triad in $M'$.
  Suppose that $s_2\in B^{*}-\{a,b\}$.
  Then $A_{s_1s_2}\neq 0$ because $\{s_1,s_2,u\}$ is a triad of $M'$.
  If $A_{xs_2}=A_{ys_2}=0$, then a pivot on $A_{s_1s_2}$ is allowable, and $s_2$ is an $(N,B\triangle \{s_1, s_2\})$-strong element with $s_2\in (B\triangle \{s_1, s_2\})-\{x,y\}$, which contradicts \cref{nostrongbasis}.
  Thus we shall assume that $A_{xs_2}\neq 0$.
  Then a pivot on $A_{xs_2}$ is an allowable pivot, and $s_2$ takes the place of $x$ as a member of the set $\{s_2,y,a,b\}$ that incriminates $(M,A^{xs_2})$. 
\end{subproof}

By \ref{unstablemeetsxyclaim} we may assume that $s_2=x$. Since $\{a,s_1, s_2\}$ is an unstable triple of $M \del b,u$, it follows that $a\in \cl_{M}(\{s_1,s_2\})$ where $\{s_1,s_2\}\subseteq B$. Hence $A_{ja}\neq 0$ if and only if $j\in \{s_1,s_2\}$. But then $A_{ya}=0$, contradicting that the bad submatrix has no zero entries. This contradiction arose from the assumption that some member of $S\cap B$ was outside of $\{x,y\}$. Therefore $S\cap B\subseteq \{x,y\}$. 
\end{proof}

%Next we prove that unstable series classes involving the same element (i.e. $a$ or $b$) are simple to handle: they must be contained in a $4$-point cosegment of $M'$.

%Let $u$ and $v$ be $(N,B)$-strong elements outside of $\{x,y\}$. If both $M\del u,b$ and $M\del v,b$ are not $N$-stable, then $\{u,v,x,y\}$ is a cosegment of $M'$.

\begin{lemma}
\label{noweirdtriads}
Let $u$ and $v$ be distinct $(N,B)$-strong elements of $M'$ outside of $\{x,y\}$ such that both $M'\del u$ and $M'\del v$ are $3$-connected up to series pairs.
%If $M\del a,u$ is not $N$-stable, then $M\del a,v$ is $N$-stable. 
Then at least one of $M\del a,u$ or $M\del a,v$ is $N$-stable. 
\end{lemma}

\begin{proof}
Suppose that both $M\del a,u$ and $M\del a,v$ are not $N$-stable.
By \cref{plusu24}, there is a series pair~$S_u$ of $M'\del u$ such that $S_u \cup b$ is an unstable triple of $M \del a,u$, and there is a series pair~$S_v$ of $M'\del v$ such that $S_v \cup b$ is an unstable triple of $M \del a,v$.

First, suppose that $S_u\cap S_v=\emptyset$.
If $u \notin S_v$, then $S_v\subseteq E(M \ba a)-(S_u \cup u)$, so $b\in \cl_{M \del a}(S_v) \subseteq \cl_{M \del a}(E(M \ba a)-(S_u \cup u))$,
implying $b \notin \cl^*_{M \del a}(S_u \cup u)$.
But $S_u \cup b$ is an unstable triple of $M\del a,u$, so $b \in \cocl_{M \del a,u}(S_u)=\cl^*_{M \del a}(S_u \cup u)$; a contradiction.
We deduce that $u \in S_v$ and, by symmetry, $v \in S_u$.
Now, as $S_u \cup u$ and $S_v \cup v$ are triads of $M'$, the set $S_u \cup S_v$ is a $4$-element cosegment of $M'$ that contains $\{u,v\}$.
This contradicts that %$M' \del u$ is $3$-connected up to series pairs.
$S_u$ is a series pair of $M' \del u$.

Next, suppose that $|S_u\cap S_v|=1$.
Then, as %by \cref{plusu24},
$S_u\cup b$ and $S_v\cup b$ are triangles of $M \ba a$ and $M$, %so
it follows that $S_u\cup S_v$ is a triangle of $M'$, and so $\{u,v\}\cup S_u\cup S_v$ is a $5$-element fan of $M'$ with rim ends $u$ and $v$.
But then $\co(M'\del v)$ is not $3$-connected by \cref{fanends}; a contradiction. %, because $v$ is an $(N,B)$-strong element of $M'$.
Therefore $S_u=S_v$.
But now $\{u,v\}\cup S_u$ is a $4$-element cosegment of $M'$, contradicting that %$M'\del u$ is $3$-connected up to series pairs.
$S_u$ is a series pair of $M'\del u$.
We deduce that either $M\del a,u$ or $M\del a,v$ is $N$-stable.
\end{proof}

%Suppose that both $M\del u,b$ and $M\del v,b$ are not $N$-stable, and let $S_u$ and $S_v$ be unstable series pairs for $u$ and $v$ respectively. %Suppose that $S_u\cup u$ is a cosegment of $M'$. Then $S_u\cup u=\{u,t,x,y\}$ for some $t\in B^{*}$ by Lemma \ref{nostronglongline}. We claim that $t=v$. Suppose not. Then we may assume by Lemma \ref{unstablemeetsxy} that $S_v=\{x,s\}$ for some $s\in B^{*}-\{a,b\}$. But then $S_u\cup S_v\cup\{u,v\}$ is a corank-$3$ subset with $\{u,v,s,t\}$ contained in the cobasis $B^{*}$; a contradiction. Thus $S_v=\{x,y\}$, so $t=v$ by Lemma \ref{nostronglongline}. 

%Then $S_u\cup u$ and $S_v\cup v$ are triads. Suppose $S_u\cap S_v=\emptyset$. Then $S_v\subseteq E(M)-S_u$, so $a\in \cl(E(M)-S_u)$; a contradiction because $a\notin \cl(E(M)-S_u)$. Thus $S_u\cap S_v\neq \emptyset$. Suppose that $|S_u\cap S_v|=1$. Then, since $S_u\cup a$ and $S_v\cup a$ are triangles, it follows that $S_u\cup S_v$ is a triangle of $M'$, and so $\{u,v\}\cup S_u\cup S_v$ is a $5$-element fan of $M'$ with rim ends $u,v$. But then $\co(M'\del v)$ is not $3$-connected; a contradiction because $v$ is an $(N,B)$-strong element. Therefore $S_u=S_v$. But then $\{u,v\}\cup S_u$ is a $4$-point cosegment, so $S_u\cup \{u,v\}=\{u,v,x,y\}$ by Lemma \ref{nostronglongline}.

\begin{lemma}
\label{cosegstrongbound}
If $M'$ has a $4$-element cosegment $G$ such that $G\cap B=\{x,y\}$, then $M'$ has no $(N,B)$-strong elements outside of $G$. 
\end{lemma}

\begin{proof}
Towards a contradiction, suppose that $M'$ has an $(N,B)$-strong element $v$ outside of $G$.
By \cref{nostrongbasis}, $v \in B^*$.
%Then it follows from \cref{nostronglongline} that
By \cref{nostronglongline}, if $v$ is in a $4$-element cosegment~$G'$ of $M'$, then $G \cup G'$ is a cosegment consisting of more than four elements; a contradiction.
So $M'\del v$ is $3$-connected up to series pairs.
Hence $M'\del v$ has a series pair~$S$ such that % $S$ is contained in an unstable triple of $M' \del a,v$ or $M' \del b,v$, by \cref{unstablepair}, and $S$
$\emptyset \subsetneqq S \cap B \subseteq \{x,y\}$, by \cref{unstablemeetsxy}.
Now, in $M'$, the triad $S \cup v$ meets the cosegment $G$, so $r^*_{M'}(G \cup S \cup v) \le 3$.
It follows that $|(G \cup S \cup v) \cap B^*| =3$.
Thus $S \cup v$ intersects $G$ in two elements, implying $r^*(G \cup S \cup v)=2$; a contradiction.
\end{proof}

These results are enough to bound the number of $(N,B)$-strong elements outside of $\{x,y\}$. The bound on the number of $(N,B)$-strong elements is a key ingredient in many subsequent arguments.

\begin{prop}
\label{atmost2outxy}
$M'$ has at most two $(N,B)$-strong elements outside of $\{x,y\}$. 
\end{prop}

\begin{proof}
  Let $u$ be an $(N,B)$-strong element of $M'$ outside of $\{x,y\}$.  By \cref{nostrongbasis}, $u \in B^*$.
  Suppose that $M'\del u$ has a series class %$S$ with $|S|\geq 3$.
  of size at least three.
  Then $M'$ has a $4$-element cosegment $G$ such that $\{u,x,y\}\subseteq G$ and $G\cap B=\{x,y\}$, by \cref{nostronglongline}.
  Thus, by \cref{cosegstrongbound}, $M'$ has at most two $(N,B)$-strong elements outside of $\{x,y\}$.

We may now assume that $M'\del u$ is $3$-connected up to series pairs for each $(N,B)$-strong element~$u$ of $M'$ outside of $\{x,y\}$.
Suppose there exist distinct $(N,B)$-strong elements $u,v_1,v_2 \in B^*$ such that $M'\del u$, $M' \del v_1$, and $M \del v_2$ are $3$-connected up to series pairs.
By \cref{unstablepair}, we may assume without loss of generality that $M\ba b,u$ is not $N$-stable.  Now, by \cref{noweirdtriads}, both $M\ba b,v_1$ and $M\ba b,v_2$ are $N$-stable.
By two further applications of \cref{unstablepair}, both $M\ba a,v_1$ and $M\ba a,v_2$ are not $N$-stable.
But this contradicts \cref{noweirdtriads}.
\end{proof}

\section{Confining sets}
\label{gadget}

In this section, we work under the following setup.
Let $\mathbb{P}$ be a partial field, and
let $M$ and $N$ be matroids, where $N$ is a non-binary $3$-connected strong $\mathbb{P}$-stabilizer for the class of $\mathbb{P}$-representable matroids, and $M$ is an excluded minor for the class of $\mathbb{P}$-representable matroids with a pair of elements $\{a,b\}$ such that $M\del a,b$ is $3$-connected with an $N$-minor.
%deletion pair $\{a,b\}$ with respect to $N$.
Let $M' = M \del a,b$.

We say that a subset $G$ of $E(M')$ is a \textit{confining set} if $G\cap B_1=\{x_1,y_1\}$ for some basis $B_1$ of $M'$, and either
\begin{itemize}
  \item[(a)] $G$ is a $4$-element cosegment, or
  \item[(b)] $G$ is the union of two triads $T$ and $T'$ with $|T \cap T'|=1$, where %$G\subseteq \cl^{*}(G\cap B^{*})$, and
    $G\cap B_1^{*}$ has at least one $(N,B_1)$-strong element,
\end{itemize}
where $x_1$ and $y_1$ are elements of $B_1$ such that $\{x_1,y_1,a,b\}$ incriminates $(M,A_1)$ for some $B_1\times B_1^{*}$ companion $\mathbb{P}$-matrix~$A_1$ of $M$.
In this case, we also say $G$ is a \emph{confining set relative to $B_1$}.
Note that a confining set satisfying (b) has corank $3$ in $M'$.
Every confining set $G$ relative to a basis $B_1$ has the property that $G \cap B_1^*$ cospans $G$, since $|G \cap B_1^*| = |G|-2 = r^*_{M'}(G)$.

We first show that $M'$ either has a confining set, or at most one $(N,B)$-strong element outside of $\{x,y\}$ for some basis $B$ of $M$ such that $\{x,y,a,b\}$ incriminates $(M,A)$ where $A$ is a $B\times B^{*}$ companion $\mathbb{P}$-matrix of $M$ with $\{x,y\}\subseteq B$ and $\{a,b\}\subseteq B^{*}$.
We then prove the main result of this section: if $M'$ has a confining set, then $|E(M)|$ is bounded relative to $|E(N)|$.

\begin{prop}
  \label{nogadgetsetup}
  Suppose $M'$ does not have a confining set. Then there is some basis~$B_0$ of $M'$, and $B_0\times B_0^{*}$ companion $\mathbb{P}$-matrix $A_0$ of $M$ such that $\{x_0,y_0,a,b\}$ incriminates $(M,A_0)$, for some $\{x_0,y_0\}\subseteq B_0$, and either
  \begin{itemize}
    \item [(i)] $M'$ has exactly one $(N,B_0)$-strong element $u$ outside of $\{x_0,y_0\}$, and $\{u,x_0,y_0\}$ is a triad of $M'$; or
    \item [(ii)] $M'$ has no $(N,B_0)$-strong elements outside of $\{x_0,y_0\}$ for every choice of basis $B_0$ %.
      with a $B_0\times B_0^{*}$ companion $\mathbb{P}$-matrix $A_0$ of $M$ such that $\{x_0,y_0,a,b\}$ incriminates $(M,A_0)$, for some $\{x_0,y_0\}\subseteq B_0$.
  \end{itemize}
\end{prop}

\begin{proof}
  We first prove the following claim.

\begin{claim}
  \label{spstep}
  Let $B_1$ be a basis of $M'$, and let $A_1$ be a $B_1\times B_1^{*}$ companion $\mathbb{P}$-matrix of $M$ such that $\{x_1,y_1,a,b\}$ incriminates $(M,A_1)$, for some $\{x_1,y_1\}\subseteq B_1$.
  If $u$ is an $(N,B_1)$-strong element of $M'$ outside of $\{x_1,y_1\}$, then $M'\ba u$ is $3$-connected up to series pairs.
\end{claim}
\begin{subproof}
By \cref{nostrongbasis}, $u \in B_1^*$.  If $u$ is in a cosegment~$G$ consisting of at least four elements, then, by \cref{nostronglongline}, $G$ is a confining set of $M'$; a contradiction.
So we may assume that $M'\del u$ is $3$-connected up to series pairs for each $(N,B_1)$-strong element~$u$ of $M'$ outside of $\{x_1,y_1\}$. 
\end{subproof}

If, for every choice of basis $B_1$, with corresponding incriminating set $\{x_1,y_1,a,b\}$, the matroid $M'$ has no $(N,B_1)$-strong elements outside of $\{x_1,y_1\}$, then clearly the proposition holds.
So let $B_1$ be a basis of $M'$ such that $u$ is an $(N,B_1)$-strong element of $M'$ outside of $\{x_1,y_1\}$.

\begin{claim}
  \label{cl1}
Either the proposition holds, or there is a $B_2 \times B_2^*$ companion $\mathbb{P}$-matrix $A_2$ such that $\{x_2,y_2,a,b\}$ incriminates $(M,A_2)$ for some $\{x_2,y_2\} \subseteq B_2$, and $M'$ has exactly two $(N,B_2)$-strong elements outside of $\{x_2,y_2\}$.
\end{claim}
\begin{subproof}
  By \cref{atmost2outxy}, $M'$ has at most two $(N,B_1)$-strong elements outside of $\{x_1,y_1\}$. 
  Thus if $M'$ has two $(N,B_1)$-strong elements outside of $\{x_1,y_1\}$, then \cref{cl1} holds with $B_2 = B_1$.
So suppose that $u$ is the only $(N,B_1)$-strong element of $M'$ outside of $\{x_1,y_1\}$.
Then $M'\ba u$ is $3$-connected up to series pairs, by \cref{spstep}.
By \cref{unstablepair} we may assume, up to swapping $a$ and $b$, that $M \del a,u$ is not $N$-stable, $S_u$ is a series pair of $M' \del u$, and $S_u \cup b$ is an unstable triple of $M\del a,u$.
Since $M'$ is $3$-connected, $S_u \cup u$ is a triad of $M'$.
Thus, if $S_u=\{x_1,y_1\}$, then $\{u,x_1,y_1\}$ is a triad, so the proposition holds in this case.
Assume that $S_u\neq \{x_1,y_1\}$.
Then it follows from \cref{unstablemeetsxy} that, without loss of generality, $S_u=\{x_1,s\}$ for some $s\in B_1^{*}-\{a,b,u\}$.
Now $b$ is spanned by $S_u$ in $M$, and $A_{yb}\neq 0$ because the bad submatrix has no zero entries, so it follows that $A_{ys}\neq 0$.
Hence a pivot on $A_{ys}$ is allowable.
So $\{x_1,s,a,b\}$ incriminates $(M,A_1^{ys})$.
Let $B_2=B_1\triangle \{y_1,s\}$.
If $y_1$ is not $(N,B_2)$-strong, then the proposition holds, since $\{u,x_1,s\}$ is a triad.
Otherwise, $u$ and $y_1$ are distinct $(N,B_2)$-strong elements outside of $\{x_1,s\}$, satisfying \cref{cl1}.
\end{subproof}

By \cref{cl1}, we may now assume that $B_2$ is a basis for $M'$, the matrix $A_2$ is a $B_2 \times B_2^*$ companion $\mathbb{P}$-matrix where $\{x_2,y_2,a,b\}$ incriminates $(M,A_2)$ for some $\{x_2,y_2\} \subseteq B_2$, and $M'$ has exactly two $(N,B_2)$-strong elements, $u$ and $v$, in $B_2^*$.
By \cref{spstep}, $M'\ba u$ and $M'\ba v$ are $3$-connected up to series pairs.
We may assume, up to swapping $a$ and $b$, that $M\del a,u$ and $M\del b,v$ are not $N$-stable, but that $M\del b,u$ and $M\del a,v$ are $N$-stable, by \cref{noweirdtriads,unstablepair}.
Let $S_u$ be a series pair of $M' \del u$, and let $S_v$ be a series pair of $M' \del v$, where $S_u \cup b$ is an unstable triple of $M\del a,u$, and $S_v \cup a$ is an unstable triple of $M\del b,v$.
Next, we show that $S_u \cup u = S_v \cup v$.
In fact, we prove a more general claim that we can apply even after an allowable pivot.

\begin{claim}
  \label{foregoing}
  Let $B_3$ be a basis of $M'$ such that $\{x_3,y_3,a,b\}$ incriminates $(M,A_3)$, for some $B_3\times B_3^{*}$ companion $\mathbb{P}$-matrix $A_3$ of $M$ and $\{x_3,y_3\}\subseteq B_3$.
  Suppose $M'$ has exactly two $(N,B_3)$-strong elements $u,v \in B_3^*$,
  where $M \ba a,u$ and $M \ba b,v$ are not $N$-stable.
  Let $S_u$ and $S_v$ be pairs
  such that $S_u \cup b$ and $S_v \cup a$ are unstable triples of $M \ba a,u$ and $M \ba b,v$, respectively.
  Then $S_u\cup u = S_v\cup v$.
\end{claim}

\begin{subproof}
Suppose that the triads $S_u\cup u$ and $S_v\cup v$ of $M'$ are disjoint.
Then, by \cref{unstablemeetsxy}, we may assume that $S_u=\{s,x_3\}$ and $S_v=\{t,y_3\}$ for some $s,t\in B_3^{*}-\{a,b,u,v\}$. Then $A_{y_3s}=0$ because $S_v\cup v$ is a triad of $M'$. But then $s$ is spanned by $B_3-y_3$.  Since $S_u \cup b$ is a triangle, it follows that $b$ is spanned by $B_3-y_3$. Then $A_{y_3b}=0$; a contradiction because the bad submatrix has no zero entries.

Let $G = S_u\cup S_v\cup \{u,v\}$.
Suppose that $|(S_u\cup u)\cap(S_v\cup v)|=1$.
Then $G$ has corank $3$ in $M'$, so $|G \cap B_3^*| \le 3$. It now follows from \cref{unstablemeetsxy} that $G\cap B_3=\{x_3,y_3\}$, so $G$ is a confining set of $M'$; a contradiction.
If $|(S_u\cup u)\cap(S_v\cup v)|=2$, then $G$ is a $4$-element corank-$2$ subset of $M'$, and it follows from \cref{nostronglongline} that $G$ is a confining set; a contradiction.
So $S_u\cup u = S_v\cup v$, completing the proof of \cref{foregoing}.
\end{subproof}

By \cref{foregoing} we have that $S_u\cup u = S_v\cup v$.
Then, by \cref{unstablemeetsxy}, we may assume that $S_u = \{v,x_2\}$ and $S_v = \{u,x_2\}$.
Since $b$ is spanned by $S_u = \{v,x_2\}$, and $A_{y_2b}\neq 0$ because the bad submatrix has no zero entries, it follows that $A_{y_2v}\neq 0$.
Hence a pivot on $A_{y_2v}$ is allowable. 
Now $\{x_0,y_0,a,b\}$ incriminates $(M,A^{y_2v})$, where $x_0=x_2$ and $y_0=v$.
Let $B_0=B_2\triangle \{y_2,v\}$.
Then $u$ is an $(N,B_0)$-strong element outside of $\{x_0,y_0\}$, and $\{u,x_0,y_0\}$ is a triad.
If $y_2$ is not $(N,B_0)$-strong, then the proposition holds.
So suppose that $y_2$ is an $(N,B_0)$-strong element of $M'$.
By \cref{spstep}, $M'\ba y_2$ is $3$-connected up to series pairs.
By \cref{unstablemeetsxy}, $M'\ba y_2$ has a series pair $S_{y_2}$, and $S_{y_2} \cup y_2$ is a triad in $M'$.
By \cref{foregoing}, $S_u \cup u = S_{y_2} \cup y_2$.
But $y_2 \notin S_u \cup u$, so this is contradictory.
\end{proof}

Later, we refer to a basis $B_0$ satisfying \cref{nogadgetsetup} as a \emph{strengthened basis}.

In the remainder of this section we show that if $M'=M \ba a,b$ has a confining set, then $|E(M)|\leq |E(N)| +9$.
Let $B$ be a basis of $M$, and let $A$ be a $B\times B^{*}$ companion $\mathbb{P}$-matrix of $M$ such that $\{x,y,a,b\}$ incriminates $(M,A)$, where $\{x,y\}\subseteq B$ and $\{a,b\}\subseteq B^{*}$.

We begin with the following constraint on the strong elements of $M'$.

\begin{lemma}
\label{nostrongoutsidegadget}
If $M'$ has a confining set $G$ relative to the basis $B$, then $M'$ has no $(N,B)$-strong elements outside of $G$.   
\end{lemma}

\begin{proof}
  Suppose $M'$ has a confining set~$G$ relative to the basis $B$.
If $G$ is a $4$-element cosegment, then it follows from \cref{cosegstrongbound} that $M'$ has no $(N,B)$-strong elements outside of $G$. 

Assume now that $G=\{u,v,w,x,y\}$ has corank $3$ in $M'$. 
By the definition of a confining set, $\{u,v,w\}\subseteq B^{*}$, and $\{u,v,w\}$ contains an $(N,B)$-strong element.
Suppose $t$ is an $(N,B)$-strong element outside of $G$. Then $t\in B^{*}$ by \cref{nostrongbasis}.
Now $M' \ba t$ is $3$-connected up to series classes;
we next show that $M'\del t$ is in fact $3$-connected up to series pairs.

Suppose that $M'$ has a cosegment~$G'$ containing $t$ with $|G'|\geq 4$. Then $G'=\{s,t,x,y\}$ for some $s\in B^{*}$ by \cref{nostronglongline}. If $s\notin \{u,v,w\}$, then, as there is an $(N,B)$-strong element in $\{u,v,w\}$, there is some $(N,B)$-strong element of $M'$ outside of the $4$-element cosegment~$G'$, contradicting \cref{cosegstrongbound}.
Thus we may assume, without loss of generality, that $G'=\{u,t,x,y\}$. Then $G\cup G'$ has corank $3$; a contradiction, because $G\cup G'$ has a four-element subset $\{t,u,v,w\}$ contained in $B^{*}$. Therefore $M'\del t$ is $3$-connected up to series pairs.

By \cref{unstablemeetsxy}, $M'\del t$ has a series pair $S_t$ that meets $\{x,y\}$.
Then $T^*=S_t \cup t$ is a triad of $M'$; without loss of generality, we may assume that $S_t = \{x,z\}$ for some $z \in E(M')-\{x,t\}$.
If $z\in G$, then $G\cup T^*$ has corank at most $3$ but contains a $4$-element subset $\{t,u,v,w\}$ of $B^{*}$; a contradiction. Thus $z\notin G$, so $z\in B^{*}$ by \cref{unstablemeetsxy}. Then $G\cup T^*$ has corank $4$ but contains a $5$-element subset $\{t,u,v,w,z\}$ of $B^{*}$; a contradiction.  
\end{proof}

The following results consider allowable pivots when $M'$ has a confining set.
The routine proof of the first lemma is omitted.  The second is a straightforward consequence of the first, using the fact that $\{x,y\} \subseteq \cocl_{M'}(G\cap B^*)$.

\begin{lemma}
\label{rowsinagadget}
Let $G$ be a confining set relative to the basis $B$, and let $p\in B$.
Then $A_{pq}= 0$ for all $q\in (B^{*}-\{a,b\})-G$ if and only if $p\in \cl^{*}_{M'}(G\cap B^{*})$. 
\end{lemma}

\begin{lemma}
\label{gadgetpivots}
Let $G$ be a confining set relative to the basis $B$. Then $A_{xq}=0$ and $A_{yq}=0$ for all $q\in (B^{*}-\{a,b\})-G$.
\end{lemma}

%\begin{proof}
%Since $\{x,y\}\subseteq \cl^{*}(G\cap B^{*})$, it follows that $A_{xz}=A_{yz}=0$ for all $z\in B^{*}-G$.
%\end{proof}

\begin{lemma}
\label{gadgetpivots2}
Let $G$ be a confining set relative to the basis $B$.
If $A_{pq}\neq 0$ for some $p\in B-\{x,y\}$ and $q\in (B^{*}-\{a,b\})-G$, then a pivot on $A_{pq}$ is allowable. Moreover, $G$ is a confining set relative to the basis $B\triangle \{p,q\}$.
\end{lemma}

\begin{proof}
  Suppose that $A_{pq}\neq 0$ for some $p\in B-\{x,y\}$ and $q\in (B^{*}-\{a,b\})-G$. Then the pivot on $A_{pq}$ is allowable by \cref{gadgetpivots}. Since $G\cap B=G\cap (B\triangle\{p,q\})$ and $G\cap B^{*}=G\cap (B^{*}\triangle\{p,q\})$, $G$ is a confining set relative to $B \triangle\{p,q\}$.
\end{proof}

Due to the existence of these allowable pivots when $M'$ has a confining set, the following restrictions are imposed on elements of $M'$.

\begin{lemma}
\label{noextrastrong}
Let $G$ be a confining set relative to the basis $B$.
For every $z\in E(M')$,
\begin{itemize}
 \item[(i)] if $z$ is $N$-contractible and $\si(M'/z)$ is $3$-connected, then $z \in G$; and
 \item[(ii)] if $z$ is $N$-deletable and $\co(M'\del z)$ is $3$-connected, then $z\in %\cl^{*}(G\cap B^{*})$.
   \cl^{*}_{M'}(G)$.
\end{itemize}
\end{lemma}

\begin{proof}
  Suppose there is an element $z \in E(M')-G$ that is $N$-contractible, and $\si(M'/z)$ is $3$-connected.
  Since $z\notin G$ it follows from \cref{nostrongoutsidegadget} that $z\in B^{*}-G$.
  Then $A_{xz}=A_{yz}=0$ by \cref{gadgetpivots}, so there is some $p\in B-\{x,y\}$ such that $A_{pz}\neq 0$ because $M'$ has no loops.
  Let $B'=B\triangle\{p,z\}$.
  Now, a pivot on $A_{pz}$ is allowable by \cref{gadgetpivots2}.
  So $M'$ has an $(N,B')$-strong element $z$ in $B'-\{x,y\}$; a contradiction of \cref{nostrongbasis}. This proves (i).

  Now suppose there is an element $z \in E(M')-\cocl(G)$ % \cap B^*)$
  that is $N$-deletable, and $\co(M'\ba z)$ is $3$-connected.
  Then $z \notin G$, so $z\in B-\{x,y\}$ by \cref{nostrongoutsidegadget}.
  It follows from \cref{rowsinagadget} that there is some $q\in (B^{*}-\{a,b\})-G$ such that $A_{zq}\neq 0$.
  Let $B'=B\triangle\{z,q\}$.
  Now, a pivot on $A_{zq}$ is allowable by \cref{gadgetpivots2}, and $G$ is a confining set relative to $B'$.
  But $z$ is an $(N,B')$-strong element outside of $G$; a contradiction of \cref{nostrongoutsidegadget}. This proves (ii).
\end{proof}

When $C$ and $D$ are disjoint subsets of $E(M')$ such that $M'/C \ba D \cong N$, we say $(C,D)$ is an \emph{$N$-labelling of $M'$}. 
For the remainder of the section, suppose $M'$ has a confining set $G$, and let $(C,D)$ be an $N$-labelling of $M'$. %, so $N\cong M'/C\del D$.
Recall that if $G$ has corank three, then there is a $(N,B)$-strong element $u \in G \cap B^*$.
In this case, we choose an $N$-labelling $(C,D)$ such that $u \in D$.
Having fixed $(C,D)$, our goal is to bound the size of $C \cup D$, and thus bound $|E(M)|-|E(N)|$.

We write $r^*(X)$ instead of $r^*_{M'}(X)$, and $\cocl(X)$ instead of $\cocl_{M'}(X)$, for the remainder of the section.

\begin{lemma}
\label{cospanningG}
Suppose that the confining set $G$ has corank $3$ in $M'$.
If $z',z''\in C\cup D$ are in $\cl^{*}(G)-G$, then for every partition $(X,Y)$ of $G\cup \{z',z''\}$, either $r^{*}(X) \ge 3$ or $r^{*}(Y) \ge 3$.
\end{lemma}

\begin{proof}
  Suppose that $(X,Y)$ is a partition of $G\cup \{z',z''\}$ such that $\max\{r^{*}(X),r^{*}(Y)\}\leq 2$. We claim that either $z'$ or $z''$ is an element that contradicts \cref{noextrastrong}(i). Since $|G\cup \{z',z''\}|=7$, we may assume that $|X|\geq 4$. Then $X$ is a cosegment with at least four elements that contains at least one element $z\in \{z',z''\}$, so $\si(M/z)$ is $3$-connected by the dual of \cref{longline3conn}. Hence $z$ is not $N$-contractible by \cref{noextrastrong}(i), so $z\in D$. 

First suppose that $z',z''\in X$. Then $z',z''\in D$, but $z'$ is in a series class $X\cup z'$ of $M'\del z''$, so $z'$ is $N$-contractible in $M'\del z''$ and hence in $M'$; a contradiction of \cref{noextrastrong}(i).

We may now assume that $z'\in X$ and $z''\in Y$, so $X$ and $\cl^{*}(Y)$ are both $4$-element cosegments. Hence both $\si(M'/z')$ and $\si(M'/z'')$ are $3$-connected by the dual of \cref{longline3conn}. By the definition of a confining set, there is some element $u\in G-\{x,y\}$ that is $(N,B)$-strong in $M'$, and $u$ belongs to either $X$ or $Y$. Hence, in $M'\del u$, either $z'$ or $z''$ is in a non-trivial series class, so at least one of $z'$ and $z''$ is $N$-contractible in $M'$; a contradiction of \cref{noextrastrong}(i).
\end{proof}

\begin{lemma}
\label{atmostonedelcospan}
There are at most two elements of $D$ that belong to $\cl^{*}(G)-G$. 
%at most 1 if $G$ is a cosegment.
\end{lemma}

\begin{proof}
Suppose that there are distinct elements $z,z',z''\in (\cl^{*}(G)-G)\cap D$.
Then $z,z',z''\in B-\{x,y\}$, since $G \cap B^*$ is a basis for $\cocl(G)$.

If $G$ is a $4$-element cosegment of $M'$, then $\cl^*(G)$ is a cosegment containing $z$ and $z'$.
Since $z'$ is $N$-deletable, $z$ is in a non-trivial series class of $M'\del z'$, and $|E(N)| \ge 4$, the element~$z$ is $N$-contractible in $M'$.
By the dual of \cref{longline3conn}, $M'/z$ is $3$-connected, so $z'$ is an $(N,B)$-strong element of $B-\{x,y\}$; a contradiction of \cref{nostrongbasis}. 

Now we may assume that $\cl^{*}(G)$ has corank $3$ in $M'$.
We first show that $z$ is $N$-contractible in $M'$.
If $\{z,z',z''\}$ is a triad of $M'$, then $z$ is $N$-contractible since it is in a series pair of $M'\del z'$, and $z'$ is $N$-deletable in $M'$.
So suppose $\{z,z',z''\}$ is coindependent in $M'$.
Then $\{z,z',z''\}$ is a cobasis for $\cl^{*}(G)$.
As $M'\del z',z''$ has an $N$-minor, and $G \cup z$ is contained in a series class in this matroid, it follows that $M' \del \{z',z''\} / z$ has an $N$-minor.
In particular, $z$ is $N$-contractible in $M'$.

Now $z$ is an $N$-contractible element of $M'$, so it follows from \cref{noextrastrong}(i) that $\si(M'/z)$ is not $3$-connected. Hence there is a vertical $3$-separation $(X,z,Y)$ of $M'$ for some $X$ and $Y$. But then either $X$ or $Y$ cospans $\cl^{*}(G)$ by \cref{cospanningG}. Assume $X$ cospans $\cl^{*}(G)$. Then $z\in \cl^{*}(X)$, and by the definition of a vertical $3$-separation, $z\in \cl(Y)$; a contradiction to orthogonality. 
\end{proof}

\begin{lemma}
\label{concosp}
 If $c \in E(M')$ is $N$-flexible, then $c\in \cl^{*}(G)$.
\end{lemma}

\begin{proof}
  Suppose that $c$ is $N$-flexible.
  By Bixby's Lemma, either $\si(M'/c)$ or $\co(M'\del c)$ is $3$-connected.
  \Cref{noextrastrong} then implies that $c \in \cocl(G)$, as required.
\end{proof}

\begin{lemma}
 \label{DinGplus}
 %$D\subseteq \cl^{*}(G)$.
 If $z \in E(M')$ is $N$-deletable, then $z \in \cocl(G)$.
\end{lemma}

\begin{proof}
Let $z\in E(M') - \cl^{*}(G)$, and suppose that $z$ is $N$-deletable.
It then follows from \cref{noextrastrong}(ii) that $\co(M'\del z)$ is not $3$-connected. Thus, by the dual of \cref{ZspanningB}, there is a cyclic $3$-separation $(X,z,Y)$ of $M'$ such that at most one element of $X$ is not $N$-flexible. We claim that $X \subseteq \cl^{*}(G)$. The claim follows immediately from \cref{concosp} unless $s\in X$ is the single element of $X$ that is not $N$-flexible. By the dual of \cref{ZspanningB}, the element $s$ is $N$-deletable and $\co(M'\del s)$ is $3$-connected, so, by \cref{noextrastrong}(ii), $s\in \cl^{*}(G)$. Thus $X \subseteq \cl^{*}(G)$, as claimed.

Since $(X,z,Y)$ is a cyclic $3$-separation, 
$r^{*}(X)\geq 3 \ge r^*(G)$.
Thus $\cl^{*}(X)=\cl^{*}(G)$. But $z\in \cl^{*}(X)$ because $(X,z,Y)$ is a cyclic $3$-separation in $M'$, so $z\in \cl^{*}(G)$; a contradiction. 
\end{proof}

\begin{lemma}
  \label{Gcoseg}
  Suppose that the confining set $G$ is a cosegment.
  Then $|(\cl^{*}(G)-G) \cap (C\cup D)| \le 1$.
  In particular, no elements of $\cocl(G)-G$ are $N$-contractible.
\end{lemma}

\begin{proof}
  As $\cocl(G)$ has corank two, $M'/p$ is $3$-connected for any $p \in \cocl(G)$, by \cref{longline3conn}.
  Thus, for any $p \in \cocl(G)-G$, \cref{noextrastrong}(i) implies that $p$ is not $N$-contractible.
  Let $p$ and $q$ be distinct elements in $C \cup D$ such that $p,q \in \cocl(G)-G$.
  Then $p,q \in D$, but $p$ is in a series class in $M' \ba q$, so $p$ is $N$-contractible; a contradiction.
  %The result follows.
\end{proof}

\begin{lemma}
  \label{bashingremix}
  Suppose that the confining set $G$ has corank three, and there is an element $p \in \cocl(G)-G$ that is $N$-contractible. Then either
  \begin{itemize}
    \item[(i)] $\cocl(G)-G = \{p\}$, or
    \item[(ii)] $|E(M)| \le |E(N)| + 9$.
  \end{itemize}
\end{lemma}

\begin{proof}
  Suppose that (i) does not hold.
  Then there are distinct elements $p$ and $q$ in $\cocl(G)-G$, where $p$ is $N$-contractible.
  By \cref{noextrastrong}(i), $\si(M'/p)$ is not $3$-connected.
  Let $(U,p,V)$ be a vertical $3$-separation of $M'$ such that $|U \cap E(N)| \le 1$ and $V \cup p$ is closed.  
  If $U$ (or $V$) cospans $\cocl(G)$, then $U$ (or $V$, respectively) also cospans $p$, as $p \in \cocl(G)$.  But this contradicts that $p \in \cl(U) \cap \cl(V)$.
  Thus $r^*(\cocl(G) \cap U) \le 2$ and $r^*(\cocl(G) \cap V) \le 2$.
  Recall that $G$ is the union of triads $T_1^*$ and $T_2^*$.
  It follows that $\cocl(G)-p$ is the union of two cosegments $G_1=\cocl(T_1^*)$ and $G_2=\cocl(T_2^*)$.
  Without loss of generality, we assume that $q \in G_1$, so $|G_1| \ge 4$.
  By the dual of \cref{longline3conn}, $M'/q$ is $3$-connected, so $q$ is not $N$-contractible, by \cref{noextrastrong}(i).  

  If $|G_2| \ge 4$, then, by \cref{nostronglongline}, $|G_2| = 4$ and $G_2 \cap B = \{x,y\}$.
  So $G_2$ is also a confining set.  But then $q \notin \cocl(G_2)$, contradicting \cref{DinGplus}.
  So we may assume that $|G_2| = 3$.

  If $G_1 - q$ contains an element that is $N$-deletable, then it follows that $q$ is $N$-contractible; a contradiction.  %So no elements in $G_1$ are $N$-flexible.
  So no element in $G_1-q$ is $N$-deletable;
  in particular, the $(N,B)$-strong element~$u \in G \cap B^*$ is not in $G_1$, and no elements in $G_1$ are $N$-flexible.
  Moreover, if $q \in U$, then $q$ is $N$-contractible by \cref{ZspanningB}; a contradiction.
  Letting $G_1 \cap G_2 = \{v\}$, we may now assume that $G_1-v \subseteq V$, and $G_2-v \subseteq U$.

  By \cref{ZspanningB}, each $y \in U$ is either $N$-flexible, or $y$ is $N$-contractible and $\si(M'/y)$ is $3$-connected.  In the former case, $y \in \cocl(G)$ by \cref{concosp}; in the latter, $y \in G$ by \cref{noextrastrong}(i).
  So $U \subseteq \cocl(G)$.  Since $G_1-v \subseteq V$, and $|U| \ge 3$, it now follows that $U=G_2$, where $G_2$ is the triad containing $\{u,v\}$.  Let $G_2 = \{u,v,w\}$.
  Note that $v$ is not $N$-deletable, since $v \in G_1-q$.
  It follows, by \cref{CPL2}, that $p \in \cl(U-v)$, so $\{u,w,p\}$ is a triangle of $M'$. % containing the $(N,B)$-strong element $u$.
  This triangle is coindependent, since $M'$ is $3$-connected, so it cospans $\cocl(G)$.
  Moreover, the only $N$-flexible elements of $M'$ are $\{u,w,p\}$.

  We now bound the elements of $C \cup D$ outside of $\cocl(G)$.
  By \cref{DinGplus}, every element of $C \cup D$ that is not in $\cocl(G)$ is in $C$.
  Let $z \in C - \cocl(G)$.
  Then, by \cref{noextrastrong}(i), $\si(M'/z)$ is not $3$-connected, so there is a vertical $3$-separation $(X,z,Y)$ such that $|X \cap E(N)| \le 1$ and $Y \cup z$ is closed.  Thus,
  by \cref{ZspanningB},
  at most one element of $X$ is not $N$-flexible, and if there is such an element~$s$, then $s$ is $N$-contractible and $\si(M'/s)$ is $3$-connected.
  If $X = \{u,w,p\}$, then $z \in \cl(X)-X$, but as $\{u,w,p\}$ cospans $G$, we then have $|\cocl(X)-X| > 1$, which contradicts \cref{gutspluscoguts1}.
  So $X$ contains an element $s$, where $s$ is $N$-contractible and $\si(M'/s)$ is $3$-connected, so $s \in G$ by \cref{noextrastrong}(i).
  Note that $q$ is in the coclosure of the coindependent triangle $\{u,w,p\}$, so $\{u,w,p,q\}$ is $3$-separating.
  By uncrossing $\{u,w,p,q\}$ and $X$, we observe that the set $P=\{u,w,p,q,s\}$ is also $3$-separating.  Moreover, since $z \in \cl(X)$, we have $z \in \cl(P)$.
  Let $Q=E(M')-(P \cup z)$.
  We may assume that $|Q| \ge 3$, otherwise the lemma holds trivially.  So $P \cup z$ is exactly $3$-separating.
  As $v \in \cocl(P \cup z)$, we have $v \notin \cl(Q-v)$, so $r(Q) \ge 3$.  Thus,
  $(P,z,Q)$ is a vertical $3$-separation, where $|P \cap E(N)| \le 1$.
  Since $q \in \cocl(P-q)$, we have $q \notin \cl(Q)$.
  By \cref{CPL2}(i), it follows that $q$ is $N$-contractible; a contradiction.

  We deduce that $C-\cocl(G) = \emptyset$.
  So $C \cup D \subseteq \cocl(G)$.  As $|(C \cup D) \cap (\cocl(G)-G)| =2$,
  we have $|C \cup D| \le 7$. Thus, $|E(M)| - |E(N)| = |C \cup D| + |\{a,b\}| \le 9$, as required.
\end{proof}

By \cref{DinGplus}, $D - \cocl(G) = \emptyset$.
We now focus on bounding $|C - \cocl(G)|$.

\begin{lemma}
  \label{uncrossingwin}
  Suppose that there exist distinct $p_1,p_2 \in E(M')-\cocl(G)$ such that $M'/p_i$ has an $N$-minor for $i \in \{1,2\}$.
  Let $(X_1,p_1,Y_1)$ and $(X_2,p_2,Y_2)$ be vertical $3$-separations of $M'$. 
  Then $|X_1 \cap X_2| \le 1$ or $|Y_1 \cap Y_2| \le 1$.
\end{lemma}
\begin{proof}
  Towards a contradiction, suppose that $|X_1 \cap X_2| \ge 2$ and $|Y_1 \cap Y_2| \ge 2$.
  By uncrossing, the sets $X_1\cup X_2$, $X_1\cup X_2\cup p_1$, $X_1\cup X_2\cup p_2$, and $X_1\cup X_2\cup \{p_1,p_2\}$ are all $3$-separating.
  Since $|Y_1 \cap Y_2| \ge 2$, the sets
  $X_1\cup X_2$, $X_1\cup X_2\cup p_1$, $X_1\cup X_2\cup p_2$, and $X_1\cup X_2\cup \{p_1,p_2\}$ are sides of exact $3$-separations of $M'$ and $p_1,p_2$ are guts elements.
  In particular, $(X_1\cup X_2\cup p_2, p_1, Y_1\cap Y_2)$ is a vertical $3$-separation of $M'$ unless $r(Y_1\cap Y_2)\leq 2$.
  But if $r(Y_1\cap Y_2)\leq 2$, then $(Y_1\cap Y_2)\cup \{p_1,p_2\}$ is a segment of $M'$ with at least four elements, so $p_1$ belongs to a non-trivial parallel class of $M'/p_2$.
  Then $p_1$ is $N$-deletable in $M'/p_2$ and hence in $M'$, so $p_1\in \cl^{*}(G)$ by \cref{concosp}; a contradiction.
  Thus $(X_1\cup X_2\cup p_2, p_1, Y_1\cap Y_2)$ is a vertical $3$-separation of $M'$, and either $|(X_1\cup X_2\cup p_2)\cap E(N)|\leq 1$ or $|(Y_1\cap Y_2)\cap E(N)|\leq 1$.

  If $|(X_1\cup X_2\cup p_2)\cap E(N)|\leq 1$, then there is an element $p_2$ in the non-$N$-side of $(X_1\cup X_2\cup p_2, p_1, Y_1\cap Y_2)$.
  Since $p_2 \in \cl(Y_1 \cap Y_2)$, it follows from \cref{CPL2}(ii) that $p_2$ is $N$-deletable.
  Hence $p_2\in \cl^{*}(G)$ by \cref{DinGplus}; a contradiction.
  So $|(Y_1\cap Y_2)\cap E(N)|\leq 1$.
  But $(X_1\cup X_2, p_1, (Y_1\cap Y_2) \cup p_2)$ is also a vertical $3$-separation of $M'$.
  Moreover, as $|E(N)| \ge 4$, we have $|(X_1\cup X_2\cup p_2)\cap E(N)|\ge 3$, so $|(X_1\cup X_2)\cap E(N)|\ge 2$
  and hence, by \cref{cplminorlemma}, $|(Y_1 \cap Y_2) \cup p_2| \le 1$.
  Again, it follows that $p_2 \in \cocl(G)$; a contradiction.
\end{proof}

\begin{lemma}
  \label{maybe}
  Suppose $\cocl(G)$ has at most six $N$-contractible elements. Then $|C-\cocl(G)| \le 2$.
\end{lemma}
\begin{proof}
  Suppose that $|C-\cocl(G)| \ge 3$.
  Let $p_1,p_2,p_3$ be distinct elements in $C-\cocl(G)$.
  It follows from \cref{noextrastrong}(i) that $\si(M'/p_i)$ is not $3$-connected, so there is a vertical $3$-separation $(X_i,p_i,Y_i)$ of $M'$, for each $i\in \{1,2,3\}$, where $|X_i\cap E(N)|\leq 1$ and $Y_i \cup p_i$ is closed.
  Then, by \cref{ZspanningB}, each element $x \in X_i$ is either $N$-flexible, or $x$ is $N$-contractible and $\si(M'/x)$ is $3$-connected.
  By \cref{concosp}, in the former case, and \cref{noextrastrong}(i), in the latter, $X_i \subseteq \cocl(G)$.
  Note that $|X_i| \ge 3$, for each $i$, and if $|X_i|=3$, then $X_i$ is a triad.

  Let $H$ be the set of $N$-contractible elements of $\cocl(G)$.
  Since, for $i \in \{1,2,3\}$,
  each element in $X_i$ is $N$-contractible, $X_i \subseteq H$, where $|H| \le 6$.
  We claim that $|X_i \cap X_j| \ge 2$ for some distinct $i,j \in \{1,2,3\}$.
  If, for some $\{i,j,k\} = \{1,2,3\}$, the sets $X_i$ and $X_j$ are disjoint, then $X_i \cup X_j = H$, so $X_k$ intersects $X_i$ or $X_j$ in two elements, as claimed.
  Similarly, if $|X_i| \ge 4$, then either $|X_i \cap X_j| \ge 2$, or $X_i \cup X_j = H$, in which case $X_k$ intersects $X_i$ or $X_j$ in two elements.
  So we may assume that $|X_i| = 3$ for each $i \in \{1,2,3\}$, and
  the pairwise intersection between any two of the three sets has size one.
  Let $X_2 = \{x_1,x_2,x_3\}$ where $X_1 \cap X_2 = \{x_1\}$ and $X_2 \cap X_3 = \{x_3\}$.
  Now $X_2 \cup p_2$ contains a circuit, since $(X_2,p_2,Y_2)$ is a vertical $3$-separation.  By orthogonality, this circuit does not meet the triad $X_3$, nor the triad $X_1$, so $X_2 \cup p_2$ contains a circuit of size at most two; a contradiction.
  This proves the claim.
  Without loss of generality, we may now assume that $|X_1\cap X_2|\geq 2$.

  If $|E(M')-\cocl(G)| \ge 4$, then $|Y_1 \cap Y_2| \ge |E(M')-(\cocl(G) \cup \{p_1,p_2\})| \ge 2$, which contradicts \cref{uncrossingwin}.
  So $|E(M')-\cocl(G)| \le 3$, in which case $E(M')-\cocl(G) = \{p_1,p_2,p_3\}$.
  Now, as $p_1 \notin \cocl(G)$, we have $p_1 \in \cl(\{p_2,p_3\})$, by orthogonality.
  Since $M'$ is $3$-connected, $\{p_1,p_2,p_3\}$ is a triangle.
  Then $\{p_2,p_3\}$ is a parallel pair in $M'/p_1$, so the element $p_2$ is $N$-deletable.
  As $\si(M/p_2)$ is not $3$-connected, $\co(M \ba p_2)$ is $3$-connected by Bixby's Lemma.
  But $p_2 \notin \cocl(G)$; contradicting \cref{noextrastrong}(ii).
  We deduce that $|C-\cocl(G)| \le 2$, thus completing the lemma.
\end{proof}

We handle one more special case.
\begin{lemma}
  \label{onemorecase}
  Suppose that the confining set $G$ has corank three, $\cocl(G)-G = \{q\}$ for some $q \in C \cup D$, and $|C-\cocl(G)| =2$. %\ge 2$.
  Then $|E(M)|\leq |E(N)|+9$.
\end{lemma}
\begin{proof}
  Let $p_1$ and $p_2$ be distinct elements in $C-\cocl(G)$.
  By \cref{noextrastrong}(i), $\si(M'/p_i)$ is not $3$-connected, so there is a vertical $3$-separation $(X_i,p_i,Y_i)$ of $M'$ for $i \in \{1,2\}$, where $|X_i \cap E(N)| \le 1$ and $Y_i \cup p_i$ is closed.
  By \cref{ZspanningB}, each element $x \in X_i$ is either $N$-flexible, or $x$ is $N$-contractible and $\si(M'/x)$ is $3$-connected.
  By \cref{concosp}, in the former case, and \cref{noextrastrong}(i), in the latter, $X_i \subseteq \cocl(G)$.
  Note that $|X_i| \ge 3$, for each $i$, and if $|X_i|=3$, then $X_i$ is a triad.
  Recall also that $G$ is the union of two triads $T_1^*$ and $T_2^*$.

  Suppose $q$ together with one of the triads, $T_1^*$ say, forms a cosegment.
  Then $\si(M/q)$ is $3$-connected, by the dual of \cref{longline3conn}, so $q$ is not $N$-contractible, by \cref{noextrastrong}(i).
  Now $X_1 \cup X_2 \subseteq G$, so it follows that $\{X_1,X_2\} = \{T_1^*,T_2^*\}$.
  But then $p_1 \in \cl(T_1^*)$, up to swapping the labels on $p_1$ and $p_2$, so, by orthogonality with $T_2^*$, we deduce that $\{p_1,s_1,t_1\}$ is a triangle, where $T_1^*-T_2^* = \{s_1,t_1\}$.
  Let $T_1^* \cap T_2^* = \{v\}$.  Now, as $T_1^* \cup q$ is a cosegment, $\{t_1,v,q\}$ is a triad that intersects the triangle $\{p_1,s_1,t_1\}$ in a single element; a contradiction.

  Now suppose $q \notin \cocl(T_1^*) \cup \cocl(T_2^*)$.
  We claim that $|X_1 \cap X_2| \ge 2$.
  Suppose not.  Then $|X_i| = 3$, for some $i\in \{1,2\}$, so we may assume $X_1$, say, is a triad.
  Either $X_1 \in \{T_1^*, T_2^*\}$, or $q \in X_1$ and $X_1$ intersects $T_1^*$ and $T_2^*$ in one element each.
  Let $T_1^* = \{v,s_1,t_1\}$ and $T_2^* = \{v,s_2,t_2\}$. 

  If $X_1 = T_1^*$, then, as $p_1 \in \cl(X_1)$, by orthogonality with $T_2^*$ we have that $\{p_1,s_1,t_1\}$ is a triangle.
  But as $\{s_2,t_2,q\}$ cospans $\cocl(G)$, the element $t_1$ is in a cocircuit contained in $\{s_2,t_2,q,t_1\}$, which contradicts orthogonality with the triangle $\{p_1,s_1,t_1\}$.

  On the other hand, if $X_1$ is a triad that meets both $\{s_1,t_1\}$ and $\{s_2,t_2\}$, and $q \in X_1$, then, $p_1$ is in a circuit contained in $X_1 \cup p_1$. But $X_1 \cup p_1$ meets $T_1^*$ and $T_2^*$ in a single element each, so, by orthogonality, $p_1$ is in a parallel pair; a contradiction.
  So $|X_1 \cap X_2| \ge 2$ as claimed.
  Note that the lemma holds trivially if $|E(M)| \le 11$, since $|E(N)| \ge 4$.
  So we may assume that $|E(M')| \ge 10$, in which case $|Y_1 \cap Y_2| \ge 2$.
  But this contradicts \cref{uncrossingwin}.
\end{proof}

Finally, we are in a position to prove the main result of this section.

\begin{prop}
\label{gadgetfinish}
Suppose that $M'$ has a confining set. Then $|E(M)|\leq |E(N)|+9$.
\end{prop}

\begin{proof}
  First, suppose that $G$ is a cosegment.
  Then, by \cref{Gcoseg}, $\cocl(G)-G$ has at most one element of $C \cup D$, and $\cocl(G)$ consists of at most four $N$-contractible elements (those elements in $G$).
  Therefore, by \cref{maybe}, $|C-\cocl(G)| \le 2$.
  As $D \subseteq \cocl(G)$, by \cref{DinGplus}, we have
  \begin{align*}
    |E(M)| - |E(N)| &\le |\cocl(G) \cap (C \cup D)| + |C-\cocl(G)| + |\{a,b\}| \\
    &\le 5 + 2 + 2 = 9.
  \end{align*}

  Now suppose that $G$ has corank three.
  Consider first the case where $|\cocl(G)-G| \ge 2$.
  If $\cocl(G)-G$ contains an element that is $N$-contractible, then, by \cref{bashingremix}, $|E(M)| \le |E(N)| + 9$, as required.
  So we may assume that no elements in $\cocl(G)-G$ are $N$-contractible.
  In particular, $\cocl(G)$ contains at most five $N$-contractible elements. %(those elements in $G$).
  Now, by \cref{maybe}, $|C-\cocl(G)| \le 2$.   
  Suppose there is an element $q \in D \cap (\cocl(G)-G)$, and let $p$ be an element in $\cocl(G)-G$, with $q \neq p$.
  Recall that $(C,D)$ was chosen such that $u \in D$, and
  note that $r^*_{M'\ba u}(\cocl(G)-u)=2$.
  Thus $p$ is in a series class in $M'\ba u \ba q$, so $p$ is $N$-contractible; a contradiction.
  It now follows that $|\cocl(G) \cap (C \cup D)| \le 5$, and hence $|E(M)| - |E(N)| \le 5 + 2 + |\{a,b\}| = 9$, as required.

  Now consider the case where $|\cocl(G)-G| \le 1$.
  Since $|\cocl(G)| \le 6$, \cref{maybe} implies that $|C-\cocl(G)| \le 2$.
  If $(\cocl(G)-G) \cap (C \cup D) = \emptyset$, then $|E(M)| - |E(N)| \le 5 + 2 + |\{a,b\}| = 9$ as required.
  So suppose that $\cocl(G)-G = \{q\}$ for some $q \in C \cup D$.
  We may also assume that $|C-\cocl(G)| \ge 2$, otherwise the result holds trivially.
  Now, by \cref{onemorecase}, $|E(M)|\leq |E(N)|+9$, as required.
\end{proof}

\section{Robust elements}
\label{robustelements}

Let $M$ be a $3$-connected matroid, let $N$ be a $3$-connected minor of $M$ such that %$N$ is non-binary, 
$|E(N)|\geq 4$,
and let $B$ be a basis of $M$. 
In this section, we consider the structure of $M$ that arises from elements that are $(N,B)$-robust but not $(N,B)$-strong.
Recall that a \textit{path of $3$-separations} of $M$ is a partition $(P_1,P_2,\ldots, P_n)$ of $E(M)$ such that $(P_1\cup \cdots \cup P_i, P_{i+1}\cup \cdots \cup P_n)$ is a $3$-separation of $M$ for each $i\in \{1,2,\ldots, n-1\}$.
The main result of this section shows that the presence of an element that is $(N,B)$-robust but not $(N,B)$-strong gives rise to a particular path of $3$-separations. 

Let $(X,z,Y)$ be a vertical $3$-separation of $M$. We say that $X$ is \textit{$z$-closed} if $X=\cl^{*}(X)$ and $X=\cl(X)-z$.
We use $z$-closure to ensure that the $(N,B)$-strong elements of $M$ are contained in the non-$N$-side of a vertical $3$-separation of $M$. A set is \textit{fully closed} if it is both closed and coclosed. Given a subset $A$ of $E(M)$, we use $\fcl_{M}(A)$ to denote the smallest fully closed set that contains $A$. Thus, the set $X$ is $z$-closed if $\fcl_{M/z}(X)=X$. 

Dually, given a cyclic $3$-separation $(X,z,Y)$, we say $X$ is \emph{$z$-coclosed} if $X$ is $z$-closed in $M^*$.

\begin{lemma}
\label{existszclosed}
If $z\in B$ and $z$ is $(N,B)$-robust but not $(N,B)$-strong, then there is some vertical $3$-separation $(X,z,Y)$ of $M$ such that $X$ is $z$-closed and $|X\cap E(N)|\leq 1$.
\end{lemma}

\begin{proof}
  By \cref{existsv3sep}, $M$ has a vertical $3$-separation $(X,z,Y)$, and we may assume that $|X\cap E(N)|\leq 1$, by \cref{cplminorlemma}.
The elements of $\fcl_{M/z}(X)-X$ can be ordered $(x_1,\ldots,x_m)$ such that $X\cup \{x_1,\ldots, x_i\}$ is $2$-separating in $M/z$ for all $i\in \{1,\ldots,m\}$.
Let $X_i=X\cup \{x_1,\ldots, x_i\}$ and $Y_i=Y-\{x_1,\ldots, x_i\}$ for each $i\in \{1,2,\ldots, m\}$.
We also let $(X_0,Y_0)=(X,Y)$. Suppose that $|X_j\cap E(N)|\geq 2$ for some $j\in \{1,2,\ldots, m\}$. We shall assume that $j$ is the smallest index such that $|X_j\cap E(N)|\geq 2$. Then %since $j$ is the smallest index, 
$|X_{j-1}\cap E(N)|\leq 1$, so $|Y_{j-1}\cap E(N)|\geq 3$ because $|E(N)|\geq 4$.
Hence $|Y_j\cap E(N)|\geq 2$. But then $(X_j,Y_j)$ is a $2$-separation of $M/z$ such that $|Y_j\cap E(N)|\geq 2$ and $|X_j\cap E(N)|\geq 2$; a contradiction. Hence $|X_i\cap E(N)|\leq 1$ for all $i\in \{1,\ldots,m\}$.
Thus, for each $i$, the partition $(X_i,Y_i)$ is a $2$-separation in $M/z$ such that $Y_i$ is the $N$-side.  It follows that $|Y_i|\geq 3$ for all $i$. In particular, $(X_m,Y_m)$ is a $2$-separation of $M/z$ such that $X_m$ is fully closed. Since $M$ is $3$-connected, $z\in \cl_{M}(X_m)\cap \cl_{M}(Y_m)$. Finally, $Y_m$ is not a parallel class of $M/z$ because $X_m$ is fully closed, so $r_{M}(Y_m)\geq 3$. Thus $(X_m,z,Y_m)$ is a $z$-closed vertical $3$-separation of $M$, as desired.
\end{proof}

Suppose that $F$ is a $4$-element fan of $M$ with ordering $(f_1,f_2,f_3,f_4)$ where $\{f_1,f_2,f_3\}$ is a triangle. 
We say that $(f_1,f_2,f_3,f_4)$ is a \textit{type-I fan relative to $B$} if $F \cap B = \{f_1,f_3\}$,
and $(f_1,f_2,f_3,f_4)$ is a \textit{type-II fan relative to $B$} if $F\cap B = \{f_1,f_3,f_4\}$.
When there is no ambiguity, we also say, in these cases, that $F$ is a type-I or type-II fan relative to $B$.

We need the following, which is one of the main results of \cite{brettell2014splitter}. 

\begin{lemma}[{\cite[Lemma 4.8]{brettell2014splitter}}]
\label{BS48}
Suppose that $z\in B$ is an element that is $(N,B)$-robust but not $(N,B)$-strong, and let $(X,z,Y)$ be a vertical $3$-separation of $M$ such that $|X\cap E(N)|\leq 1$. Then one of the following holds:

\begin{itemize}
\item[(i)] there are distinct $(N,B)$-strong elements $s_1,s_2\in X$; or
\item[(ii)] there are distinct $(N,B)$-strong elements $s_1\in X$ and $s_2\in \cl^{*}(X)\cap B$; or
\item[(iii)] there are distinct $(N,B)$-strong elements $s_1\in X$ and $s_2,s_3\in \cl(X)\cap B^*$; or
\item[(iv)] $M$ has a type-I or type-II fan relative to $B$ contained in $X\cup z$.
\end{itemize}
\end{lemma}

The next lemma is a consequence of \cref{BS48,CPL2}.

\begin{lemma}
 \label{BS}
 Let $z\in B$ be an element that is $(N,B)$-robust but not $(N,B)$-strong, and let $(X,z,Y)$ be a vertical $3$-separation of $M$ such that $X$ is $z$-closed and $|X\cap E(N)|\leq 1$. If there is at most one $(N,B)$-strong element of $M$ contained in $X$, then there is a type-I or type-II fan $(\alpha,\beta,\gamma,\delta)$ relative to $B$ that is contained in $X\cup z$ where $\beta,\gamma,\delta$ are $N$-contractible, and $\alpha,\beta,\gamma$ are $N$-deletable.
\end{lemma}

\begin{proof}
  Since $X$ is $z$-closed, it follows from \cref{BS48} that $M$ has a type-I or type-II fan $(\alpha,\beta,\gamma,\delta)$ relative to $B$ such that $\{\alpha,\beta,\gamma,\delta\} \subseteq X\cup z$. 
  Let $T^*$ be the triad $\{\beta,\gamma,\delta\}$.
  Note that $z \notin T^*$, since $z \in \cl(Y)$.
  Since $T^* \subseteq X$, it follows from orthogonality that $\beta,\gamma,\delta\notin \cl_{M}(Y)$. Hence $\beta,\gamma,\delta$  are $N$-contractible by \cref{CPL2}.
  It follows, since $\{\alpha,\beta,\gamma\}$ is a triangle of $M$ and $|E(N)| \ge 4$, that $\alpha,\beta,\gamma$ are also $N$-deletable in $M$.
\end{proof}

We will also require the following lemma, which can be proved by making routine modifications to \cite[Lemma~5.4]{whittle2013fixed} or \cite[Lemma~6.3]{brettell2014splitter}.

\begin{lemma}
\label{pathgenerator}
 Let $M$ be a $3$-connected matroid and let $(A,Z,B)$ a partition of $E(M)$ with $|A|,|B|\geq 2$. If, for all $z\in Z$, there is a path of $3$-separations $(A_z,z,B_z)$ such that $A\subseteq A_z$ and $B\subseteq B_z$, then there is an ordering $(z_1,\ldots,z_n)$ of the elements of $Z$ such that $(A,z_1,\ldots,z_n,B)$ is a path of $3$-separations of $M$.
\end{lemma}

For the remainder of this section, we work under the following assumptions.
Let $\mathbb{P}$ be a partial field, let $N$ be a non-binary $3$-connected strong $\mathbb{P}$-stabilizer for the class of $\mathbb{P}$-representable matroids, and let $M$ be an excluded minor for the class of $\mathbb{P}$-representable matroids.
Suppose that $M$ has a pair of elements $\{a,b\}$ such that $M\del a,b$ is $3$-connected with an $N$-minor, and let $M' = M \del a,b$.
Let $A$ be a $B\times B^{*}$ companion $\mathbb{P}$-matrix of $M$ such that $\{x,y,a,b\}$ incriminates $(M,A)$, where $\{x,y\}\subseteq B$ and $\{a,b\}\subseteq B^{*}$.
We assume that $M'$ has no confining set. 
We also assume that $B$ is chosen such that either there is one $(N,B)$-strong element $u$ of $M'$ outside of $\{x,y\}$, and $\{u,x,y\}$ is a triad; or there are no $(N,B)$-strong elements outside of $\{x,y\}$, and for any $B_1\times B_1^{*}$ companion $\mathbb{P}$-matrix $A_1$ where $\{x_1,y_1,a,b\}$ incriminates $(M,A_1)$, with $\{x_1,y_1\}\subseteq B_1$ and $\{a,b\}\subseteq B_1^{*}$, the matroid $M \ba a,b$ has no $(N,B_1)$-strong elements outside of $\{x_1,y_1\}$.
Note that such a $B$ exists by \cref{nogadgetsetup}.
Recall that we say that such a basis $B$ is \emph{strengthened}.

Let $S \subseteq E(M')$ be a set containing $\{x,y\}$ and any $(N,B)$-strong elements of $M'$, where either $|S|=2$ or $S$ is a triad.
%Note that all $(N,B)$-strong elements of $M'$ are contained in $S$.
In particular, observe that $S \subseteq \cocl_{M'}(\{x,y\})$.

For the remainder of the section, all ranks, coranks, closure operators, and coclosure operators are with respect to $M'$.

If $z$ is an element that is $(N,B)$-robust but not $(N,B)$-strong in $M'$, then there is a vertical (or cyclic) $3$-separation $(X,z,Y)$ of $M'$.
We now prove that if the non-$N$-side of this vertical $3$-separation is $z$-closed (or $z$-coclosed, respectively), then it contains $S$.
We first handle the case where $z \in B-\{x,y\}$.

\begin{lemma}
\label{SinY}
Let $z\in B-\{x,y\}$ be an element that is $(N,B)$-robust but not $(N,B)$-strong in $M'$, and let $(X,z,Y)$ be a vertical $3$-separation of $M'$ such that $X$ is $z$-closed and $|X\cap E(N)|\leq 1$. Then $S\subseteq X$. 
\end{lemma}

\begin{proof}
Suppose that there are at least two distinct $(N,B)$-strong elements in $X$.
By definition, the $(N,B)$-strong elements of $M$ contained in $X$ belong to $S$. If $|S|=2$, then it follows immediately that $S\subseteq X$. If $|S|=3$, then $S$ is a triad, so $S\subseteq X$ because $X$ is coclosed.
 
We may therefore assume that there is at most one $(N,B)$-strong element of $M'$ contained in $X$. Then it follows from \cref{BS} that there is a type-I or type-II fan $(\alpha,\beta,\gamma,\delta)$ relative to $B$ contained in $X\cup z$ where $\beta,\gamma,\delta$ are $N$-contractible and $\alpha,\beta,\gamma$ are $N$-deletable.
Let $F=\{\alpha,\beta,\gamma,\delta\}$.

%We need the following subproofs.

\begin{claim}
$\{x,y\}\cap \{\alpha,\gamma\}\neq \emptyset$. 
\end{claim}

\begin{subproof}
Assume that $\{x,y\}\cap \{\alpha,\gamma\}=\emptyset$. Suppose $\beta$ is an $(N,B)$-strong element of $M'$. Then, since $\beta\notin B$, it follows that $S=\{\beta,x,y\}$ is a triad of $M'$. Since $\{\alpha,\beta,\gamma\}$ is a triangle that meets $\{\beta,x,y\}$, it follows from orthogonality that $x$ or $y$ is in $\{\alpha, \gamma\}$; a contradiction because $\{x,y\}\cap \{\alpha,\gamma\}=\emptyset$. Thus $\beta$ is not an $(N,B)$-strong element of $M'$. Since $\beta$ is $N$-flexible, $\co(M'\del \beta)$ is not $3$-connected. Thus, by Bixby's Lemma, $\si(M'/\beta)$ is $3$-connected. Since $\{\alpha,\beta,\gamma\}$ is a triangle of $M'$ the cobasis element $\beta$ is spanned by the basis elements $\alpha$ and $\gamma$, so $A_{i\beta}\neq 0$ if and only if $i\in \{\alpha,\gamma\}$.
In particular, since $\{x,y\}\cap \{\alpha,\gamma\} = \emptyset$, this means that $A_{\alpha\beta}\neq 0$ and $A_{x\beta}=A_{y\beta}=0$.
Thus a pivot on $A_{\alpha\beta}$ is an allowable pivot. But then $\beta$ is an $(N,B\triangle \{\alpha,\beta\})$-strong element outside of $\{x,y\}$ such that $\beta \in B\triangle \{\alpha,\beta\}$; a contradiction of \cref{nostrongbasis}.
\end{subproof}

Now $\alpha$ or $\gamma$ is a member of $\{x,y\}$.
Suppose $\delta \in B$, in which case $(\alpha,\beta,\gamma,\delta)$ is a type-II fan.
Since $\delta$ is $N$-contractible and $\si(M/\delta)$ is $3$-connected by \cref{fanends}, it follows from \cref{nostrongbasis} that $\delta \in \{x,y\}$.
Hence $\{x,y\}\subseteq F-z$, and $S\subseteq \cl^{*}(F-z)\subseteq \cocl(X)=X$ as required.
Thus we may assume that $\delta\in B^{*}$, in which case $F$ is a type-I fan.

We first handle the case when $\alpha\in \{x,y\}$.

\begin{claim}
If $\alpha\in \{x,y\}$, then $S\subseteq X$. 
\end{claim}

\begin{subproof}
  Assume that $\alpha=x$.
  If $\beta$ is an $(N,B)$-strong element of $M'$, then $\{\beta,x,y\}$ is a triad of $M'$, so $S\subseteq \cl^{*}(\{\beta,x\}) \subseteq \cl^{*}(F-z)\subseteq X$, as required.
  So suppose that $\beta$ is not an $(N,B)$-strong element of $M'$.
  Consider the entry $A_{\alpha\beta}$.
  Since $\{\alpha,\beta,\gamma\}$ is a triangle of $M'$ it follows that $A_{\alpha\beta}\neq 0$, so a pivot on $A_{\alpha\beta}$ is an allowable pivot.
  Then $B'=B\triangle \{\alpha,\beta\}$ is a basis of $M'$, the set $\{\beta,y,a,b\}$ incriminates $(M,A^{\alpha\beta})$, and $\alpha$ is an $(N,B')$-strong element outside of $\{\beta,y\}$.
  Since $B$ is strengthened,
  %(in the case where there is one $(N,B)$-strong element in $M'$)
  there is some element $u\in B^{*}$ such that $u$ is $(N,B)$-strong and $\{u,x,y\}$ is a triad.
  Since $\beta$ and $\delta$ are not $(N,B)$-strong, it follows that $u\in E(M')-F$.
  But then, by orthogonality between the triad $\{u,\alpha,y\}$ and the triangle $\{\alpha, \beta, \gamma\}$, we have $y\in \{\beta,\gamma\}$, so $y=\gamma$.
  Therefore $S\subseteq \cl^{*}(\{x,y\}) \subseteq \cl^{*}(F-z)\subseteq X$.
\end{subproof}

We may now assume that $\alpha\notin \{x,y\}$, so $\gamma \in \{x,y\}$.
Suppose that $\gamma=x$. If $\beta$ is $(N,B)$-strong, then $\{\beta,x,y\}$ is a triad, and $\{\beta, \delta, x,y\}$ is a $4$-element cosegment, contradicting that $M'$ has no confining set.
We deduce that $\beta$ is not $(N,B)$-strong.

Suppose that $\co(M'\del x)$ is $3$-connected.
Since $\{\alpha,\beta,x\}$ is a triangle of $M'$, we have $A_{x\beta}\neq 0$, so a pivot on $A_{x\beta}$ is allowable.
Then $B'=B\triangle \{x,\beta\}$ is a basis such that $x$ is an $(N,B)$-strong element outside of $\{\beta,y\}$, where $\{\beta,y,a,b\}$ incriminates $(M,A^{x\beta})$.
Since $B$ is strengthened, there is some $(N,B)$-strong element $u\in B^{*}$ such that $\{u,x,y\}$ is a triad. Since $\beta$ is not $(N,B)$-strong, $u$ is not in the triangle $\{\alpha,\beta,x\}$.  It then follows from orthogonality that $\alpha=y$; a contradiction. %Thus $S\subseteq \cl^{*}(F)\subseteq X$.
So $\co(M'\del x)$ is not $3$-connected, and thus $\si(M'/x)$ is $3$-connected by Bixby's Lemma.

Since $\beta$ is not $(N,B)$-strong, there is a cyclic $3$-separation $(P,\beta,Q)$ of $M'$.
By orthogonality, we may assume that $x\in P$ and $\alpha\in Q$.
Consider $(P-x,x,Q\cup \beta)$.
Observe that $Q\cup \beta$ and $Q\cup \{\beta,x\}$ are exactly $3$-separating, since $x \in \cl(Q\cup \beta)$.
But $(P-x,x,Q\cup \beta)$ is not a vertical $3$-separation of $M'$, since $\si(M'/x)$ is $3$-connected.
Thus $r(P-x)\leq 2$, so $P$ contains a triangle.
By orthogonality, $P$ is a triangle and $P=\{x,\delta,\mu\}$ for some $\mu\in E(M')$.
Thus $M'$ has a $5$-element fan with ordering $(\alpha,\beta,x,\delta,\mu)$.
Moreover, $\mu\in \cl(Q)$ or else $\{\beta,x,\delta,\mu\}$ is a $4$-element cosegment; a contradiction to orthogonality.
Now $\co(M'\del \mu)$ is $3$-connected by \cref{fanends}, and $\mu$ is $N$-deletable since $\mu$ is in a non-trivial parallel class in $M'/\gamma$.
Suppose $\mu\in B^{*}-\{a,b\}$. Then $\mu$ is $(N,B)$-strong and outside of $\{x,y\}$, so $\{\mu,x,y\}$ is a triad.
By orthogonality, it follows that $\alpha=y$, contradicting the assumption that $\alpha\notin \{x,y\}$.
We deduce that $\mu\in B$. 

We now repeat this argument, interchanging the roles of $x$ and $\beta$.
Since $\co(M'\del x)$ is not $3$-connected, there is a cyclic $3$-separation $(P',x,Q')$ of $M'$.
By orthogonality, we may assume that $\beta\in P'$ and $\alpha\in Q'$.
Consider $(P'-\beta, \beta, Q'\cup x)$.
Observe that $Q'\cup x$ and $Q'\cup \{x,\beta\}$ are exactly $3$-separating, since $\beta \in \cl(Q'\cup x)$.
But $(P'-\beta, \beta, Q'\cup x)$ is not a vertical $3$-separation of $M'$, since $\si(M'/\beta)$ is $3$-connected, by Bixby's Lemma.
Thus $r(P'-\beta)=2$, and it follows by orthogonality that $P'$ is a triangle of $M'$.
By orthogonality between $P'$ and $\{\beta,x,\delta\}$, we have $\delta\in P'$.
Since $\beta\notin \cl(P)$, it follows that $\mu\in Q'$.
Let $P' = \{\beta,\delta,\varepsilon\}$ for some $\varepsilon\in E(M')$.
Now, $\varepsilon \in \cl(Q')$ or else $\{\varepsilon,x,\beta,\delta\}$ is a $4$-element cosegment; a contradiction to orthogonality.

Now $\alpha,\mu,\varepsilon$ are in the closure of the triad $\{\beta,x,\delta\}$, so $\{\alpha,\mu,\varepsilon\}$ is a triangle.
But $\alpha,\mu \in B$, so $\varepsilon \in B^{*}$.
We claim that $\varepsilon$ is an $(N,B)$-strong element of $M'$.
That $\varepsilon$ is $(N,B)$-robust follows from the fact that $\beta$ is $N$-contractible and $\{\delta,\varepsilon\}$ is a parallel pair in $M'/\beta$.
Since $(F,\varepsilon,E(M')-F)$ is a vertical $3$-separation of $M'$, \cref{existsv3sep} and Bixby's Lemma imply that $\co(M'\del \varepsilon)$ is $3$-connected.
As $\varepsilon$ is an $(N,B)$-strong element of $M'$ outside of $\{x,y\}$,
we have that $\{\varepsilon,x,y\}$ is a triad of $M'$.
But $\{\varepsilon,x,y\}$ intersects the triangle $\{\beta,\delta,\varepsilon\}$ in a single element; a contradiction to orthogonality.
\end{proof}

Next we handle the case where the element $z$, which is $(N,B)$-robust but not $(N,B)$-strong, is in $B^*$.  %The proof is similar, but not identical, to \cref{SinY}, due to a lack of duality; we provide the proof for completeness.

\begin{lemma}
\label{dualSinY}
Let $z\in B^{*}$ be an element of $M'$ that is $(N,B)$-robust but not $(N,B)$-strong, and let $(X,z,Y)$ be a cyclic $3$-separation of $M'$ such that $X$ is $z$-coclosed and $|X\cap E(N)|\leq 1$. Then $S\subseteq X$. 
\end{lemma}

\begin{proof}
Suppose that there are at least two distinct $(N,B)$-strong elements in $X$. The $(N,B)$-strong elements of $M'$ contained in $X$ must belong to $S$ by the definition of $S$. If $|S|=2$, then it follows immediately that $S\subseteq X$. If $|S|=3$, then $S$ is a triad, so $S\subseteq \cocl(X) = X \cup z$, as $X$ is $z$-coclosed, but $z \notin S$, so $S \subseteq X$ as required.
 
We may therefore assume that there is at most one $(N,B)$-strong element of $M'$ contained in $X$.  Then it follows from the dual of \cref{BS} that there is a type-I or type-II fan $(\alpha,\beta,\gamma,\delta)$ relative to $B^{*}$ in $(M')^*$ that is contained in $X\cup z$ where $\beta,\gamma,\delta$ are $N$-deletable and $\alpha,\beta,\gamma$ are $N$-contractible in $M'$.

Suppose that $F$ is a type-II fan relative to $B^{*}$ in $(M')^*$.
Then $\delta$ is an $(N, B)$-strong element of $M'$ by \cref{fanends}. Hence $M'$ has a triad $\{\delta, x,y\}$.
By orthogonality with the triangle $\{\beta,\gamma,\delta\}$, we have $\{\beta,\gamma\} \cap \{x,y\} \neq \emptyset$; but $\gamma \notin B$, so $\beta\in \{x,y\}$.
Since $\{\beta,\delta\} \subseteq F-z \subseteq X$, we have
$S\subseteq \cocl(\{\beta,\delta\}) \subseteq \cocl(X) = X \cup z$, as $X$ is $z$-coclosed.  But $z \notin S$, so $S \subseteq X$ as required.

We may now assume that $F$ is a type-I fan relative to $B^{*}$ in $(M')^*$.
If $\gamma$ is an $(N, B)$-strong element of $M'$, then $M'$ has a triad $\{\gamma, x,y\}$.
By orthogonality with $\{\beta,\gamma,\delta\}$, either $\beta\in \{x,y\}$ or $\delta\in \{x,y\}$.  As $z \notin \{\beta,\delta\}$, we have $S\subseteq \cocl(F-z) \subseteq X$ because $X$ is $z$-coclosed and $z \notin S$.
Therefore we may also assume that $\co(M'\ba \gamma)$ is not $3$-connected, so $\si(M'/ \gamma)$ is $3$-connected, by Bixby's Lemma.

\begin{claim}
$\{x,y\}\cap \{\beta,\delta\}\neq \emptyset$. 
\end{claim}

\begin{subproof}
Suppose that $\{x,y\}\cap \{\beta,\delta\}=\emptyset$. Then, since $\{\beta,\gamma,\delta\}$ is a triangle of $M'$, it follows that $A_{x\gamma}=A_{y\gamma}=0$ and $A_{\beta\gamma}\neq 0$.
Hence a pivot on $A_{\beta\gamma}$ is allowable, and $\gamma$ is in the basis $B'=B\triangle \{\beta,\gamma\}$ of $M'$, where $\{x,y,a,b\}$ incriminates $(M,A^{\beta\gamma})$.
But then $\gamma$ is an $(N,B')$-strong element of $M'$ in $B'-\{x,y\}$; a contradiction of \cref{nostrongbasis}.
\end{subproof}

Suppose $\delta\in \{x,y\}$.
Then, since $\{\beta,\gamma,\delta\}$ is a triangle of $M'$, $A_{\delta\gamma}\neq 0$, and a pivot on $A_{\delta\gamma}$ is allowable.
Hence $M'$ has a basis $B'=B\triangle \{\delta,\gamma\}$ with an $(N,B')$-strong element $\delta$ in $(B')^{*}$.
Since $B$ is a strengthened basis, there is an $(N,B)$-strong element $u\in B^{*}$ such that $S=\{u,x,y\}$ is a triad of $M'$.
By orthogonality, either $\beta\in S$ or $\gamma\in S$.
Hence $S\subseteq \cocl(F-z) \subseteq X$ because $X$ is $z$-coclosed and $z \notin S$.
A similar argument holds if $\beta\in \{x,y\}$ and $\co(M'\ba \beta)$ is $3$-connected.     

We may now assume that $\beta\in \{x,y\}$ and that $\co(M'\ba \beta)$ is not $3$-connected. Let $(P,\beta,Q)$ be a cyclic $3$-separation of $M'$. Since $\beta$ is in a triangle of $M'$, we may assume that $\gamma\in P$ and $\delta\in Q$.
Consider $(P-\gamma,\gamma,Q\cup \beta)$.
Observe that $Q\cup \beta$ and $Q\cup \{\beta,\gamma\}$ are exactly $3$-separating, the latter since $\gamma \in \cl(Q \cup \beta)$.
But $(P-\gamma,\gamma,Q\cup \beta)$ is not a vertical $3$-separation of $M'$, since $\si(M'/\gamma)$ is $3$-connected, so it follows that $r(P)=2$.
By orthogonality with the triad $\{\alpha,\beta,\gamma\}$, it follows that $\alpha\in P$ and $P$ is a triangle of $M'$.
Thus $P=\{\alpha,\gamma,p\}$ for some $p \in E(M')-F$.

Let $(P',\gamma,Q')$ be a cyclic $3$-separation of $M'$. Since $\{\beta,\gamma,\delta\}$ is a triangle of $M'$, we may assume that $\beta\in P'$ and $\delta\in Q'$.
Now $Q' \cup \gamma$ and $Q' \cup \{\beta,\gamma\}$ are exactly $3$-separating, but $(P'-\beta,\beta,Q'\cup \gamma)$ is not a vertical $3$-separation of $M'$, since $\si(M'/\beta)$ is $3$-connected, by Bixby's Lemma.
It follows, by orthogonality, that $\alpha\in P'$ and $P'$ is a triangle of $M'$.
Therefore $P'=\{\alpha,\beta,p'\}$ for some $p' \in E(M')-F$.  Note also that $p \neq p'$, since the triad $\{\alpha,\beta,\gamma\}$ is independent.

Now $\{\alpha,\beta,\gamma,\delta,p,p'\}$ is a rank-$3$ subset of $M'$, with $\{\beta,\delta\} \subseteq B$.
Hence, at least one of $p$ and $p'$ is in $B^{*}$.
Suppose $p'\in B^{*}$. It follows, by \cref{fanends}, that $p'$ is an $(N, B)$-strong element of $M'$.
But then $S=\{p',x,y\}$ is a triad of $M'$ that meets the triangle $\{\beta,\gamma,\delta\}$, since $\beta \in \{x,y\}$.
By orthogonality, and since $p' \notin F$, we have $\{x,y\} \subseteq F-z$.
It follows by $z$-coclosure that $S\subseteq X$.
A similar argument applies if $p \in B^*$.
\end{proof}

In the next lemma, we show that elements on the non-$N$-side of a vertical $3$-separation that are not $N$-flexible are not $(N,B)$-robust.

\begin{lemma}
 \label{notrobust} 
Let $z\in B-\{x,y\}$ be an element that is $(N,B)$-robust but not $(N,B)$-strong in $M'$, and let $(X,z,Y)$ be a vertical $3$-separation of $M'$ such that $|X \cap E(N)|\leq 1$ and $S \subseteq X$.
Then for $e\in X-S$, the element~$e$ is $N$-flexible if and only if $e$ is $(N,B)$-robust.
Moreover, at most one element in $X-S$ is not $N$-flexible in $M'/z$, and if such an element~$\mu$ exists, then $(X- \mu, z,Y\cup \mu)$ is a vertical $3$-separation of $M'$. 
\end{lemma}

\begin{proof}
  Clearly, if $e \in X-S$ is $N$-flexible, then $e$ is $(N,B)$-robust.
  Suppose $e \in X-S$ is not $N$-flexible.
  By \cref{CPL2}, either $e$ is $N$-deletable but not $N$-contractible, or $e$ is $N$-contractible but not $N$-deletable.

  First, suppose that $e$ is $N$-deletable but not $N$-contractible.
  Then $e \in \cl(Y)$, by \cref{CPL2}(i).
  It follows that $((X-e) \cup z, e, Y)$ is a vertical $3$-separation of $M'$, so $\co(M' \ba e)$ is $3$-connected by Bixby's Lemma.
  Since $e \notin S$, it follows that $e \in B-\{x,y\}$, so $e$ is not $(N,B)$-robust.
  Moreover, if $e$ and $e' \in X-S$ are $N$-deletable but not $N$-contractible, then $\{z,e,e'\} \subseteq \cl(Y)-Y$, so $r(\{z,e,e'\}) = 2$.  But $\{z,e,e'\} \subseteq B$, so $e=e'$.
  %Finally, observe that $(X-e,z,Y \cup e)$ is a vertical $3$-separation of $M'$.

  Now suppose that $e$ is $N$-contractible but not $N$-deletable.
  Let $Y' = \cl(Y)-z$.
  By \cref{CPL2}(ii), $e \in \cocl(Y')-Y'$ and $z \in \cl(X-(Y' \cup e))$, and there is only one such element $e$.
  Observe that $Y'$ and $Y' \cup e$ are exactly $3$-separating.
  Moreover, $(X\cup z)-(Y' \cup e)$ contains a circuit, implying $r^*((X \cup z)-(Y'\cup e)) \ge 3$.
  Now $((X\cup z)-(Y' \cup e), e, Y')$ is a cyclic $3$-separation, so $\si(M' / e)$ is $3$-connected by Bixby's Lemma.
  As $e \notin S$, it follows that $e \in B^*$, so $e$ is not $(N,B)$-robust.

  Suppose $\mu$ and $\mu'$ are distinct elements of $X-S$ that are not $N$-flexible.  Then, by the foregoing, we may assume that $\mu$ is not $N$-deletable, and $\mu'$ is not $N$-contractible.
  Note that $M/z$ is the two sum of $M_X$ and $M_Y$ with basepoint $z'$ say, where $M_X\del z'=(M/z)|X$ and $M_Y\del z'=(M/z)|Y$.
  Since $\mu$ is not $N$-deletable and $\mu'$ is not $N$-contractible,
  $\{z',\mu\}$ is a cocircuit in $M_X$, and $\{z',\mu'\}$ is a circuit in $M_X$; a contradiction to orthogonality.
  So at most one element in $X-S$ is not $N$-flexible.

  Now let $\mu$ be the unique element in $X-S$ that is not $N$-flexible, and consider $(X- \mu, z,Y\cup \mu)$.
  If $\mu$ is $N$-deletable but not $N$-contractible, then, as $\mu \in \cl(Y)$, clearly $(X-\mu,z,Y \cup \mu)$ is a vertical $3$-separation of $M'$.
  Suppose that $\mu$ is $N$-contractible but not $N$-deletable.
  Since $\mu$ is the only element in $X-S$ that is not $N$-flexible, $\cl(Y) = Y \cup z$, so $\mu \in \cocl(Y)$.
  Thus, if $(X- \mu, z,Y\cup \mu)$ is not a vertical $3$-separation of $M'$, then $r(X-\mu)\leq 2$. But $X-\mu$ spans $z$, and $\{x,y\}\subseteq S \subseteq X-\mu$, so $\{x,y,z\}$ is a triangle of $M'$ contained in $B$; a contradiction.
\end{proof}

Note that a similar argument applies when $z \in B^{*}$ is an element of $M'$ that is $(N,B)$-robust but not $(N,B)$-strong; we omit the proof.

\begin{lemma}
 \label{notrobustdual} 
Let $z\in B^{*}$ be an element of $M'$ that is $(N,B)$-robust but not $(N,B)$-strong, and let $(X,z,Y)$ be a cyclic $3$-separation of $M'$ such that $|X \cap E(N)|\leq 1$ and $S\subseteq X$.
Then for $e\in X-S$, the element~$e$ is $N$-flexible if and only if $e$ is $(N,B)$-robust.
Moreover, at most one element in $X-S$ is not $N$-flexible in $M' \ba z$, and if such an element~$\mu$ exists, then $(X- \mu, z,Y\cup \mu)$ is a cyclic $3$-separation of $M'$.
\end{lemma}

%\begin{proof}
%The first part is easy from Lemma \ref{CPL2}. For the second part, we need to show that $(X\cup \mu, z,z,Y-\mu)$ is a vertical $3$-separation of $(M')^{*}$. This is clear unless $Y-\mu$ has rank $2$ in $(M')^{*}$. But since $Y-\mu$ spans $z$ by Lemma \ref{CPL2}, it follows that $Y-\mu=\{x,y\}$ because $M'$ has no gadget. Then $Y$ is a triad of $(M')^{*}$ contained in the cobasis $B$ of $(M')^{*}$; a contradiction. 
%\end{proof}

%In other words, the same sort of thing works, but we get the contradiction that $\{x,y,\mu\}$ is a triangle contained in $B$.

We now come to the main result of the section.

\begin{prop}
\label{thepathexist}
Let $z \in E(M')-\{x,y\}$ be an element that is $(N,B)$-robust but not $(N,B)$-strong in $M'$.
Then $M'$ has a path of $3$-separations $(S,z_1,z_2,\dotsc,z_n,z,Y)$ where the elements in $\{z_{1},\ldots, z_n\}$ are $N$-flexible, $|(S \cup \{z_1,\ldots, z_n,z\}) \cap E(N)|\leq 1$, and $|S \cup \{z_1,\dots,z_n\}| \ge 3$.
\end{prop}

\begin{proof}
Let $z$ be an element of $M'$ that is $(N,B)$-robust but not $(N,B)$-strong.
First, suppose $z \in B-\{x,y\}$.
By \cref{existszclosed}, there exists a vertical $3$-separation $(X',z,Y')$ such that $X'$ is $z$-closed and $|X' \cap E(N)| \le 1$.
By \cref{SinY}, $S \subseteq X'$.
By \cref{notrobust}, $X'-S$ contains at most one element that is not $(N,B)$-robust. If such an element~$\mu$ exists, let $(X,Y) = (X'-\mu,Y' \cup \mu)$; otherwise, let $(X,Y)=(X',Y')$.
Now, by \cref{notrobust} again, $(X,z,Y)$ is a vertical $3$-separation of $M'$ where $S \subseteq X$, $|X \cap E(N)| \le 1$, and every element in $X-S$ is $N$-flexible. % (and $(N,B)$-robust);
We say that $(X,z,Y)$ is a \textit{good separation for $z$} in $M'$.
Similarly, if $z \in B^*-\{a,b\}$, then, by the dual of \cref{existszclosed}, and \cref{dualSinY,notrobustdual}, there is a cyclic $3$-separation $(X,z,Y)$ that is a good separation for $z$ in $(M')^*$.
Thus, for each $(N,B)$-robust element $z$ of $M'$ outside of $\{x,y\}$, there is a good separation for $z$.
%All of the elements on the non-$N$-side of the good separation are $N$-flexible, and therefore $(N,B)$-robust.

We now show that a good separation induces a path of $3$-separations in $M'$.
Let $(X,z,Y)$ be a good separation for $z$, in either $M'$ or $(M')^*$, and let $Z=X-S$.
Consider the partition $(S, Z, z\cup Y)$ of $E(M')$, and
note that each $z_i \in Z$ is $N$-flexible, $|(S \cup Z \cup z) \cap E(N)| \le 1$, and $|S \cup Z| \ge 3$.

We claim that, for each $z_i\in Z$, there is a path of $3$-separations $(X_i,z_i,Y_i)$ of $M'$ such that $S\subseteq X_i$ and $z \cup Y\subseteq Y_i$.
In what follows, we assume that $z,z_i \in B$, but the argument is similar if one or both of $z,z_i$ are in $B^*$.
Since $z_i$ is $N$-flexible in $M'/z$, we can fix an $N$-minor of $M'/z/z_i$ on ground set $E_N$.
We may assume that $|X \cap E_N| \le 1$.
As $z_i$ is $N$-flexible, and hence $(N,B)$-robust, in $M'$, but $z_i$ is not $(N,B)$-strong, there is a vertical $3$-separation $(X'_i,z_i,Y'_i)$ of $M'$ where $|X'_i \cap E_N|\leq 1$.
By \cref{existszclosed}, we may assume that $X'_i$ is $z_i$-closed. 
Hence $S\subseteq X'_i$ by \cref{SinY} (in the case that $z_i \in B^*$, we can use \cref{dualSinY}).
Since $|E_N|\geq 4$, it now follows that $|Y\cap Y'_i|\geq |E_N|-2\geq 2$.
Therefore, by uncrossing $Y\cup z$ and $Y'_i$, the set $Y\cup Y'_i\cup z$ is $3$-separating.
Similarly, by uncrossing $Y\cup z$ and $Y'_i\cup z_i$, the set $Y\cup Y'_i\cup \{z, z_i\}$ is $3$-separating.
Now $(X_i,z_i,Y_i) = (X\cap X'_i, z_i, Y\cup Y'_i\cup z)$ is a path of $3$-separations of $M'$ that satisfies the claim.

It now follows from \cref{pathgenerator} that there is an ordering $(z_1,\ldots, z_n)$ of $Z$ such that $(S,z_1,\ldots,z_n,z,Y)$ is a path of $3$-separations of $M'$, satisfying the proposition.
%The second statement follows from \cref{notrobust,notrobustdual}.
\end{proof}

\section{Proof of \cref{mainthm1}}
\label{robustpath}

Let $\mathbb{P}$ be a partial field,
let $N$ be a non-binary $3$-connected strong $\mathbb{P}$-stabilizer for the class of $\mathbb{P}$-representable matroids, and let $M$ be an excluded minor for the class of $\mathbb{P}$-representable matroids with a pair of elements $\{a,b\}$ such that $M\ba a,b$ is $3$-connected with an $N$-minor.
Let $A$ be a $B\times B^{*}$ companion $\mathbb{P}$-matrix of $M$ such that $\{x,y,a,b\}$ incriminates $(M,A)$, where $\{x,y\}\subseteq B$ and $\{a,b\}\subseteq B^{*}$,
for some basis $B$ of $M$. 
Let $M'=M\ba a,b$.
We assume that $M'$ has no confining set, and that $B$ is a strengthened basis.
%(Recall, this means that $B$ is chosen such that either there is one $(N,B)$-strong element $u$ of $M'$ outside of $\{x,y\}$, and $\{u,x,y\}$ is a triad; or there are no $(N,B)$-strong elements outside of $\{x,y\}$, and for any $B_1\times B_1^{*}$ companion $\mathbb{P}$-matrix $A_1$ where $\{x_1,y_1,a,b\}$ incriminates $(M,A_1)$, with $\{x_1,y_1\}\subseteq B_1$ and $\{a,b\}\subseteq B_1^{*}$, the matroid $M \ba a,b$ has no $(N,B_1)$-strong elements outside of $\{x_1,y_1\}$.)
%Note that such a $B$ exists by \cref{nogadgetsetup}.

Let $S \subseteq E(M')$ be a set containing $\{x,y\}$ and any $(N,B)$-strong elements of $M'$, where either $|S|=2$ or $S$ is a triad.
If $M'$ has an element~$z$ that is $(N,B)$-robust but not $(N,B)$-strong, 
then, by \cref{thepathexist}, $M'$ has a path of $3$-separations of the form $(S,z'_1,\ldots,z'_{n'},z,Y)$ where each element in $\{z'_1,\ldots, z'_{n'}\}$ is $N$-flexible.
In this section, we study such paths of $3$-separations, in order to prove \cref{mainthm1}.

It is convenient to write such a path of $3$-separations as $(\{x,y\},z_1,\dotsc,z_n,z,Y)$.
(In the case where $|S|=3$, $z_1$ labels the $(N,B)$-strong element outside of $\{x,y\}$.)
We say that $(\{x,y\},z_1,\dotsc,z_n,z,Y)$ is a \emph{good path of $3$-separations} for $z$.
Note that $n \ge 1$, since $|S \cup \{z'_1,\dotsc,z'_{n'}\}| \ge 3$, 
and observe that $z_i$ is $(N,B)$-robust for each $i \in \{1,2,\dotsc,n\}$. 

\begin{lemma}
\label{nottrianglejoin}
Suppose $M'$ has an element~$z$ that is $(N,B)$-robust but not $(N,B)$-strong, and let $(\{x,y\},z_1,\dotsc,z_n,z,Y)$ be a good path of $3$-separations for $z$.
Then 
\begin{itemize}
  \item[(i)] $\{x,y,z_1\}$ is a triad of $M'$, and
  \item[(ii)] $z_1\in B^{*}$. 
\end{itemize}
\end{lemma}

\begin{proof}
  First we prove (i).
  Clearly (i) holds when $z_1 \in S$, so we may assume that $z_1 \notin S$.
  It suffices to show that $\{x,y,z_1\}$ is not a triangle of $M'$.
  Towards a contradiction, suppose $\{x,y,z_1\}$ is a triangle.
  Then $z_1 \in B^*$, since $\{x,y\} \subseteq B$.
  Thus $\co(M'\del z_1)$ is not $3$-connected, as $z_1$ is $(N,B)$-robust but not $(N,B)$-strong. 

  Let $(P,z_1,Q)$ be a cyclic $3$-separation of $M'$.
  Since $\{x,y,z_1\}$ is a triangle, it follows from orthogonality that $|P\cap \{x,y\}|=|Q\cap \{x,y\}|=1$.
  We shall therefore assume that $x\in P$ and $y\in Q$.
  Since $z_1$ is $N$-flexible, it follows that $x$ and $y$ are $N$-deletable in $M'$.
  Suppose $\co(M'\del x)$ is $3$-connected.
  Due to the triangle $\{x,y,z_1\}$, $A_{xz_1} \neq 0$, so a pivot on $A_{xz_1}$ is allowable.
  Thus $B' = B \triangle \{x,z_1\}$ is a basis, $\{z_1,y,a,b\}$ incriminates $(M,A^{xz_1})$, and $x$ is an $(N,B')$-strong element outside of $\{z_1,y\}$, contradicting that $B$ is a strengthened basis.
  Hence $\co(M'\del x)$ and, by symmetry, $\co(M'\del y)$ are not $3$-connected. 

  Now $P \cup z_1$ is exactly $3$-separating, and $y \in \cl(P \cup z_1)$, so $(P\cup z_1, y,Q-y)$ is a path of $3$-separations of $M'$, and $y \in \cl(Q-y)$.
  Similarly, $(P-x,x,Q\cup z_1)$ is a path of $3$-separations of $M'$.
  If $r(P\cup z_1)\geq 3$ and $r(Q-y)\geq 3$, then $(P\cup z_1, y,Q-y)$ is a vertical $3$-separation of $M'$, in which case $\si(M'/y)$ is not $3$-connected, contradicting Bixby's Lemma.
  Therefore $r(P\cup z_1)\leq 2$ or $r(Q-y)\leq 2$.
  But if $r(P\cup z_1)\leq 2$, then $M'\del z_1$ is $3$-connected by \cref{longline3conn}; a contradiction. 
  Thus $r(Q-y)\leq 2$, and hence $r(Q)\leq 2$. Similarly, it follows that $r(P-x)\leq 2$, and hence $r(P)\leq 2$. Since $x\in P$, $y\in Q$, and $\co(M'\del x)$ and $\co(M'\del y)$ are not $3$-connected, it follows from \cref{longline3conn} that $|P| = 3$ and $|Q| = 3$.
  But now $|E(M')| = 7$, so $n=1$, and it is readily checked that the $(N,B)$-robust element~$z_2$ is $(N,B)$-strong; a contradiction.

  We now prove (ii).
  When $z_1$ is an $(N,B)$-strong element of $M'$, (ii) holds by \cref{nostrongbasis}.
  So we may assume that $M'$ has no $(N,B)$-strong elements outside of $\{x,y\}$.
  Towards a contradiction, suppose that $z_1\in B$. Then $\si(M'/z_1)$ is not $3$-connected. 
  Now $M'$ has an $(N,B)$-robust element $z'\in \{z_2,z_3,\dotsc,z_n,z\}$ that is either in the closure or coclosure of the triad $\{x,y,z_1\}$. 
  If $z'$ is in the coclosure of $\{x,y,z_1\}$, then $\{x,y,z_1,z'\}$ is a $4$-element cosegment of $M'$, so $\si(M'/z_1)$ is $3$-connected by the dual of \cref{longline3conn}; a contradiction.
  Thus $z'$ is in the closure of $\{x,y,z_1\}$.
  Then $(\{x,y,z_1\},z',E(M')-\{x,y,z_1,z'\})$ is a vertical $3$-separation of $M'$, so $\si(M'/z')$ is not $3$-connected.
  Hence $\co(M'\del z')$ is $3$-connected by Bixby's Lemma.
  But, since $\{x,y,z_1,z'\}$ contains a circuit of $M'$, it follows that $z'\in B^{*}$.
  Thus $z'$ is an $(N,B)$-strong element outside of $\{x,y\}$; a contradiction. 
\end{proof}

Let $(\{x,y\},z_1,\dotsc,z_n,z,Y)$ be a good path of $3$-separations for some element $z \in E(M')$ that is $(N,B)$-robust but not $(N,B)$-strong.
Recall that an element~$z_i \in \{z_1,z_2,\dotsc,z_n\}$ is a guts or coguts element according to whether $z_i$ is in the guts or coguts of the $3$-separation $(\{x,y,z_1,\dotsc,z_{i-1}\}, \{z_i,\dotsc,z_n,z\} \cup Y)$. 
Similarly, $z$ is a guts or coguts element depending on whether $z$ is in the guts or coguts of the $3$-separation $(\{x,y,z_1,\dotsc,z_n\}, z \cup Y)$. 

\begin{lemma}
  \label{gorcog}
  Suppose $M'$ has an element~$z$ that is $(N,B)$-robust but not $(N,B)$-strong, and let $(\{x,y\},z_1,\dotsc,z_n,z,Y)$ be a good path of $3$-separations for $z$.
  Let $z'\in \{z_1,\dotsc,z_n,z\}$. Then $z'$ is a guts element if and only if $z'\in B$. 
\end{lemma}

\begin{proof}
Suppose that $z'$ is a guts element.
Then, by \cref{nottrianglejoin}(i), $z' \neq z_1$.
If $z'$ is not $N$-deletable, then $z'\in B$ because $z'$ is an $(N,B)$-robust element.
Thus we may assume that $z'$ is $N$-deletable.
Since $z'$ is in the guts of a vertical $3$-separation, $\co(M'\del z')$ is $3$-connected by Bixby's Lemma, so $z' \in B$ because no element in $\{z_2,\dotsc,z_n,z\}$ is $(N,B)$-strong in $M'$.

Conversely, suppose that $z'$ is a coguts element.
If $z'=z_1$, then $z' \in B^*$ by \cref{nottrianglejoin}(ii).
So we may assume that $z'\in \{z_2,\dotsc,z_n,z\}$.
If $z'$ is not $N$-contractible, then $z'\in B^{*}$ because $z'$ is an $(N,B)$-robust element. Thus we may assume that $z'$ is $N$-contractible. Now $z'$ is in the coguts of a cyclic $3$-separation of $M'$, so Bixby's Lemma implies that $\si(M'/z')$ is $3$-connected.
Thus $z'\in B^{*}$ because no element in $\{z_2,\dotsc,z_n,z\}$ is $(N,B)$-strong in $M'$. 
\end{proof}

\begin{lemma}
\label{4closed}
Suppose $M'$ has elements $z$ and $z'$ that are $(N,B)$-robust but not $(N,B)$-strong, and let $(\{x,y\},z_1,\dotsc,z_n,z,Y)$ and $(\{x,y\},z_1',\dotsc,z'_{n'},z',Y')$ be good paths of $3$-separations for $z$ and $z'$ respectively.
Let $z_{n+1}=z$ and $z'_{n'+1}=z'$.
Then 
\begin{itemize}
  \item[(i)] $z_1 = z_1'$, and
  \item[(ii)] $z_2 = z_2'$, where $z_2 \in B$.
\end{itemize}
Moreover, $\{x,y,z_1,z_2\}$ is closed in $M'$.
\end{lemma}
\begin{proof}
  By \cref{nottrianglejoin}, $\{x,y,z_1\}$ and $\{x,y,z_1'\}$ are triads of $M'$, and $z_1,z_1'\in B^{*}$.  Since $M'$ has no confining sets, $z_1 = z_1'$.

  Consider the element $z_2$. 
  Suppose $z_2\notin B$. Then $z_2$ is a coguts element by \cref{gorcog}, so $\{x,y,z_1,z_2\}$ is a $4$-element cosegment, contradicting that $M'$ has no confining set. Thus $z_2\in B$. 
  Similarly, $z_2' \in B$.

  By \cref{gorcog}, $z_2$ and $z_2'$ are spanned by the triad $\{x,y,z_1\}$.
  Now it suffices to show that $\cl(\{x,y,z_1\}) = \{x,y,z_1,z_2\}$.
  Towards a contradiction,
  suppose there is some $z''\in \cl_{M'}(\{x,y,z_1\}) - \{x,y,z_1,z_2\}$.
  Then, as $\{x,y,z_1\}$ and $\{x,y,z_1,z''\}$ are exactly $3$-separating, it follows that $z'' \notin \cocl_{M'}(\{x,y,z_1\})$, by \cref{calc2,orthogpartition}.
  Since $z_2 \in B$ is $(N,B)$-robust but not $(N,B)$-strong, $r(E(M')-\{x,y,z_1,z''\}) \ge 3$.
  Thus $(\{x,y,z_1\},z'',E(M')-\{x,y,z_1,z''\})$ is a vertical $3$-separation of $M'$, so $\si(M'/z'')$ is not $3$-connected.
  Thus $\co(M'\del z'')$ is $3$-connected by Bixby's Lemma.
  Moreover, $(\{x,y,z_1,z''\},z_2,E(M')-\{x,y,z_1,z_2,z''\})$ is a vertical $3$-separation of $M'$, so \cref{CPL2}(ii) implies that $z''$ is $N$-deletable.
  Now, if $z''\in B^{*}$, then $z''$ is an $(N,B)$-strong element; a contradiction.
  Thus $z''\in B$.
  But then the rank-$3$ set $\cl_{M'}(\{x,y,z_1\})$ contains $\{x,y,z_2,z''\}$, and $\{x,y,z_2,z''\} \subseteq B$; a contradiction.
\end{proof}

We need the following result of Whittle and Williams.

\begin{lemma}[{\cite[Lemma 2.13]{whittle2013fixed}}]
\label{triadin4circuit}
Let $M_0$ be a $3$-connected matroid with a triad $\{c,d,e\}$ and circuit $\{c,d,e,f\}$. Then at least one of the following holds:
\begin{itemize}
 \item[(i)] either $\co(M_0\del c)$ or $\co(M_0\del e)$ is $3$-connected,
 \item[(ii)] there exist $c',e'\in E(M_0)$ such that both $\{c,c',d\}$ and $\{d,e,e'\}$ are triangles, or
 \item[(iii)] there exists $g\in E(M_0)$ such that $\{c,d,e,g\}$ is a $4$-element cosegment.
\end{itemize}
\end{lemma}

We require one more definition.
Let $B$ be a strengthened basis, and let $A$ be a $B\times B^{*}$ companion $\mathbb{P}$-matrix of $M$ such that $\{x,y,a,b\}$ incriminates $(M,A)$, where $\{x,y\}\subseteq B$ and $\{a,b\}\subseteq B^{*}$.
Suppose that $M'$ has no $(N,B)$-strong elements outside of $\{x,y\}$.
We say that $B$ is \emph{bolstered} if for any $B_1\times B_1^{*}$ companion $\mathbb{P}$-matrix~$A_1$ where $\{x_1,y_1,a,b\}$ incriminates $(M,A_1)$, with $\{x_1,y_1\}\subseteq B_1$ and $\{a,b\}\subseteq B_1^{*}$, the number of $(N,B)$-robust elements of $M'$ outside of $\{x,y\}$ is at least the number of $(N,B_1)$-robust elements of $M'$ outside of $\{x_1,y_1\}$.
When $M'$ has an $(N,B)$-strong element $u$ of $M'$ outside of $\{x,y\}$, where $u \in B^*$ by \cref{nostrongbasis}, then we say that $B$ is \emph{bolstered} if for any $B_1\times B_1^{*}$ companion $\mathbb{P}$-matrix~$A_1$ where $\{x,y,a,b\}$ incriminates $(M,A_1)$, with $\{x,y\}\subseteq B_1$ and $\{u,a,b\}\subseteq B_1^{*}$, the number of $(N,B)$-robust elements of $M'$ %outside of $\{x,y\}$
is at least the number of $(N,B_1)$-robust elements of $M'$. % outside of $\{x,y\}$.
(Loosely speaking, a basis $B$ is bolstered if no allowable pivot increases the number of $(N,B)$-robust elements.) %outside of the rows labelling the incriminating set
Note that for any strengthened basis~$B$, we can perform allowable pivots in order to obtain a bolstered basis~$B'$, where $B'$ is also strengthened.

We can now show that either $M$ is bounded relative to $N$, or $M'$ has at most two elements outside of $\{x,y\}$ that are $(N,B)$-robust but not $(N,B)$-strong. %and we can describe the local structure of these elements.
Recall that when $F$ is a $4$-element fan with ordering $(f_1,f_2,f_3,f_4)$ such that $\{f_1,f_2,f_3\}$ is a triangle, and $B'$ is a basis,
we say that $(f_1,f_2,f_3,f_4)$ is a \textit{type-II fan relative to $B'$} if $F\cap B' = \{f_1,f_3,f_4\}$.

\begin{lemma}
  \label{atleasttwobound}
  Suppose $B$ is a bolstered basis.
  Then there are at most two elements outside of $\{x,y\}$ that are $(N,B)$-robust but not $(N,B)$-strong in $M'$.
  Moreover, either
  \begin{itemize}
    \item[(i)] $M'$ has a maximal type-II fan $(z,u,x,y)$ relative to $B$, where $u$ is $(N,B)$-strong, and $z$ is the only element outside of $\{x,y\}$ that is $(N,B)$-robust but not $(N,B)$-strong;
    \item[(ii)] $M'$ has a maximal fan $(w,z,x,u,y)$, where $u$ is $(N,B)$-strong, $z \in B$ is $N$-flexible, % but not $(N,B)$-strong,
      $w \in B^*$, and the elements outside of $\{x,y\}$ that are $(N,B)$-robust but not $(N,B)$-strong are contained in $\{z,w\}$; or
    \item[(iii)] $|E(M)|\leq |E(N)|+7$.
  \end{itemize}
\end{lemma}

\begin{proof}
  Suppose that $z$ is $(N,B)$-robust but not $(N,B)$-strong, let $(\{x,y\},z_1,z_2,\dotsc,z_n,z,Y)$ be a good path of $3$-separations for $z$, where $n \ge 1$, and let $z_{n+1}=z$.
  By \cref{4closed,gorcog,nottrianglejoin}, $\{x,y,z_1\}$ is a triad with $z_1 \in B^*$, and $z_2 \in B$ is a guts element.
  Observe that $\{x,y,z_1,z_2\}$ is either a $4$-element fan or a circuit of $M'$. 
  In the case that $\{x,y,z_1,z_2\}$ is a $4$-element fan, $\{z_2,x,y\}$ is not a triangle, since $z_2 \in B$, so we may assume, up to swapping $x$ and $y$, that this fan has ordering $(z_2,z_1,x,y)$, where $\{z_2,z_1,x\}$ is a triangle.

  \begin{claim}
    \label{firstsl}
    Suppose $(z_2,z_1,x,y)$ is a maximal fan and $\co(M'\ba x)$ is not $3$-connected. Then case (i) of the lemma holds, with $z=z_2$ and $u=z_1$.
  \end{claim}
  \begin{subproof}
%Suppose $(z_2,z_1,x,y)$ is maximal.
By \cref{f2f3}, $\si(M'/z_1)\cong \co(M'\del x)$, so $\si(M'/z_1)$ is not $3$-connected.
As $z_1$ is $(N,B)$-robust, Bixby's Lemma implies that $z_1$ is $(N,B)$-strong.
We claim that $z_2$ is the only element outside of $\{x,y\}$ that is $(N,B)$-robust but not $(N,B)$-strong.
Suppose $z' \in E(M')-\{x,y,z_1,z_2\}$ is $(N,B)$-robust but not $(N,B)$-strong, and let $(\{x,y\},z'_1,\dotsc,z'_{n'},z',Y')$ be a good path of $3$-separations for $z'$.
Let $z' = z'_{n'+1}$.
By \cref{4closed}, $z_1' = z_1$, $z_2' = z_2$, $n' \ge 2$, and $z'_3$ is an $(N,B)$-robust coguts element. % (where if $n' = 2$, then $z'_3 = z'$).
By \cref{gorcog}, $z'_3 \in B^*$.
Now $z'_3$ is in a cocircuit~$C^*$ contained in $\{x,y,z_1,z_2,z'_3\}$.
Since $\{x,y,z_1\}$ is a triad, $\{x,y,z_1\} \nsubseteq C^*$.
Moreover, if $y \in C^*$, then by cocircuit elimination with $\{x,y,z_1\}$, there is a cocircuit contained in $\{x,z_1,z_2,z'_3\}$, and this cocircuit also contains $z'_3$.
So we may assume that $y \notin C^*$.
Since $M'$ has no confining sets, $z_2 \in C^*$.
Since $(z_2,z_1,x,y)$ is maximal, neither $\{x,z_2,z'_3\}$ nor $\{z_1,z_2,z'_3\}$ is a triad.
So $\{x,z_1,z_2,z'_3\}$ is a cocircuit of $M'$.
Now $\{z_2,z_1,x\}$ is not contained in a $4$-element segment by orthogonality, and $\{x,z_2\}$ is not contained in a triad because $(z_2,z_1,x,y)$ is maximal.
Therefore, by the dual of \cref{triadin4circuit}, either $\si(M'/z_2)$ or $\si(M'/z_1)$ is $3$-connected.
But $\si(M'/z_2)$ is not $3$-connected because $z_2$ is not $(N,B)$-strong, and $\si(M'/z_1)\cong \co(M'\del x)$ is not $3$-connected; a contradiction.
We deduce that $z_2$ is the only element outside of $\{x,y\}$ that is $(N,B)$-robust but not $(N,B)$-strong.
  \end{subproof}

\begin{claim}
  \label{secondsl}
  Suppose $\{z_2,z_1,x,y\}$ is contained in a fan $(w,z_2,x,z_1,y)$, for some $w \in E(M')-\{x,y,z_1,z_2\}$, and $\co(M'\ba x)$ is not $3$-connected. Then case (ii) of the lemma holds, with $z=z_2$ and $u=z_1$.
\end{claim}

\begin{subproof}
  As $z_1 \in B^*$ is $(N,B)$-robust, and $\co(M' \ba z_1)$ is $3$-connected by \cref{fanends}, $z_1$ is $(N,B)$-strong.
  Since $z_1$ is $N$-deletable, and $\{x,y\}$ is a series pair in $M'\ba z_1$, it follows that $x$ is $N$-contractible.
  Similarly, since $x$ is $N$-contractible, $z_2$ is $N$-deletable.  As $z_2 \in B$ is $(N,B)$-robust, $z_2$ is $N$-flexible.
  Moreover, as $z_2$ is $N$-deletable, $w$ is $N$-contractible.
  Now, if $w \in B$, then $w$ is $(N,B)$-strong by \cref{fanends}; so $w \in B^*$.

  Next we show that the fan $(w,z_2,x,z_1,y)$ is maximal.
  First, observe that $\{z_1,y\}$ is not contained in a triangle by \cref{4closed}.
  Suppose $\{z_2,w\}$ is contained in a triangle $\{z_2,w,z'\}$ say.
  Since $z_2$ is $N$-contractible, it follows that $z'$ is $N$-deletable.
  By \cref{fanends}, $\co(M\ba z')$ is $3$-connected, so, as the $(N,B)$-strong elements are contained in $\{x,y,z_1\}$, %$z'$ is not $(N,B)$-strong,
  we have $z' \in B$.
  Now, as $\{z_2,w,z'\}$ is a triangle, $A_{xw}=A_{yw}=0$ and $A_{z_2w} \neq 0$, so a pivot on $A_{z_2w}$ is allowable.
  But then $B'= B \triangle \{z_2,w\}$ is a basis,
  and $z_2$ is an $(N,B')$-strong element in $B'-\{x,y\}$, contradicting \cref{nostrongbasis}.

  Now, if the elements outside of $\{x,y\}$ that are $(N,B)$-robust but not $(N,B)$-strong are contained in $\{z_2,w\}$, then \cref{secondsl} holds.

  Suppose there is some $N$-contractible element $w' \in \cocl(\{x,y,z_1,z_2\})-\{x,y,z_1,z_2,w\}$.
  Then $\{y,w,w'\}$ is in the coclosure of the $3$-separating triangle $\{x,z_1,z_2\}$, so, as $M'$ is $3$-connected, $\{y,w,w'\}$ is a triad.
  Recall that $w \in B^*$.
  By Bixby's Lemma, $\si(M'/w')$ is $3$-connected.
  Since $w'$ is $N$-contractible, and the $(N,B)$-strong elements of $M'$ are contained in $\{x,y,z_1\}$, it follows that $w' \in B^*$.
  Now $\{x,y,z_1,w,w'\}$ is a confining set; a contradiction.

  Suppose there is some $w' \in E(M')-\{x,y,z_1,z_2,w\}$ that is $(N,B)$-robust.
  Let $(\{x,y\},z_1',z_2',\dotsc,z'_{n'},w',Y')$ be a good path of $3$-separations for $w'$, and let $z'_{n'+1}=w'$.
  It follows from \cref{4closed,gorcog} and the preceding paragraph that $z_1'=z_1$, $z_2'=z_2$, $z_3'=w$, and $z_4'$ is a guts element, so $z_4' \in B$.
  We work towards a contradiction.
  %Note also that if $n' \ge 4$, %$z_4' \neq w'$, then $z_5'$ is a coguts element, otherwise $|\cl(F) \cap B| = 5$ but $r(F) = 4$.

  We first claim that $\{y,z_1,z_2,w,z_4'\}$ is a circuit of $M'$.
  Certainly, $z_4'$ is in a circuit contained in $\{x,y,z_1,z_2,w,z_4'\}$.
  If this circuit contains $x$, then, by circuit elimination with the triangle $\{z_2,x,z_1\}$, there is a circuit contained in $\{y,z_1,z_2,w,z_4'\}$.
  So we may assume that there is a circuit~$C$ contained in $\{y,z_1,z_2,w,z_4'\}$, which may or may not contain $z_4'$.
  By orthogonality with the triad $\{w,z_2,x\}$, either $C$ contains $\{w,z_2\}$ or $C \cap \{w,z_2\} = \emptyset$.  But in the latter case, $\{z_1,y,z_4'\}$ is a triangle of $M'$, contradicting the maximality of the fan $(w,z_2,x,z_1,y)$.
  So $C$ contains $\{w,z_2\}$ and, similarly, $\{z_1,y\}$.
  Finally, if $z_4' \notin C$, then $\{w,z_2,x,z_1,y\}$ is $2$-separating; a contradiction.  This proves the claim.

  By orthogonality, the only triads containing $x$ are $\{w,z_2,x\}$ and $\{x,z_1,y\}$, so $\co(M'\ba x) \cong M'\ba x / z_1,z_2$.
  Let $M'' = M'\ba x / z_1,z_2$.
  As $\{w,z_2,z_1,y,z_4'\}$ is a circuit of $M'$, the set $T=\{w,z_4',y\}$ is a triangle of $M''$.
  Let $(P,Q)$ be a $2$-separation of $M''$, where $|P \cap T| \ge 2$.
  Now $(\fcl_{M''}(P),Q-\fcl_{M''}(P))$ is also a $2$-separation of $M''$, so we may also assume, without loss of generality, that $P$ is fully closed.
  In particular, $T \subseteq P$.
  Since $\{z_1,z_2\} \subseteq \cocl_{M' \ba x}(P)$, we have that $(P \cup \{z_1,z_2\},Q)$ is a $2$-separation in $M' \ba x$.
  As $x \in \cl(P \cup \{z_1,z_2\})$, it follows that $(P \cup \{z_1,z_2,x\}, Q)$ is a $2$-separation of $M'$; a contradiction.

  We deduce that no $w' \in E(M')-\{x,y,z_1,z_2,w\}$ is $(N,B)$-robust, so the elements of $M'$ that are $(N,B)$-robust but not $(N,B)$-strong are contained in $\{z_2,w\}$, as required.
\end{subproof}

Suppose that $(z_2,z_1,x,y)$ is a $4$-element fan, and $\co(M'\del x)$ is not $3$-connected.
If $\{z_2,z_1\}$ is contained in a triad, then $x$ is a spoke end of a $4$-element fan, contradicting \cref{fanends}; whereas if $y$ is in a triangle, then, by orthogonality, this contradicts \cref{4closed}.
Thus either the fan $(z_2,z_1,x,y)$ is maximal, in which case (i) holds by \cref{firstsl}; or it is contained in a fan $(w,z_2,x,z_1,y)$ for some $w \in E(M')-\{x,y,z_1,z_2\}$, in which case (ii) holds by \cref{secondsl}.
So we may assume that when $(z_2,z_1,x,y)$ is a $4$-element fan, $\co(M'\del x)$ is $3$-connected.

Now consider the case where $\{x,y,z_1,z_2\}$ is a circuit.
Suppose $\{x,y,z_1\}$ is contained in a $4$-element cosegment $\{x,y,z_1,f\}$.
Then, as $M'$ has no confining sets, $f \in B$.  Since $z_1 \in B^*$ is $(N,B)$-robust, it follows that $f$ is $N$-contractible.  But $\si(M'/f)$ is $3$-connected by the dual of \cref{longline3conn}, so $f$ is $(N,B)$-strong; a contradiction.
Now, since $\{x,y,z_1,z_2\}$ is closed by \cref{4closed}, it follows from \cref{triadin4circuit} that either $\co(M'\del x)$ or $\co(M'\del y)$ is $3$-connected.
Thus, when $\{x,y,z_1,z_2\}$ is a circuit, we may assume without loss of generality that $\co(M'\del x)$ is $3$-connected.

Now, in either case, we may assume that $\co(M' \ba x)$ is $3$-connected.

\begin{claim}
  \label{boundingab}
  $A_{pa}=A_{pb}=0$ for all $p\in B-\{x,y,z_2\}$, and $z_1$ is $(N,B)$-strong.
\end{claim}
\begin{subproof}
  Either $\{x,y,z_1,z_2\}$ is a circuit, or this set is a $4$-element fan containing the triad $\{x,z_1,z_2\}$.
  If $x$ is in a $4$-element cosegment of $M'$, then, by orthogonality, this cosegment intersects the circuit $\{x,y,z_1,z_2\}$ or $\{x,z_1,z_2\}$ in three elements.
  But this implies that $\{x,y,z_1\}$ is contained in a $4$-element cosegment; a contradiction.
  So $M' \del x$ is $3$-connected up to series pairs.
  Similarly, $z_1$ is not in a $4$-element cosegment of $M'$.

  Next we claim that $x$ is $N$-deletable.
  Observe that $(\{x,y,z_1\},z_2,E(M')-\{x,y,z_1,z_2\})$ is a vertical $3$-separation, where $M'/z_2$ has an $N$-minor since $z_2 \in B$ is $(N,B)$-robust.
  Since $\{x,y,z_1\}$ is a triad, $E(M')-\{x,y,z_1\}$ is closed.
  Moreover, $x \in \cl(\{y,z_1,z_2\})$, so $x \notin \cocl_{M'}(E(M') - \{x,y,z_1,z_2\})$.
  Thus, by \cref{CPL2}(ii), the element $x$ is $N$-deletable.

  We work towards showing that $a,b \in \cl_M(\{x,y,z_2\})$.
  Observe that $A_{xz_1}\neq 0$ because $\{x,y,z_1,z_2\}$ is a circuit with $\{x,y,z_2\} \subseteq B$. So a pivot on $A_{xz_1}$ is allowable.
  Now $\{z_1,y,a,b\}$ incriminates $(M,A^{xz_1})$.
  Let $B'=B\triangle \{x,z_1\}$.
  Then $x$ is an $(N,B')$-strong element outside of $\{z_1,y\}$.
  By \cref{unstablemeetsxy}, $M' \del x$ has a series pair that meets $\{z_1,y\}$ and is contained in an unstable triple of $M\ba a,x$ or $M\ba b,x$.
  Since $\{z_1,y\}$ is a series class of $M' \ba x$, the series pair $\{z_1,y\}$ is contained in an unstable triple $\{z_1,y,b\}$ of $M \ba a,x$, up to swapping labels on $a$ and $b$.
  So $b\in \cl_M(\{z_1,y\})$.
  Since $z_1\in \cl(\{x,y,z_2\})$, it follows that $b\in \cl_M(\{x,y,z_2\})$.

  As $B$ is a strengthened basis and $x$ is an $(N,B')$-strong element outside of $\{z_1,y\}$, it follows that $z_1$ is an $(N,B)$-strong element outside of $\{x,y\}$.
  Since $z_1$ is not in a $4$-element cosegment of $M'$, the matroid $M' \ba z_1$ is $3$-connected up to series pairs.
  Thus, by \cref{unstablemeetsxy} again, $M' \del z_1$ has a series pair that meets $\{x,y\}$ and is contained in an unstable triple of $M\ba a,z_1$ or $M\ba b,z_1$.
  It now follows that this series pair is $\{x,y\}$, so either $a \in \cl_M(\{x,y\})$ or $b \in \cl_M(\{x,y\})$.
  But in the latter case, $\{x,y,z_1\}$ is a triangle, contradicting \cref{nottrianglejoin}(i).
  So $a\in \cl_M(\{x,y\})$.
  Now $a,b\in \cl_M(\{x,y,z_2\})$, so $A_{pa}=A_{pb}=0$ for all $p\in B-\{x,y,z_2\}$.
  Thus \cref{boundingab} holds.
\end{subproof}

\begin{claim}
  \label{anotherclaim}
  There are no $N$-contractible elements of $M'$ outside of $\{x,y,z_1,z_2\}$.
\end{claim}
\begin{subproof}
  Suppose that $M'$ has an $N$-flexible element $q\in B^{*}-z_1$.
  Then, since $q$ is not $(N,B)$-strong in $M'$, the matroid $\co(M'\del q)$ is not $3$-connected.
  Hence $\si(M'/q)$ is $3$-connected by Bixby's Lemma.
  By \cref{4closed}, $q\notin \cl(\{x,y,z_1\})=\cl(\{x,y,z_2\})$, so $A_{pq}\neq 0$ for some $p\in B-\{x,y,z_2\}$.
  Since $A_{pa}=A_{pb}=0$, by \cref{boundingab}, a pivot on $A_{pq}$ is allowable.
  But then $B'=B\triangle \{p,q\}$ has an $(N,B')$-strong element~$q$ in $B'-\{x,y\}$, contradicting \cref{nostrongbasis}.
  Thus $M'$ has no $N$-flexible elements in $B^{*}-z_1$. 

  Suppose that $M'$ has an $(N,B)$-robust element $p\in B-\{x,y,z_2\}$.
  Let $(\{x,y\},z_1',\dotsc,z'_{n'},p,Y')$ be a good path of $3$-separations for $p$.
  By \cref{4closed}, $z_1'=z_1$, $z_2'=z_2$, and $\{x,y,z_1,z_2\}$ is closed.
  Hence, as $p$ is a guts element by \cref{gorcog}, $n' \ge 3$.
  But then $z_3' \in B^*-z_1$ is $N$-flexible; a contradiction.
  Therefore no element in $B-\{x,y,z_2\}$ is $N$-contractible.

  For each element $q\in B^{*}-z_1$, there is some $p\in B-\{x,y,z_2\}$ such that $A_{pq}$ is non-zero, by \cref{4closed}, so a pivot on $A_{pq}$ is allowable.
  Since $B$ is bolstered, there are at least as many $(N,B)$-robust elements as $(N,B\triangle \{p,q\})$-robust elements, so $M'$ has no $N$-contractible elements in $B^{*}-z_1$. 
\end{subproof}

\begin{claim}
  \label{yetanotherclaim}
  There is at most one element outside of $\{x,y\}$ that is $(N,B)$-robust but not $(N,B)$-strong. 
\end{claim}
\begin{subproof}
  By \cref{anotherclaim}, each $(N,B)$-robust element of $M'$ outside of $\{x,y,z_1,z_2\}$ is in $B^*$, so any such element is a coguts element by \cref{gorcog}.
  Suppose $q$ and $q'$ are distinct $(N,B)$-robust elements of $M'$ in $B^*-z_1$.
  Then $(\{x,y\},z_1,z_2,q,Y)$ and $(\{x,y\},z_1,z_2,q',Y')$ are the good paths of $3$-separations for $q$ and $q'$ respectively, otherwise there is an $N$-contractible element in $B^*-z_1$.
  Now $(\{x,y,z_1,z_2,q\},q',Y-q')$ is a cyclic $3$-separation, and $q \in \cocl_{M'}(Y-q')$, so $q \in B^*-z_1$ is $N$-contractible in $M'$ by the dual of \cref{CPL2}(ii); a contradiction.
  Hence there is at most one $(N,B)$-robust element of $M'$ outside of $\{x,y,z_1,z_2\}$. 
\end{subproof}

By \cref{yetanotherclaim}, it now suffices to show that $|E(M')|\leq |E(N)|+5$.
Towards a contradiction, suppose that $|E(M')|\geq |E(N)|+6$.
Let $R$ be the set consisting of $\{x,y,z_1,z_2\}$ and the $(N,B)$-robust element of $M'$ outside of $\{x,y,z_1,z_2\}$, if such an element exists.
So the set of $(N,B)$-robust elements of $M'$ is contained in $R$, where $|R| \le 5$.
Since $|E(M')|\geq |E(N)|+6$, there is an element $p$ outside of $R$ that is either $N$-deletable or $N$-contractible in $M'$, but is not $(N,B)$-robust in $M'$.
By \cref{anotherclaim}, $p\in B-R$ and $p$ is $N$-deletable.
Since $z_1$ is in a circuit of $M'$ contained in $\{x,y,z_1,z_2\}$, it follows from orthogonality that $p\notin \cl_{M'}^{*}(\{z_1,z_3\}) = \cl_{M'}^{*}(R\cap B^{*})$.
Thus there is some $q\in B^{*}-R$ such that $A_{pq}\neq 0$.
Since $A_{pa}=A_{pb}=0$, by \cref{boundingab}, a pivot on $A_{pq}$ is allowable.
Again letting $B'=B\triangle \{p,q\}$, we see there are more $(N,B')$-robust elements than $(N,B)$-robust elements, so $B$ is not bolstered; a contradiction. 
We deduce that $|E(M')| \le |E(N)|+5$, as required.
\end{proof}

\begin{lemma}
  \label{norobustfragile}
  Suppose $B$ is a bolstered basis.
  If $M'$ has no $(N,B)$-robust elements outside of $\{x,y\}$, then $M\del a,b$ is $N$-fragile. 
\end{lemma}

\begin{proof}
  Suppose $M'$ has no $(N,B)$-robust elements outside of $\{x,y\}$.
  Since the elements outside of $\{x,y\}$ are not $(N,B)$-robust, it suffices, by symmetry, to show that $M'\del x$ has no $N$-minor.
  Towards a contradiction, suppose that $x$ is $N$-deletable. There is some $x'\in B^{*}-\{a,b\}$ such that $A_{xx'}\neq 0$ because $x$ is not a coloop of $M'$,
  so a pivot on $A_{xx'}$ is allowable.
  Let $B'=B\triangle \{x,x'\}$.
  Now $x$ is an $(N,B')$-robust element of $M'$ outside of $\{x',y\}$, contradicting the fact that $B$ is bolstered.
  We deduce that $x$ is not $N$-deletable, as required.
\end{proof}

We now prove \cref{mainthm1}, which we restate here for ease of reference.

\begin{thm}
  \label{mainthm1proof}
  Let $M$ be an excluded minor for the class of $\mathbb{P}$-representable matroids, and let $N$ be a non-binary $3$-connected strong $\mathbb{P}$-stabilizer for the class of $\mathbb{P}$-representable matroids.
  Suppose $M$ has a pair of elements $\{a,b\}$ such that $M\ba a,b$ is $3$-connected with an $N$-minor.
  Then either
  \begin{itemize}
    \item[(i)] $|E(M)|\leq |E(N)|+9$, or 
    \item[(ii)] $M$ has a $B\times B^{*}$ companion $\mathbb{P}$-matrix $A$ for which $\{x,y,a,b\}$ incriminates $(M,A)$, where $\{x,y\}\subseteq B$ and $\{a,b\}\subseteq B^{*}$, and either
      \begin{itemize} 
        \item[(a)] $M\del a,b$ is $N$-fragile, and $M\del a,b$ has at most one $(N,B)$-robust element outside of $\{x,y\}$, where if such an element $u$ exists, then $u\in B^{*}-\{a,b\}$ is an $(N,B)$-strong element of $M\del a,b$, and $\{u,x,y\}$ is a coclosed triad of $M\del a,b$, or
        \item[(b)] $M\del a,b$ is not $N$-fragile, but there is an element $u \in B^*-\{a,b\}$ that is $(N,B)$-strong in $M \ba a,b$; either
          \begin{itemize}
            \item[(I)] the $N$-flexible, and $(N,B)$-robust, elements of $M\del a,b$ are contained in $\{u,x,y\}$, or
            \item[(II)] the $N$-flexible, and $(N,B)$-robust, elements of $M\del a,b$ are contained in $\{u,x,y,z\}$, where $z \in B$, and $(z,u,x,y)$ is a maximal fan of $M \del a,b$, or
            \item[(III)]the $N$-flexible, and $(N,B)$-robust, elements of $M\del a,b$ are contained in $\{u,x,y,z,w\}$, where $z \in B$, $w \in B^*$, and $(w,z,x,u,y)$ is a maximal fan of $M \del a,b$;
          \end{itemize}
          the unique triad in $M \ba a,b$ containing $u$ is $\{u,x,y\}$; and $M$ has a cocircuit $\{x,y,u,a,b\}$ and a triangle $\{d,x,y\}$ for some $d\in \{a,b\}$.
      \end{itemize}
      Moreover, %if $M \ba a,b$ has no $(N,B)$-robust elements outside of $\{x,y\}$, then, for any $B_1\times B_1^{*}$ companion $\mathbb{P}$-matrix $A_1$ where $\{x_1,y_1,a,b\}$ incriminates $(M,A_1)$, with $\{x_1,y_1\}\subseteq B_1$ and $\{a,b\}\subseteq B_1^{*}$, the matroid $M \ba a,b$ has no $(N,B_1)$-robust elements outside of $\{x_1,y_1\}$.
      $B$ is a bolstered basis.
  \end{itemize}
\end{thm}

\begin{proof}
It follows from \cref{nogadgetsetup} that $M'$ has either a confining set or a strengthened basis $B$.
If $M'$ has a confining set, then (i) holds by \cref{gadgetfinish}.
Assume that $M'$ has a strengthened basis $B$ and that (i) does not hold, so $|E(M)| \ge |E(N)| + 10$ and $M'$ has no confining sets.
We may assume that the strengthened basis~$B$ is chosen to be bolstered.
If $M'$ has no $(N,B)$-robust elements outside of $\{x,y\}$, then (ii)(a) holds by \cref{norobustfragile}.
We shall therefore assume $M'$ has an $(N,B)$-robust element outside of $\{x,y\}$. 

We distinguish two cases.
First, suppose that all $(N,B)$-robust elements of $M'$ outside of $\{x,y\}$ are $(N,B)$-strong. Then $M'$ has exactly one $(N,B)$-strong element $u$, and $\{u,x,y\}$ is a triad of $M'$ by \cref{nogadgetsetup}.
Since $M'$ has no confining sets, \cref{nostronglongline} implies that $M' \ba u$ is $3$-connected up to series pairs; in particular, the triad $\{u,x,y\}$ is coclosed.
If $M'$ is $N$-fragile, then (ii)(a) holds. Suppose then that $M'$ is not $N$-fragile. Since $N$-flexible elements are $(N,B)$-robust, it follows that the $N$-flexible elements of $M'$ are contained in $\{u,x,y\}$.
To show that (ii)(b)(I) holds, it remains to prove that $\{d,x,y\}$ is a triangle of $M$ for some $d\in \{a,b\}$, the unique triad in $M \ba a,b$ containing $u$ is $\{u,x,y\}$, and $\{x,y,u,a,b\}$ is a cocircuit of $M'$.
Since $M' \ba u$ is $3$-connected up to series pairs, the former follows from \cref{unstablemeetsxy}.
We return to the latter two claims momentarily.

Second, suppose that some $(N,B)$-robust element of $M'$ outside of $\{x,y\}$ is not $(N,B)$-strong.
Since $|E(M)| \ge |E(N)| + 10$, \cref{atleasttwobound} implies that $u$ is $(N,B)$-strong, and either $M'$ has a maximal type-II fan $(z,u,x,y)$ relative to $B$, or $M'$ has a maximal fan $(w,z,x,u,y)$ such that $z \in B$ and $w \in B^*$
, where $\{u,x,y,z\}$ or $\{u,x,y,z,w\}$, respectively, contains all of the $(N,B)$-robust elements of $M'$.
Hence $M'$ is not $N$-fragile, and $\{u,x,y,z\}$, or $\{u,x,y,z,w\}$, contains all of the $N$-flexible elements of $M'$.
\Cref{unstablemeetsxy} implies that $\{d,x,y\}$ is a triangle of $M$ for some $d\in \{a,b\}$.

Now, in either of the two cases, $M'$ has an $(N,B)$-strong element $u$.

\begin{claim}
  \label{below2}
  $\{x,y,u,a,b\}$ is a cocircuit of $M$.
\end{claim}

\begin{subproof}
By orthogonality, $\{x,y\}$ is not contained in a $4$-element segment of $M'$, so there is at most one element in $B^*$ that is in a triangle of $M'$ with $\{x,y\}$.
Thus, as $|E(M)| \ge |E(N)| + 10$, there is either some $p \in B-\{x,y\}$ that is $N$-deletable but not $(N,B)$-robust, or some $q \in B^*-\{a,b,u\}$ that is $N$-contractible but not $(N,B)$-robust such that $\{q,x,y\}$ is not a triangle.
In the former case, as $\{p,a,b\}$ is not a triad of $M$ since $M \ba a,b$ is $3$-connected, we can choose $q \in B^*-\{a,b,u\}$ such that the entry $A_{pq}$ is non-zero.
In the latter case, we can choose $p \in B-\{x,y\}$ so that the entry $A_{pq}$ is non-zero, since $\{q,x,y\}$ is not a triangle.
Now, if $A_{xq}=0$ and $A_{yq}=0$, then the pivot on $A_{pq}$ is allowable, in which case $B \triangle \{p,q\}$ is a basis, and there are more $(N,B \triangle \{p,q\})$-robust elements than $(N,B)$-robust elements in $M'$, contradicting the fact that $B$ is bolstered.
So we may assume that $A_{yq}\neq 0$. % for some $q\in B^{*}-\{a,b,u\}$.
Now a pivot on $A_{yq}$ is allowable, so $A^{yq}$ is a companion $\mathbb{P}$-matrix where $\{x,q,a,b\}$ incriminates $(M,A^{yq})$.
If $\{b,u,x,y\}$ is a cocircuit of $M$, then 
$(A^{yq})_{xa}=0$ because $\{b,y,u\}$ cospans $x$, contradicting that the bad submatrix $A^{yq}[\{x,q,a,b\}]$ has no zero entries.
So $\{b,u,x,y\}$, and similarly $\{a,u,x,y\}$, are not cocircuits of $M$.
Therefore $\{x,y,u,a,b\}$ is a cocircuit of $M$. % that contains the triangle $\{b,x,y\}$,
%so (ii)(b)(I) holds.
\end{subproof}

%We make the following observation about the series pairs of $M'\del u$. 

\begin{claim}
\label{newboy}
$\{x,y\}$ is the only series pair of $M'\del u$.
\end{claim}

\begin{subproof}
  Suppose $\{p,q\}$ is a series pair of $M' \ba u$ that is distinct from $\{x,y\}$.
  Since $\{u,x,y\}$ is a coclosed triad of $M'$, the pairs $\{x,y\}$ and $\{p,q\}$ are not contained in the same series class of $M' \ba u$; in particular, they are disjoint.
  As $\{u,p,q\}$ is a triad of $M'$ and $u$ is an $N$-deletable element in $B^*$, both $p$ and $q$ are $N$-contractible in $M'$, and at least one of $p$ and $q$ is in $B-\{x,y\}$.
  So $p$, say, is an $(N,B)$-robust element in $B-\{x,y\}$.
  Since $\{u,p,q\}$ is a triad of $M'$, for some $q \in E(M')-\{u,p,x,y\}$, it now follows that we are in the case where (ii)(b)(III) holds.
  Now $M'$ has a $5$-element fan~$F$ with ordering $(w,p,x,u,y)$, where $q \notin F$ and $q$ is $N$-contractible.
  Since $q$ is not $(N,B)$-robust, $q \in B^*$.
  Moreover, as $\{y,w,q\} \subseteq \cocl_{M'}(\{u,x,p\})-\{u,x,p\}$, where $\{u,x,p\}$ is $3$-separating, it follows that $\{y,w,q\}$ is a triad of $M'$.
  Now $\{x,y,u,p,q\}$ is a confining set; %, so $|E(M)| \le |E(N)| + 9$ by \cref{gadgetfinish};
  a contradiction.
\end{subproof}

Finally, either (I), (II), or (III) of (ii)(b) holds, by \cref{below2,newboy}.
\end{proof}

\section{Spike-like $3$-separators}
\label{secdetachable}

Suppose that $M$ is an excluded minor for the class of $\mathbb{P}$-representable matroids, for some partial field $\mathbb{P}$, with a minor $N$ where $N$ is a $3$-connected strong $\mathbb{P}$-stabilizer. % that is non-binary.
By \cref{detachthm}, if $M$ has no spike-like $3$-separator, then, after replacing $M$ by a \dY-equivalent matroid, and possibly dualising, we obtain a matroid with a deletion pair with respect to $N$ or $N^*$.
In this section, we show that in the case that $M$ has a spike-like $3$-separator, $|E(M)|$ is bounded relative to $|E(N)|$.

We require the following lemma which shows, in particular, that an element that is in a quad but not in a triangle (or, dually, a triad) can be contracted (or deleted, respectively) without destroying $3$-connectivity.

\begin{lemma}[{\cite[Lemma 3.8]{stabilizers}}]
  \label{r3cocircsi3}
  Let $C^*$ be a rank-$3$ cocircuit of a $3$-connected matroid $M$.
If $x \in C^*$ has the property that $\cl_M(C^*)-x$ contains a triangle of $M/x$, then $\si(M/x)$ is $3$-connected.
\end{lemma}

\begin{lemma}
  \label{spikelikes}
  Let $\mathbb{P}$ be a partial field, let $N$ be a non-binary $3$-connected strong stabilizer for the class of $\mathbb{P}$-representable matroids, and let $M$ be an excluded minor for the class of $\mathbb{P}$-representable matroids, where $M$ has an $N$-minor.
  If $M$ has a spike-like $3$-separator $P$ such that at most one element of $E(M)-E(N)$ is not in $P$, then $|E(M)| \le |E(N)| + 5$.
\end{lemma}
\begin{proof}
  Towards a contradiction, suppose that $|E(M)| \ge |E(N)| + 6$.
  By the definition of a spike-like $3$-separator, there is a partition $\{L_1,\dotsc,L_t\}$ of $P$ such that $|L_i|=2$ for each $i\in\{1,\dotsc,t\}$, and $L_i\cup L_j$ is a quad for all distinct $i,j\in\{1,\dotsc,t\}$, where $t \ge 3$.
  Since at most one element of $E(M)-E(N)$ is not in $P$, we have $|P-E(N)| \ge 5$.  %In particular, $t \ge 3$.

  Up to possibly replacing $(M,N)$ with $(M^*,N^*)$, there are distinct elements $a,b \in P$ such that $\{a,b\}$ is $N$-deletable,
  $a \in L_i$, and $b \in L_j$, with $i \neq j$.
  It follows from orthogonality, and the fact that $i \neq j$ and $t \ge 3$, that if $\{a,b\}$ is contained in a triad, then this triad meets $L_{i'}$ for each $i' \in \{1,\dotsc,t\}$.
  But then $t=3$ and $r^*(P) = 3$, implying $\lambda(P)=1$; a contradiction.
  Thus, by the dual of \cref{r3cocircsi3}, $M \ba a$ and $M \ba b$ are $3$-connected, and $M \ba a,b$ is $3$-connected up to series classes.
  Thus $\{a,b\}$ is a weak deletion pair.
  By \cref{companion,incrim4}, there exists a $B \times B^*$ companion $\mathbb{P}$-matrix~$A$ with $\{x,y\} \subseteq B$ and $\{a,b\} \subseteq B^*$ such that $\{x,y,a,b\}$ incriminates $(M,A)$.

  Since $L_i \cup L_j$ is a cocircuit, there is some $u \in (L_i \cup L_j) \cap B$.  As $u$ is in a series pair of $M \ba a,b$, the element $u$ is $N$-contractible in $M \ba a,b$, and $M \ba a,b /u$ is $3$-connected up to series classes.
  Without loss of generality, we may assume that $u \in L_i$.
  By the definition of a spike-like $3$-separator, $L_i = \{a,u\}$ is not contained in a triangle.
  Thus, if $u$ is in a triangle, then,
  by orthogonality with the cocircuits $L_{i} \cup L_{j'}$ for $j' \in \{1,\dotsc,t\}-i$, this triangle meets each $L_{j'}$.
  But then $t=3$ and $r(P) = 3$, implying $\lambda(P)=1$; a contradiction.
  So $M/u$ is $3$-connected by \cref{r3cocircsi3}.
  It now follows that $\co(M \ba a /u)$ and $\co(M \ba b /u)$ are $3$-connected.
  In particular, $M \ba a /u$ and $M \ba b /u$ are $N$-stable, and $M \ba a,b /u$ is connected.
  Thus, by \cref{notrepcert2}, $M/u$ is not strongly $\mathbb{P}$-stabilized by $N$.
  But, as $M/u$ is $3$-connected, and hence $N$-stable, this contradicts \cref{stableisstable}.
\end{proof}

The following is a consequence of \cref{spikelikes,detachthm}.

\begin{cor}
  \label{detachsetup}
Let $\mathbb{P}$ be a partial field, let $M$ be an excluded minor for the class of $\mathbb{P}$-representable matroids, and let $N$ be a non-binary $3$-connected strong stabilizer for the class of $\mathbb{P}$-representable matroids, where $M$ has an $N$-minor.
Suppose that $|E(M)| \ge |E(N)| + 10$.
Then, there exists a matroid~$M_0$, where $M_0$ is obtained from $M$ by at most one \dY\ or \Yd\ exchange, and $(M_1,N_1) \in \{(M_0,N), (M_0^*,N^*)\}$ such that $M_1$ has a pair of elements $\{a,b\}$ for which $M_1 \ba a,b$ is $3$-connected and has an $N_1$-minor.
\end{cor}

\section{Proof of \cref{mainthm2}}
\label{mainthmsec}

Let $M$ be an excluded minor for the class of $\mathbb{P}$-representable matroids, for some partial field~$\mathbb{P}$, and let $N$ be a non-binary $3$-connected strong $\mathbb{P}$-stabilizer for the class of $\mathbb{P}$-representable matroids. 

In this section we prove \cref{mainthm2}.
We first address a few more cases where we can bound $|E(M)|$ relative to $|E(N)|$.

\begin{lemma}
\label{no4line}
Suppose $M$ has a pair of elements $\{a,b\}$ such that $M \ba a,b$ is $3$-connected with an $N$-minor.
If (ii)(b) of \cref{mainthm1proof} holds, and $\{a,b\}\subseteq \cl_M(\{x,y\})$, then $|E(M)|\leq |E(N)|+7$.
\end{lemma}

\begin{proof}
Suppose that (ii)(b) of \cref{mainthm1proof} holds, and $\{a,b\}\subseteq \cl_M(\{x,y\})$, but $|E(M)|\geq |E(N)|+8$.
Then there is at least one element in $E(M')-\{x,y\}$ that is $N$-deletable or $N$-contractible in $M'$ but not $(N,B)$-robust, where $B$ is a bolstered basis.

Suppose that $p$ is $N$-deletable but not $(N,B)$-robust. Then $p\in B-\{x,y\}$.
Now $A_{pa}=A_{pb}=0$ because $\{a,b\}\subseteq \cl_M(\{x,y\})$.
We claim that there is some element $q\in B^{*}-\{a,b\}$ that is not $(N,B)$-robust and $A_{pq}\neq 0$.
By \cref{atleasttwobound}, there is a single element $u \in B^*$ that is $(N,B)$-strong in $M'$, and at most one element in $B^*-\{u,a,b\}$ that is $(N,B)$-robust. % but not $(N,B)$-strong.
First consider the case where no element in $B^*-\{u,a,b\}$ is $(N,B)$-robust. % but not $(N,B)$-strong.
Then there is some $q\in B^{*}-\{u,a,b\}$ such that $A_{pq}\neq 0$, because $M'$ has no coloops or series pairs, and $q$ is not $(N,B)$-robust.
Now consider the case where there is an element $w \in B^*-\{u,a,b\}$ that is $(N,B)$-robust. %but not $(N,B)$-strong in $M'$.
Then $(w,z,x,u,y)$ is a $5$-element fan by \cref{atleasttwobound}, and it follows that $\{u,w\}$ is not contained in a triad.
Hence, there is some $q\in B^{*}-\{u,w,a,b\}$ such that $A_{pq}\neq 0$ and $q$ is not $(N,B)$-robust.
Now, in either case, a pivot on $A_{pq}$ is allowable, and $B'=B\triangle \{p,q\}$ is a basis of $M'$ for which there are more $(N,B')$-robust elements than $(N,B)$-robust elements, contradicting that $B$ is a bolstered basis. 

We may now assume that there is an element $q$ that is $N$-contractible but not $(N,B)$-robust in $M'$, so $q \in B^*$.
Since $x$ is in a triad with the $(N,B)$-strong element $u$, it follows that $x$ is $N$-contractible in $M'$.
If $q\in \cl(\{x,y\})$, then, since $\{q,y\}$ is a parallel pair in $M'/x$, it follows that $q$ is $N$-deletable, and hence $(N,B)$-robust, in $M'$; a contradiction.
Thus $q\notin \cl(\{x,y\})$.
Moreover, in the case that there is an $(N,B)$-robust element $z \in B$, as $z$ is not $N$-deletable, it follows that $q \notin \cl(\{x,y,z\})$.
So $A_{pq}\neq 0$ for some element $p\in B-\{x,y\}$ that is not $(N,B)$-robust.
Now a pivot on $A_{pq}$ is allowable, and $B'=B\triangle \{p,q\}$ is a basis for $M'$ such that there are more $(N,B')$-robust elements than $(N,B)$-robust elements, contradicting that $B$ is a bolstered basis.
\end{proof}

\begin{lemma}
\label{ab3spanned}
Suppose $M$ has a pair of elements $\{a,b\}$ such that $M \ba a,b$ is $3$-connected with an $N$-minor.
If (ii)(b) of \cref{mainthm1proof} holds, and there is some $p \in (B-\{x,y\}) \cap \cl(\{u,x,y\})$ such that $\{a,b\}\subseteq \cl_M(\{p,x,y\})$, then $|E(M)|\leq |E(N)|+7$.  
\end{lemma}

\begin{proof}
  Let $R$ be the set consisting of $\{p,x,y,a,b\}$ and the $(N,B)$-robust elements of $M'$ outside of $\{x,y\}$.
  Consider the case where (ii)(b)(II) or (ii)(b)(III) of \cref{mainthm1proof} holds.
  Then $\{u,x,y\}$ is contained in a (not necessarily maximal) $4$-element fan $(z,u,x,y)$, where $z \in B$.
  Since $\{p,x,y,z\} \subseteq B$, but $r(\cl(\{u,x,y\}))=3$, we deduce that $p = z$.
  Thus $|R| \le 7$.
  Towards a contradiction, suppose that $|E(M)|\geq |E(N)|+8$.
  Then $M$ has at least one element outside of $R$ that is either $N$-deletable or $N$-contractible, but not $(N,B)$-robust.
  %Since $R$ contains all of the $(N,B)$-robust elements of $M'$, these elements are not $(N,B)$-robust.

  Suppose first that there is some $p'\in B-\{x,y,p\}$ that is $N$-deletable.
  Then there is an element $q\in B^{*}-R$ such that $A_{p'q}\neq 0$, because $M'$ is $3$-connected and, in the case that \cref{mainthm1proof}(ii)(b)(III) holds, $\{z,w,u\}$ is not a triad.
  Since $\{a,b\}\subseteq \cl_M(\{p,x,y\})$, it follows that $A_{p'a}=A_{p'b}=0$, so a pivot on $A_{p'q}$ is allowable.
  But, with $B'=B\triangle \{p',q\}$, there are more $(N,B')$-robust elements than there are $(N,B)$-robust elements, contradicting that $B$ is bolstered. 

  So $M'$ has an $N$-contractible element $q \in B^{*}-R$.
  Suppose that $q \in \cl(\{x,y,p\})$.
  Then, as $\{u,x,y\}$ is a triad of $M'$, it follows that $(\{q,u,x,y\},p,E(M')-\{p,q,u,x,y\})$ is a vertical $3$-separation of $M'$.
  But then $q$ is $N$-deletable by \cref{CPL2}(ii), contradicting that $q$ is not $(N,B)$-robust.

  Thus we may assume that $q\notin \cl(\{x,y,p\})$, so there is some $p'\in B-\{x,y,p\}$ such that $A_{p'q}\neq 0$.
  Then $A_{p'a}=A_{p'b}=0$, so a pivot on $A_{p'q}$ is allowable.
  But with $B'=B\triangle \{p',q\}$, there are more $(N,B')$-robust elements than there are $(N,B)$-robust elements, contradicting that $B$ is bolstered. 
\end{proof}

We also use the following, which is proved in \cite{brettell2014splitter}.
\begin{lemma}[{\cite[Lemma~3.1]{brettell2014splitter}}]
  \label{r3cocircco}
  Let $M_0$ be a $3$-connected matroid with $r(M_0) \ge 4$.
  Suppose that $C^*$ is a rank-$3$ cocircuit of $M_0$.
  If there exists some $x \in C^*$ such that $x \in \cl(C^*-x)$, then $\co(M_0 \ba x)$ is $3$-connected.
\end{lemma}

We now prove our second main result, \cref{mainthm2}, first restating it.

\begin{thm}
  \label{mainthm2proof}
  Let $M$ be an excluded minor for the class of $\mathbb{P}$-representable matroids, and let $N$ be a non-binary $3$-connected strong $\mathbb{P}$-stabilizer, where $M$ has an $N$-minor.
  For some $M_1$ that is \dY-equivalent to $M$, and some $(M_0,N_0)$ in $\{(M_1,N),(M_1^{*}, N^{*})\}$, the matroid $M_0$ is an excluded minor with an $N_0$-minor, and at least one of the following holds:
  \begin{itemize}
    \item[(i)] $|E(M_0)|\leq |E(N_0)|+9$;
    \item[(ii)] $r(M_0)\leq r(N_0)+7$; or 
    \item[(iii)] there is a pair $\{a,b\} \subseteq E(M)$ such that $M_0\del a,b$ is $3$-connected with an $N_0$-minor, and $M_0\del a,b$ is $N_0$-fragile.
      Moreover, there is some bolstered basis $B$ for $M_0$ and a $B\times B^{*}$ companion $\mathbb{P}$-matrix~$A$ for which $\{x,y,a,b\}$ incriminates $(M,A)$, where $\{x,y\}\subseteq B$, $\{a,b\}\subseteq B^{*}$, and both of the following hold:
      \begin{itemize}
        \item[(a)] $M_0\del a,b$ has at most one $(N_0,B)$-robust element outside of $\{x,y\}$, and
        \item[(b)] if $u$ is an $(N_0,B)$-robust element of $M_0\ba a,b$, then $u\in B^{*}-\{a,b\}$, the element $u$ is $(N_0,B)$-strong in $M_0\del a,b$, and $\{u,x,y\}$ is a triad of $M_0\del a,b$.
      \end{itemize}
  \end{itemize}
\end{thm}

\begin{proof}
  Suppose that neither (i) nor (ii) holds; in particular, $|E(M)| \ge |E(N)| + 10$ and $r^*(M) \ge r^*(N) + 8$.
  By \cref{detachsetup}, there exists a matroid~$M_0$, where $M_0$ is obtained from $M$ by at most one \dY\ or \Yd\ exchange, and $(M_1,N_1) \in \{(M_0,N), (M_0^*,N^*)\}$ such that $M_1$ has a pair of elements $\{a,b\}$ for which $M_1 \ba a,b$ is $3$-connected and has an $N_1$-minor.
  By \cref{osvdelta}, $M_1$ is an excluded minor for the class of $\mathbb{P}$-representable matroids.
  We relabel $(M_1,N_1)$ as $(M,N)$ and apply \cref{mainthm1proof}.
  If (ii)(a) of \cref{mainthm1proof} holds, then (iii) holds.
  We may therefore assume that (ii)(b) of \cref{mainthm1proof} holds.
  Without loss of generality, we may assume that $\{b,x,y\}$ is a triangle of $M$.

Note that $M' = M \ba a,b$ has an element $u \in B^*$ that is $(N,B)$-strong, where $\{u,x,y\}$ is a triad.

\begin{claim}
  \label{ucon}
  The element $u$ is $N$-contractible in $M'$.
\end{claim}
\begin{subproof}
  As $M'$ is not $N$-fragile, $M'$ has at least one $N$-flexible element.
  If $x$ is $N$-deletable, then, as $u$ is in a series pair of $M' \ba x$, the element $u$ is $N$-contractible.
  Similarly, if $y$ is $N$-deletable, then $u$ is $N$-contractible.
  Thus, if the $N$-flexible elements of $M'$ are contained in the triad $\{u,x,y\}$, then, since $M'$ has at least one $N$-flexible element, it follows that $u$ is $N$-contractible in $M'$.
  Next, suppose that $(z,u,x,y)$ is a fan of $M'$, and $z$ is $N$-flexible.  
  As $x$ is in a parallel pair of $M'/z$, the element $x$ is $N$-deletable, so $u$ is $N$-contractible.
  Finally, we may assume that $(w,z,x,u,y)$ is a fan of $M'$, and $w$ is $N$-flexible.
  As $z$ is in a series pair of $M' \ba w$, the element $z$ is $N$-contractible, and it follows that $x$ is $N$-deletable, so $u$ is $N$-contractible.
\end{subproof}

Next, we show that, up to duality and replacing $M$ by a \dY-equivalent matroid, there is some deletion pair that is contained in a triangle.
This triangle will provide additional leverage in later orthogonality arguments.

\begin{claim}
\label{no4cct}
For some $M_2 \in \{M,\nabla_T(M^*)\}$, where $T=\{b,x,y\}$, there is a pair $\{a',b'\} \subseteq E(M_2)$ such that $M_2\del a',b'$ is $3$-connected with an $N$-minor, and $\{a',b'\}$ is contained in a triangle of $M_2$.
\end{claim}

\begin{subproof}
  We first consider the case where $\{a,u\}$ is contained in a triangle with either $x$ or $y$.
  If (ii)(b)(II) or (ii)(b)(III) of \cref{mainthm1proof} holds, then $z \in (B - \{x,y\}) \cap \cl(\{u,x,y\})$, and $\{a,b\} \subseteq \cl_M(\{x,y,z\})$, so $|E(M)| \le |E(N)| + 7$ by \cref{ab3spanned}; a contradiction.
  So we may assume that (ii)(b)(I) of \cref{mainthm1proof} holds.
  Now we have symmetry between $x$ and $y$, so we may assume that $\{a,u,x\}$ is a triangle.

We claim that $\{b,x\}$ is a deletion pair with the desired properties.
Clearly $M\del b$ is $3$-connected and has an $N$-minor.
By \cref{ucon}, $u$ is $N$-contractible in $M\del b$.
But $\{a,x\}$ is a parallel pair in $M\del b/u$, so $M\del b,x/u$, and hence $M \ba b,x$, has an $N$-minor.
As $\{a,u,x,y\}$ is a rank-$3$ cocircuit of $M \ba b$, the matroid $\co(M \ba b,x)$ is $3$-connected by \cref{r3cocircco}.
Thus, if $M\del b,x$ is not $3$-connected, then there is a triad $T^*$ of $M\del b$ that contains $x$.
By orthogonality with the triangle $\{a,u,x\}$, the triad $T^*$ meets $\{a,u\}$.
But $a\notin T^*$ because $M\del a,b$ is $3$-connected.
Thus $T^*$ contains $\{x,u\}$.
But since $M\ba a,b$ is $3$-connected, $T^*$ is also a triad of $M\del a,b$, so $T^*\cup y$ is a $4$-element cosegment of $M \ba a,b$.
Let $T^*-\{x,u\}=\{q\}$.
Now $q \in B^*$, since $q$ is $N$-contractible but not $(N,B)$-robust.
But then $T^* \cup y$ is a confining set, so \cref{gadgetfinish} implies that $|E(M)| \le |E(N)| + 9$; a contradiction.
Thus $M\del b,x$ is $3$-connected with an $N$-minor, and $\{b,x\}$ is contained in a triangle of $M$.

We may now assume that neither $\{a,u,x\}$ nor $\{a,u,y\}$ is a triangle of $M$.
Suppose that (ii)(b)(II) or (ii)(b)(III) of \cref{mainthm1proof} holds.
Consider the matroid $\Delta_T(M)$ obtained by a \dY\ exchange on $T=\{b,x,y\}$.
Observe that $\Delta_T(M)/b\cong M\del b$, where the labels on $x$ and $y$ are swapped.
Thus, if $\Delta_T(M)/b,x \cong M \del b/y$ is $3$-connected with an $N$-minor, then $\{b,x\}$ is a deletion pair of $\nabla_T(M^{*})$ with the desired properties.
Since $y$ is a rim end of a maximal fan in $M \ba a,b$, the matroid $M \ba a,b / y$ is $3$-connected by \cite[Lemma~1.5]{ow2000}.
Moreover, as $M \ba a,b,u$ has an $N$-minor, and $y$ is in a series pair in this matroid, %$M \ba a,b,u/y$, and hence
$M\ba b/y$, has an $N$-minor.
If $M \ba b/y$ is $3$-connected, then $\{b,x\}$ is a deletion pair of $\nabla_T(M^{*})$ as desired.

So we may assume that $M \ba b/y$ is not $3$-connected; then $a$ is in a parallel pair of $M \ba b/y$.
Since $M \ba b$ is $3$-connected, $\{a,y,q'\}$ is a triangle of $M \ba b$ for some $q' \in E(M)-\{a,y,u\}$.  Note also that $q' \neq x$, by \cref{no4line}.
If $q' \in B$, then $\{a,b\}\subseteq \cl_M(\{q',x,y\})$, so (i) holds by \cref{ab3spanned}; a contradiction.
So $q' \in B^*$.
Moreover, $q'$ is $N$-deletable because $x$ is $N$-contractible in $M\del b$ and $q'$ is in a parallel pair of $M\del b/x$.
So $q'$ is $(N,B)$-robust, implying that (ii)(b)(III) holds and $(q',z,x,u,y)$ is a maximal fan in $M\ba a,b$.
We will show that $M \ba a,y$ is $3$-connected with an $N$-minor.
Since $\{x,y,b\}$ is a triangle and $\{x,y,u,b\}$ is a rank-$3$ cocircuit of $M \ba a$, the matroid $\co(M \ba a,y)$ is $3$-connected by \cref{r3cocircco}.
Suppose $y$ is in a triad~$T^*$ of $M \ba a$.
By orthogonality, $T^*$ meets $\{x,b\}$.
But $b \notin T^*$, since $M \ba a,b$ is $3$-connected, so $x \in T^*$.
Now, by orthogonality with the triangle $\{u,x,z\}$, either $T^* = \{y,x,u\}$ or $T^*=\{y,x,z\}$.
Since $\{x,y,u,b\}$ is a cocircuit of $M \ba a$, we deduce $T^* = \{y,x,z\}$.
But then $\{q',z,x,y\}$ is a cosegment of $M \ba a,b$, contradicting orthogonality with the triangle $\{z,x,u\}$.
Hence $M \ba a,y$ is $3$-connected.
Since $M \ba a/x$ has an $N$-minor, and $\{b,y\}$ is a parallel pair in this matroid, $M \ba a,y$ has an $N$-minor.
So $\{a,y\}$ is a deletion pair of $M$ that meets the requirements.

We may now assume that (ii)(b)(I) of \cref{mainthm1proof} holds.
Again, consider the matroid $\Delta_T(M)$, where $T=\{b,x,y\}$.
We claim that either $\Delta_T(M)/b,x$ or $\Delta_T(M)/b,y$ is $3$-connected with an $N$-minor, so either $\{b,x\}$ or $\{b,y\}$ is a deletion pair of $\nabla_T(M^{*})$ with the desired properties.
Observe that $\Delta_T(M)/b\cong M\del b$, so $\Delta_T(M)/b$ is $3$-connected and has an $N$-minor.
Now $\Delta_T(M)/b,x\cong M\del b/y$ and $\Delta_T(M)/b,y\cong M\del b/x$.
Since $u$ is $N$-deletable in $M'$, the elements $x$ and $y$ are $N$-contractible, so $M\del b/x$ and $M\del b/y$ have $N$-minors. Thus $\Delta_T(M)/b,x$ and $\Delta_T(M)/b,y$ have $N$-minors. 

Suppose that $\si(M\del b/x)$ is not $3$-connected.
Then there is a vertical $3$-separation $(P,x,Q)$ of $M\del b$.
Recall that $\{x,y\}$ is the only series pair of $M' \ba u$. %, by \cref{newboy}.
Now, as $\co(M' \ba u) = M \ba b / x \ba a,u$ is $3$-connected, it follows that $Q=\{a,u,q\}$ for some $q \in E(M')-\{u,x\}$, up to swapping $P$ and $Q$.
Since $Q$ is $3$-separating and $r(Q) \ge 3$, the set $Q$ is a triad of $M\del b$.
But then $\{u,q\}$ is a series pair in $M\del a,b$; a contradiction.
Thus $M\del b/x$, and hence $\Delta_T(M)/b,x$, is $3$-connected up to parallel pairs.
The same argument shows that $\Delta_T(M)/b,y$ is $3$-connected up to parallel pairs.

Now, if $\Delta_T(M)/b,x$ or $\Delta_T(M)/b,y$ is $3$-connected, then \cref{no4cct} holds.
Thus we may assume that $x$ and $y$ are in triangles $T_x$ and $T_y$ of $M \ba b$.
If $\{a,x,y\}$ is a triangle, then $|E(M)| \le |E(N)|+7$ by \cref{no4line}; a contradiction.
Suppose that $\{p,x,y\}$ is a triangle of $M'$ for some $p \in E(M')-\{x,y\}$.
Then $p$ is not $(N,B)$-robust.
Since $u$ is $N$-deletable in $M'$, it follows that $x$ is $N$-contractible in $M'$.
Since $\{p,y\}$ is a parallel pair of $M'/x$, the element $p$ is $N$-deletable in $M'$.
Moreover, $p\in B^{*}$, since $\{x,y\}\subseteq B$ and $\{p,x,y\}$ is a triangle of $M'$.
Therefore $p$ is an $(N,B)$-robust element of $M'$; a contradiction.
We deduce that $\{x,y\}$ is not contained in a triangle of $M\ba b$.

By orthogonality, $T_x$ meets $\{a,y,u\}$, and $T_y$ meets $\{a,x,u\}$.
So either $T_x=\{x,a,q\}$ or $T_x=\{x,u,q\}$ for some $q\in E(M')-\{u,x,y\}$.
Now $q$ is $N$-deletable because $x$ is $N$-contractible in $M\del b$ and $q$ is in a parallel pair of $M\del b/x$.
But $q$ is not $(N,B)$-robust, since $q \notin \{u,x,y\}$, so $q \in B$.
If $T_x = \{x,a,q\}$, then, as $q\in B-\{x,y\}$, we have $\{a,b\}\subseteq \cl_M(\{q,x,y\})$, and so (i) holds by \cref{ab3spanned}; a contradiction.
So $T_x = \{x,u,q\}$.
Likewise, arguing with $y$ in the place of $x$,
%$T_y=\{y,a,q'\}$ or $T_y=\{y,u,q'\}$
we deduce that $T_y = \{y,u,q'\}$
for some $q'\in E(M')-\{u,x,y\}$ where $q'$ is $N$-deletable.

Now $T_x = \{x,u,q\}$ and $T_y = \{y,u,q'\}$ for some $N$-deletable elements $q,q'\in E(M')-\{u,x,y\}$.
Moreover, $q \neq q'$, since $\{x,y,u\}$ is not a triangle.
Since $\{q,q',u,x,y\}$ is a rank-$3$ set, and $\{x,y\} \subseteq B$, at most one of $q$ and $q'$ is in $B$.
Without loss of generality, say $q \in B^*$.
Then $q$ is $(N,B)$-robust; a contradiction.
\end{subproof}

Let $M_2$ and $\{a',b'\}$ be as given in \cref{no4cct}.
We again apply \cref{mainthm1proof}, this time on the matroid $M_2$ with minor $N$ and deletion pair $\{a',b'\}$; we may assume that (ii)(b) holds.
We relabel $M_2$ as $M$ and $\{a',b'\}$ as $\{a,b\}$.
Now $M'=M\del a,b$ has an $(N,B)$-strong element $u\in B^{*}$, there is a $5$-element cocircuit $\{x,y,u,a,b\}$ of $M$, and the only $(N,B)$-robust elements of $M'$ are contained in a set $R$ where $\{u,x,y\} \subseteq R$, %$|R| \le 5$,
and $R$ is either a triad, a maximal type-II fan $(z,u,x,y)$ relative to $B$, or a maximal $5$-element fan $(w,z,x,u,y)$. 
Up to switching the labels on $a$ and $b$, we may assume that $\{b,x,y\}$ is a triangle of $M$.
%Note that \cref{ucon} still applies after this relabelling.

Additionally, now $\{a,b\}$ is contained in a triangle of $M$; let $\{a,b,p\}$ be this triangle.
Note that if $p \in \{x,y\}$, then $\{a,b\} \subseteq \cl_M(\{x,y\})$, contradicting \cref{no4line}.  So $p \notin \{x,y\}$. %But we could have p=u$.
%Let $D^+ = D \cup p$.

\begin{claim}
\label{secondpairsetup}
Let $q$ be an $N$-deletable element of $M'$ such that $q \notin \cl^{*}_{M'}(R\cup p)$.
Either
\begin{itemize}
  \item[(I)] $M \ba b,q$ is $3$-connected with an $N$-minor, or
  \item[(II)] there exists $t \in E(M')-q$ such that $t \notin \cl^{*}_{M'}(R \cup p)$, %neither $\{t,x,u\}$ nor $\{t,y,u\}$ is a triangle,
    the matroid $M \ba b,t$ is $3$-connected with an $N$-minor, $t \in B^*$, and, for some $s\in \{x,y\}$, the matroid $M\del b$ has a triangle $T=\{s,t,a\}$ and a $4$-element cocircuit $T \cup q$.
\end{itemize}
\end{claim}
\begin{subproof}
  Note that $q \in B$, since $q \notin R$ and $q$ is $N$-deletable in $M'$.

Suppose that $\co(M\del b,q)$ is $3$-connected, but $M\del b,q$ is not $3$-connected.
Then $M\del b$ has a triad $\{q,s,t\}$.
Hence either $\{q,s,t\}$ or $\{b,q,s,t\}$ is a cocircuit of $M$.
Since $M'$ is $3$-connected, it follows that $a\notin \{q,s,t\}$ and that $\{q,s,t\}$ is a triad of $M'$.
As $q$ is $N$-deletable in $M'$, the elements $s$ and $t$ are $N$-contractible in $M'$.
We claim that $(R-\{x,y\})\cap \{s,t\}=\emptyset$.
To begin with, $u \notin \{s,t\}$ since $\{x,y\}$ is the only series pair of $M'\ba u$.
If $R$ is a maximal type-II fan $(z,u,x,y)$ of $M'$, then $z \notin \{s,t\}$ since the fan is maximal.
Finally, if $R$ is a $5$-element fan $(w,z,x,u,y)$ of $M'$ with $w \in \{s,t\}$, then $w$ is $N$-deletable, implying $q$ is $N$-contractible and hence $N$-flexible in $M'$; a contradiction.
Thus, as $s$ and $t$ are $N$-contractible but not in $R-\{x,y\}$, either $\{s,t\}\subseteq B^{*}-\{a,b,u\}$, or $\{s,t\}$ meets $\{x,y\}$.

Suppose that $\{s,t\}\subseteq B^{*}-\{a,b,u\}$.
If $\{q,s,t\}$ is a triad of $M$, then $A_{qa}=A_{qb}=0$, so there is an allowable pivot on $A_{qs}$ or $A_{qt}$ that gives a basis for $M'$ with more robust elements, contradicting that $B$ is a bolstered basis.
On the other hand, if $\{b,q,s,t\}$ is a cocircuit of $M$, then it intersects the triangle $\{b,x,y\}$ in a single element; a contradiction to orthogonality.

Therefore $\{s,t\}$ meets $\{x,y\}$.
If $\{s,t\}=\{x,y\}$, then %$M'$ has a $4$-element cosegment $\{q,x,y,u\}$, so
$q \in \cocl_{M'}(\{x,y\})$; a contradiction.
So we may assume that $s \in \{x,y\}$ and $t \notin \{x,y\}$.
If $\{q,s,t\}$ is a triad of $M$, then this triad intersects $\{b,x,y\}$ in a single element; a contradiction.
On the other hand, if $\{q,s,t,b\}$ is a cocircuit of $M$, then by orthogonality with the triangle $\{a,b,p\}$, we have $t=p$, in which case $q \in \cocl_{M'}(\{x,y,p\})$; a contradiction.

We may now assume that $\co(M\del b,q)$ is not $3$-connected.
We first show $\co(M\ba a,b,q)$ is $3$-connected.
Suppose not.
Then there is a cyclic $3$-separation $(X,q,Y)$ of $M'$ such that $|X\cap E(N)|\leq 1$ and $Y\cup q$ is coclosed in $M'$.
By the dual of \cref{CPL2}, at most one element of $X$ is not $N$-flexible in $M'$, and if such an element $v$ exists, then $q \in \cocl_{M'}(X-v)$.
But $X-v \subseteq R$, so $q \in \cocl_{M'}(R)$; a contradiction.
So $\co(M\ba a,b,q)$ is $3$-connected.

Since $\co(M \ba b,q)$ is not $3$-connected, there is a cyclic $3$-separation $(P,q,Q)$ of $M\del b$ with $a \in Q$.
Since $\co(M\ba a,b,q)$ is $3$-connected, $(P,Q-a)$ is not a cyclic $2$-separation of $M\ba a,b,q$, so $Q-a$ is a series class of $M\ba a,b,q$.
Hence $(Q-a)\cup q$ is a cosegment of $M'$.
Suppose that $|Q-a|\geq 3$.
Then $Q-a$ meets $B$, and, since $q$ is $N$-deletable in $M'$, the elements of $Q-a$ are $N$-contractible.
By the dual of \cref{longline3conn}, the elements of $(Q-a) \cap B$ are $(N,B)$-strong, so \cref{nostrongbasis} implies that $(Q-a)\cap B\subseteq \{x,y\}$.
Since $q$ is not cospanned by $\{x,y\}$ in $M'$, we have $|(Q-a)\cap B|=1$, and thus $|Q-a|=3$.
But then $\{x,y,u\}\cup (Q-a)$ is a corank-$3$ confining set of $M\del a,b$, contradicting \cref{gadgetfinish}.
Therefore $|Q-a|=2$.
Since $Q$ is a $3$-separating set of $M \ba b$ that contains a circuit, $Q=\{s,t,a\}$ is a triangle. % of $M\del b$.

Since $M'$ is $3$-connected, either $\{q,s,t\}$ or $\{q,s,t,a\}$ is a cocircuit of $M \ba b$.
By orthogonality between the triangle $\{s,t,a\}$ and the cocircuit $\{x,y,u,a\}$ of $M\del b$, we have that $\{x,y,u\}$ meets $\{s,t\}$.
Moreover, $u\notin \{s,t\}$ because the only triad containing $u$ in $M'$ is $\{u,x,y\}$. % by \cref{newboy}.
Thus $\{x,y\}$ meets $\{s,t\}$.
However, $\{x,y\} \neq \{s,t\}$, otherwise $\{x,y\}$ spans $\{a,b\}$, contradicting \cref{no4line}.
%It now follows from orthogonality and \ldots that $\{q,s,t\}$ is not a triad.
%So $\{q,s,t,a\}$ is a cocircuit.
Without loss of generality, let $s \in \{x,y\}$ and $t \notin \{x,y\}$.

Suppose $\{q,s,t\}$ is a triad of $M \ba b$.
If $\{q,s,t\}$ is a triad of $M$, then, by orthogonality with the triangle $\{b,x,y\}$, we have $t \in \{x,y\}$; a contradiction.
On the other hand, if $\{q,s,t,b\}$ is a triad of $M$, then, by orthogonality, the triangle $\{a,b,p\}$ meets $\{s,t\}$.
But then $p=t$, so %$q\in \cl^{*}_M(\{b,s,t\})\subseteq \cl_M^{*}(D^+)$;
$q \in \cocl_{M'}(\{x,y,p\})$; a contradiction.
So $\{q,s,t,a\}$ is a cocircuit of $M \ba b$.

We claim that $t$ satisfies (II).
Recall $s \in \{x,y\}$, and pick $s'$ such that $\{s,s'\} = \{x,y\}$.
Since $s$ is $N$-contractible in $M\del b$ and $\{a,t\}$ is a parallel pair of $M\del b/s$, it follows that $M\del b,t$ has an $N$-minor.
If $t\in \cocl_{M'}(R \cup p)$, then $q \in \cocl_{M'}(R \cup p)$, since $s \in R$.
Thus $t\notin \cocl_{M'}(R \cup p)$.
Now, $t$ is $N$-contractible in $M'$, since $t$ is in a series pair in $M' \ba q$.
Thus, if $t \in B$, then $t$ is $(N,B)$-robust, and $q \in \cocl_{M'}(R)$; a contradiction.
So $t \in B^*$.

It remains to prove that $M\del b,t$ is $3$-connected.
Since $\{q,s,t,a\}$ is a rank\nobreakdash-$3$ cocircuit in $M \ba b$, the matroid $\co(M \ba b,t)$ is $3$-connected by \cref{r3cocircco}.
Suppose $M\del b,t$ is not $3$-connected. Then $t$ is in a triad $T^*$ of $M\del b$.
By orthogonality with the triangle $\{s,t,a\}$, the triad $T^*$ meets $\{s,a\}$.
But $a\notin T^*$ because $M\del a,b$ is $3$-connected.
Thus $\{s,t\} \subseteq T^*$.
If $p \in T^*$, then $t \in \cocl_{M \ba b}(\{s,p\})$, so $t \in \cocl_{M'}(\{x,y,p\})$; a contradiction.  So $p \notin T^*$.
Now $T^*$ or $T^*\cup b$ is a cocircuit of $M$.
But $\{b,x,y\}$ is a triangle of $M$ that meets $T^*$ in a single element, so $T^*$ is not a cocircuit; and $\{a,b,p\}$ is a triangle of $M$ that meets $T^*\cup b$ in a single element, so $T^* \cup b$ is not a cocircuit.
We deduce that $M \ba b,t$ is $3$-connected, and \cref{secondpairsetup} follows.
\end{subproof}

\begin{claim}
  \label{secondpair}
  There are distinct elements $q',q'' \in E(M)$ such that, for $q\in \{q',q''\}$, both of the following hold:
  \begin{itemize}
    \item[(I)] $M\del b,q$ is $3$-connected with an $N$-minor, where $q\notin \cocl_{M'}(R \cup p)$; and %$q$ is not $(N,B)$-robust in $M\ba a,b$; and 
    \item[(II)] either 
      \begin{itemize}
        \item[(a)] $q \in B$, the element $q$ is $N$-deletable in $M\del a,b$, and neither $\{x,u,q\}$ nor $\{y,u,q\}$ is a triangle; or
        \item[(b)] $q \in B^*$ and, for some $s\in \{x,y\}$, the set $T=\{s,q,a\}$ is a triangle that is contained in a $4$-element cocircuit of $M\del b$.
      \end{itemize}
  \end{itemize}
\end{claim}

\begin{subproof}
  Suppose $R$ is either a $4$- or $5$-element fan.
  Then $r^*_{M'}(R) = 3$, the set $\{x,u\}$ is contained in a triangle with an element $z \in R$, and $\{y,u\}$ is not in a triangle, since $R$ is a maximal fan.
  Now $r^*_M(R \cup \{a,b,p\}) \le 6$.
  Since $r^*(M) \ge r^*(N) + 8$, there are distinct $N$-deletable elements $q',q''$ outside of $\cocl_M(R \cup \{a,b,p\})$, neither of which is in a triangle with $\{x,u\}$ or $\{y,u\}$.

  Now suppose $R=\{u,x,y\}$.
  Then $r^*_{M'}(R) = 2$, so $r^*_M(R \cup \{a,b,p\}) \le 5$.
  Since $r^*(M) \ge r^*(N) + 8$, there are at least three $N$-deletable elements outside of $\cocl_M(R \cup \{a,b,p\})$.
  Since these $N$-deletable elements are not $(N,B)$-robust, they belong to $B-\{x,y\}$.
  As $r_{M'}(\{x,y,u\})=3$, and $\{x,y\} \subseteq B$, at most one of these elements is in a triangle with $\{x,u\}$ or $\{y,u\}$.
  Thus there exist distinct elements $q',q''$ outside of $\cocl_M(R \cup \{a,b,p\})$, neither of which is in a triangle with $\{x,u\}$ or $\{y,u\}$.

  Now $q',q'' \notin \cocl_M(R \cup \{a,b,p\}) = \cocl_{M'}(R \cup p)$.
  By \cref{secondpairsetup}, either $q'$ satisfies (I) and (II)(a), or there exists an element $t'$ that satisfies (I) and (II)(b).
  Likewise, either $q''$ satisfies (I) and (II)(a), or there exists an element $t''$ that satisfies (I) and (II)(b).
  Suppose that neither $q'$ nor $q''$ satisfies (II)(a), and $t' =t''$.
  Then $\{s',t',a\}$ and $\{s'',t',a\}$ are triangles of $M \ba b$ where $s',s'' \in \{x,y\}$.
  If $\{s',s''\}=\{x,y\}$, then $\{x,y,t'\}$ is a triangle of $M'$, but then $\co(M' \ba u) \cong M' \ba u/x$ is not $3$-connected; a contradiction.  So $s'=s''$.
  Now $\{q',s',t',a\}$ and $\{q'',s',t',a\}$ are distinct cocircuits of $M \ba b$, so $\{q',q'',s',t'\}$ is a cosegment of $M'$. 
  But $q'$ is $N$-deletable in $M'$, implying $q''$ is $N$-contractible and hence $(N,B)$-robust; a contradiction.
\end{subproof}

Let $q'$ and $q''$ be elements as in \cref{secondpair}.
Suppose that $\{q',q''\} \subseteq B^*$.
Then, by \cref{secondpair}(II)(b), $M \ba b$ has a triangle $T'=\{s',q',a\}$ that is contained in a $4$-element cocircuit $C^*$, and a triangle $T''=\{s'',q'',a\}$, for some $s',s''\in \{x,y\}$.
Observe that $\{x,y\} \nsubseteq C^*$, since $q' \notin \cocl_{M'}(\{x,y\})$ by \cref{secondpair}(I).
Thus $C^*\cup b$ is a cocircuit of $M$, by orthogonality with the triangle $\{b,x,y\}$.
If $s'=s''$, then $\{q',q'',a\}$ is a triangle that intersects the cocircuit $\{x,y,u,a,b\}$ in a single element; a contradiction.
Thus we may assume that $T' = \{x,q',a\}$ and $T'' = \{y,q'',a\}$.
By orthogonality between $C^*\cup b$ and $T''$, we deduce that $q''\in C^*\cup b$, since $y\notin C^*$.
Now $\{x,q',q''\}$ is a triad of $M'$ with $\{q',q''\}\subseteq B^*$, so $\{u,x,y,q',q''\}$ is a corank-$3$ confining set, contradicting \cref{gadgetfinish}.

Without loss of generality, we may now assume that $q' \in B$ and $q'$ is $N$-deletable in $M'$.
%Let $D=\{x,y,u,a,b\}$, and recall that $D$ is a cocircuit of $M$.
Towards a contradiction, assume that (iii) does not hold for $M$ and the deletion pair $\{b,q'\}$.
Then, after applying \cref{mainthm1proof}, (ii)(b) holds.
Let $A'$ be the $B'\times (B')^{*}$ companion $\mathbb{P}$-matrix where $\{x',y',b,q'\}$ incriminates $(M,A')$ for $\{x',y'\}\subseteq B'$ and $\{b,q'\}\subseteq (B')^{*}$.
Then $M$ has a $5$-element cocircuit $D'=\{x',y',u',b,q'\}$, where $M\del b,q'$ has an $(N,B')$-strong element~$u'$ outside of $\{x',y'\}$, and either $\{b,x',y'\}$ or $\{q',x',y'\}$ is a triangle. % of $M$.

Suppose that $\{b,x',y'\}$ is a triangle of $M$.
By orthogonality between the cocircuit $D'$ of $M$ and the triangles $\{b,x,y\}$ and $\{a,b,p\}$, and using the fact that $q' \notin \{x,y,a,p\}$, we deduce that $\{x,y\}$ and $\{a,p\}$ meet $\{x',y',u'\}$.
%Therefore at least one of $\{x,y\}$ and $\{a,p\}$ meets $\{x',y'\}$.
If $\{x,y\}$ or $\{a,p\}$ intersects $\{x',y'\}$ in a single element, then
$\{b,x,y\}$ or $\{a,b,p\}$ is in the span of $\{x',y'\}$, so
$\{x',y'\}$ spans a $4$-element segment in $M$.
Thus $\{x',y'\}$ spans a triangle in $M\del b,q'$.
But then $\co(M\del b,q',u')$ is not $3$-connected by \cref{fanends}, contradicting that $u'$ is $(N,B')$-strong in $M\del b,q'$.
We deduce that $\{x',y',u'\}\subseteq \{x,y,a,p\}$. %\subseteq D^+$.
But $q' \in \cocl_{M'}(\{x',y',u'\}) \subseteq \cocl_{M'}(\{x,y,p\})$,
%Hence $D'-q'\subseteq D^+$, so $q'\in\cl^{*}_M(D^+)$,
contradicting \cref{secondpair}(I).

We may now assume that $\{q',x',y'\}$ is a triangle of $M$.
The triangles $\{b,x,y\}$ and $\{a,b,p\}$ meet the cocircuit $D'$ in the element $b$.
Thus, by orthogonality, $\{x,y\}$ and $\{a,p\}$ meet $\{x',y',u'\}$. 
Let $D=\{x,y,u,a,b\}$, and recall that $D$ is a cocircuit of $M$.

First suppose that $D\cap \{x',y'\}=\emptyset$.
Then $u'\in \{x,y\}$ and $p \notin D$, so $p\in\{x',y'\}$.
Since $q'$ is $N$-deletable in $M'$, the element $a$ is $N$-deletable in $M\del b,q'$.
If $a \in B'$, then $\{a,b,p\}$ is a triangle of $M$ with $a,p \in B'$, so $A'_{vb}=0$ for $v\in \{x',y'\}-p$, contradicting that the bad submatrix $A'[\{x',y',b,q'\}]$ has no zero entries.
So $a \in (B')^*$, hence $a$ is $(N,B')$-robust in $M \ba b,q'$.
It follows that $M \ba b,q'$ has a $5$-element fan $(a,z',x',u',y')$ for some $z'$, where the $(N,B')$-robust elements of $M \ba b,q'$ are contained in $\{x',y',u',z',a\}$.
Since $M \ba b,q',a$ has an $N$-minor, and $x'$ is in a series pair in this matroid, the element $x'$ is $N$-contractible in $M \ba a,b$.
Moreover, $\{u',z'\}$ and $\{q',y'\}$ are parallel pairs in $M \ba a,b /x'$, so $M \ba a,b,u',y'$ has an $N$-minor.
But $\{x',q'\}$ is a series pair in this matroid, so $q'$ is also $N$-contractible in $M\ba a,b$.
Now $q'$ is $(N,B)$-robust; a contradiction.

Now we may assume that $D\cap \{x',y'\}\neq \emptyset$.
By orthogonality with the triangle $\{q',x',y'\}$, we have $\{x',y'\}\subseteq D$.
If $u' \in D$, then $q' \in \cocl_M(D) \subseteq \cocl_{M'}(R)$; a contradiction.
By orthogonality between the cocircuit $D'$ and triangles $\{b,x,y\}$ and $\{a,b,p\}$, one of $\{x',y'\}$ is in $\{x,y\}$ and the other in $\{a,p\}\cap D$.
By \cref{secondpair}(II)(a), neither $\{x,u,q'\}$ nor $\{y,u,q'\}$ is a triangle, so $\{s,a,q'\}$ is a triangle for some $s \in \{x,y\}$.
But $\{s,q'\} \subseteq B$, so either $A_{xa}=0$ or $A_{ya}=0$, contradicting that the bad submatrix has no zero entries. 
We deduce that (iii) holds for $M$ and the pair $\{b,q'\}$.
\end{proof}

\section*{Acknowledgements}

We thank the anonymous referees of an earlier version of this manuscript, whose valuable suggestions lead to vast improvements in this version.

\bibliographystyle{acm}
\bibliography{fr}
\end{document}